\documentclass[11pt]{notes}
\usepackage{amssymb,amsthm,amsmath,amstext,amscd}
\usepackage{eucal}
\usepackage{relsize}
\usepackage{marvosym}
\usepackage[greek,english]{babel}
\usepackage{epsfig}
\usepackage{mathrsfs}
\usepackage{bm}
\usepackage[all]{xy}
\usepackage{comment}
\usepackage{parskip}
\usepackage{marginnote}

\usepackage{graphicx}
\usepackage{float}
\usepackage{tikz,tikz-cd}
\usepackage{dynkin-diagrams}
\usetikzlibrary{bending} 
\usepackage{utfsym}
\usepackage{fourier}

\newtheorem{exercise}{{\color{blue}Exercise}}
\newtheorem*{sol}{{\color{red}Solution}}

\renewcommand{\AA}{\mathbb{A}}
\newcommand{\BB}{\mathbb{B}}
\newcommand{\CC}{\mathbb{C}}
\newcommand{\GG}{\mathbb{G}}

\newcommand{\RR}{\mathbb{R}}
\newcommand{\OO}{\mathbb{O}}
\newcommand{\HH}{\mathbb{H}}
\newcommand{\PP}{\mathbb{P}}
\newcommand{\QQ}{\mathbb{Q}}
\newcommand{\ZZ}{\mathbb{Z}}
\newcommand{\KK}{\mathbb{K}}
\newcommand{\SSS}{\mathbb{S}}

\newcommand{\cO}{\mathcal{O}}

\newcommand{\cH}{\mathcal{H}}

\newcommand{\cL}{\mathcal{L}}

\newcommand{\cQ}{\mathcal{Q}}

\newcommand{\cN}{\mathcal{N}}

\newcommand{\cS}{\mathcal{S}}
\newcommand{\cU}{\mathcal{U}}

\newcommand{\cV}{\mathcal{V}}

\newcommand{\ff}{\mathfrak f}
\newcommand{\fs}{\mathfrak s}

\newcommand{\fe}{\mathfrak e}
\newcommand{\fn}{\mathfrak n}
\newcommand{\fg}{\mathfrak g}
\newcommand{\fh}{\mathfrak h}

\newcommand{\fso}{\mathfrak{so}}

\newcommand{\fsp}{\mathfrak{sp}}
\newcommand{\fsl}{\mathfrak{sl}}
\newcommand{\fgl}{\mathfrak{gl}}
\renewcommand{\ff}{\mathfrak{f}}

\newcommand{\ii}{\mathrm{i}} 

\newcommand{\raar}{\rightarrow}
\newcommand{\lra}{\longrightarrow}
\newcommand{\extp}{\textstyle\bigwedge} 
\newcommand{\intprod}{\mathbin{\lrcorner}} 

\DeclareMathOperator{\Stab}{Stab}
\DeclareMathOperator{\Tot}{Tot}

\DeclareMathOperator{\Aut}{Aut}
\DeclareMathOperator{\Tri}{Tri}
\DeclareMathOperator{\End}{End}

\DeclareMathOperator{\rank}{rank}
\DeclareMathOperator{\rk}{rk}

\DeclareMathOperator{\Fl}{Fl}

\DeclareMathOperator{\Spec}{Spec}
\DeclareMathOperator{\Der}{Der}

\DeclareMathOperator{\tr}{tr}

\DeclareMathOperator{\Sing}{Sing}
\DeclareMathOperator{\GL}{GL}
\DeclareMathOperator{\SL}{SL}
\DeclareMathOperator{\SU}{SU}
\DeclareMathOperator{\U}{U}
\DeclareMathOperator{\SO}{SO}
\DeclareMathOperator{\rmO}{O}
\DeclareMathOperator{\Sp}{Sp}

\DeclareMathOperator{\Pf}{Pf}

\DeclareMathOperator{\Spin}{Spin}
\DeclareMathOperator{\id}{id}
\DeclareMathOperator{\Hil}{F}
\DeclareMathOperator{\ann}{Ann}

\newcommand{\sm}{{\rm{sm}}}
\renewcommand{\Re}{\operatorname{Re}}
\renewcommand{\Im}{\operatorname{Im}}

\newcommand{\cPP}{\check{\PP}}
\DeclareMathOperator{\ad}{ad}
\newcommand{\con}{\mathcal{N}}
\DeclareMathOperator{\codim}{codim}
\DeclareMathOperator{\im}{im}
\newcommand{\spann}[1]{\mathrm{span}\left\{#1\right\}} 
\newcommand{\der}{ \mathfrak{der}}
\newcommand{\ftri}{\mathfrak{tri}}

\DeclareMathOperator{\Gr}{Gr} 
\DeclareMathOperator{\OGr}{OGr} 
\DeclareMathOperator{\CGr}{CGr} 
\newcommand{\GtwoGr}{G_2\mathrm{Gr}} 
\DeclareMathOperator{\DGr}{DGr} 
\DeclareMathOperator{\LieGr}{LieGr} 
\DeclareMathOperator{\IGr}{IGr} 
\DeclareMathOperator{\HG}{HGr} 

\newcommand{\computation}{{\color{blue}\Large{\usym{2659}}}\quad}
\newcommand{\training}{{\color{blue}\Large{\WritingHand}}\quad }
\newcommand{\necessary}{{\color{blue}\Large{\noway\PointingHand}}\quad }
\newcommand{\curiosity}{{\color{blue}\Large{\usym{2658}}}\quad }

\newcommand{\gaia}[1]{{\color{teal}\textbf{gaia:{#1}}}}

\usepackage{hyperref}
\usepackage{pdfpages}
\usepackage{minitoc}

\title{BRIDGES Lectures: \\ \, \\ $G_2$ in action, \\ 
and a mathematical theory of exceptions}

\author{Laurent Manivel \\ Toulouse Mathematics Institute \\ CNRS/Toulouse University \\ \texttt{manivel@math.cnrs.fr}}

\date{May 2025}

\begin{document}


\maketitle

\begin{abstract}
The \emph{BRIDGES meeting in gauge theory, extremal structures, and stability} was held in June 2024 at l’Institut d'\'Etudes Scientifiques de Carg\`ese in Corsica, organized by Daniele Faenzi, Eveline Legendre, Eric Loubeau, and Henrique S\'a Earp. The first week was a summer school consisting of four independent but related lecture series by Oscar Garc\'ia-Prada, Spiro Karigiannis, Laurent Manivel, and Ruxandra Moraru. The present document consists of notes for the lecture series by Laurent Manivel on exceptional complex Lie groups, especially $G_2$, and some algebraic varieties acted on by these groups. 

The first lecture is a rather elementary introduction to $G_2$ as an algebraic group, with its various incarnations in terms of trivectors or the Cayley octonions. The second lecture focuses on the geometry of several varieties acted on by $G_2$, transitively or quasi-transitively. The third lecture explores the Tits-Freudenthal magic square of Lie algebras, leading to a uniform construction of the exceptional simple Lie algebras, starting from Jordan algebras and triality. The fourth and last lecture discusses the more geometric aspects of the magic square, and their connections with a few famous problems in complex geometry.    

Some assistance in the preparation of these notes by the author was provided by several participants of the summer school. Details can be found at the end of the introduction.
\end{abstract}

\dominitoc
\tableofcontents

\clearpage
\renewcommand{\thepage}{\arabic{page}}
\setcounter{page}{1}

\chapter{Introduction}

In the introduction to the collection of short writings that Philippe Sollers published in 1985 under the title {\it Théorie des exceptions} (Theory of exceptions), he recalls that one of the original meanings of the 
greek word \textgreek{θεωρία} is {\it "une ambassade, une procession, une fête - un défilé, ou plutôt une danse d'exceptions"} ("an embassy, a procession, a celebration - a parade, or rather a dance of exceptions"). Let it be clear, once and for all, that there is no theory of exceptions, in the abstract sense of the word. What we hope to offer the reader is rather a colorful parade of exceptional mathematical objects, all deriving from the exceptional complex simple Lie algebras and groups. 

The discovery of the exceptional Lie algebras was a completely unexpected 
consequence of Killing's progress on the classification of simple Lie algebras. 
The official birthdate of $\fg_2$ is 23 May 1887, when Killing, in a letter to 
Engel, wrote the root system and  multiplication table \cite{hawkins}. The full classification was announced in October of the same year, and published in Mathematische Annalen as soon as 1888. Killing's paper actually mentioned two simple Lie algebras of dimension $52$, hence two avatars of $\ff_4$, and only after Elie Cartan's thesis in 1894
did the Killing-Cartan classification reached its final form. 

Of course this was only the beginning of a long heroic story, in which the 
role of the exceptional Lie algebras has been somehow ambiguous. On the one hand, 
they seem to disturb the beautiful harmony of the classical Lie algebras, that 
Killing expected to be the only simple ones. On the other hand, they add a little spice to this overly wise story, that has spread in different directions and gave birth to all sorts of exciting exceptional phenomena. 

The goal of these lectures is to present  some aspects of this story concerning complex projective geometry. Our main characters will be projective varieties, that are acted on by some exceptional Lie groups : we will thus see $G_2$ in action, as well as the other Lie groups of type $F_4, E_6$, $E_7$ and $E_8$. One of the main occurrences of algebraic groups is indeed as symmetry groups of varieties, and there are at least two reasons to focus on those varieties whose automorphism group is exceptional.
The first one is the particular, intriguing beauty of these objects, that are worth contemplating for themselves -- another meaning of the word {\it theory}, more from the Renaissance period, is "the science that deals with contemplation". Baudelaire wrote {\it Le beau est toujours bizarre. Tâchez de concevoir un beau banal !"} ("Beauty is always bizarre. Try to design a {\it banal beauty}!"). The second one is that exceptional objects often appear in classifications, in which they are likely to 
demonstrate some extremal, unexpected  behavior. 

But we will also take the opportunity to mention, at least in passing, several general notions and problems that exceptional objects may nicely illustrate. To name a few: symmetric spaces, affine homogeneous spaces and their compactifications, prehomogeneous spaces, 
birational maps and homolo\"idal polynomials, varieties of minimal rational 
tangents, dual varieties, Lagrangian and Legendrian varieties, contact structures, graded Lie algebras, Tits geometries, etc. Our hope is that the reader of these notes, after enjoying the 
{\it dance of exceptions}, will feel like delving deeper into the feast.

\bigskip
This text is divided into four chapters, one for each lecture, and three appendices to make the text more self-contained. The first two lectures cover general theory and should be mostly accessible to graduate students and young researchers, both from algebraic and differential geometry (starting from Section \ref{sec_CG} the arguments become more algebraic in flavour). The last two lectures are more technical and will appeal more to algebraic geometers. Basics from representation theory of semisimple Lie algebras, which are briefly recalled throughout the text, may be very useful to understand at least some of the content of the lectures. Each lecture ends with an exercise section, followed by solutions or hints;  but the reader is strongly recommended to try to solve the exercises by himself. They are advertized according to their difficulty:
\begin{exercise}
\computation Computation (easy exercise, verification)
\end{exercise}
\begin{exercise}
\training Training exercise (proper exercise)
\end{exercise}
\begin{exercise}
\necessary Necessary exercise (exercise whose solution is necessary to understand the notes)
\end{exercise}
\begin{exercise}
\curiosity Curiosity (not directly useful for reading the notes)
\end{exercise}

\noindent
The plan is the following. We start the first chapter by studying three-forms in seven variables and their connections with Cayley's algebra of octonions. It is a result of Engel that the exceptional complex Lie group $G_2$ can be realized as the stabilizer of a generic three-form. Later, E. Cartan proved that $G_2$ can also be realized as the automorphism group of the octonions. In this lecture, we classify three-forms reproving a classical theorem of Schouten on the finiteness of $\GL_7$-orbits of three-forms, thus characterizing the generic forms. Then we construct composition algebras, in particular Cayley's octonions. Finally, we describe $G_2$ and its representations. Contributed to this chapter: A. Muniz (classification of three-forms), P. Schwahn (composition algebras, representations of G2, and actions of compact Lie groups).

\noindent
The second chapter  is were we really see $G_2$ in action, since it 
focuses on several interesting algebraic varieties acted upon by $G_2$. We highlight the roles played by the octonions and the representations of $G_2$ in their constructions. First we discuss the two five-dimensional $G_2$-homogeneous varieties: the quadric $\QQ^5 = G_2/P_1$ and the $G_2$-Grassmannian $G_2\Gr(2,7) = G_2/P_2$. Then we construct the Cayley Grassmannian $\CGr\subset \OGr(3,7)$, which parametri\-zes quaternion subalgebras of the octonions;  its 
"doubled version" $\DGr\subset \OGr(6,14)$, which compactifies the affine algebraic group $G_2$;  the Lie Grassmannian $\LieGr(2,7)$, that parametrizes Lie subalgebras of the imaginary octonions; the horospherical $G_2$-variety $\HG\subset \PP(\Im(\OO^{\CC})\oplus \mathfrak{g_2}^{\CC})$, an unexpected deformation of the orthogonal Grassmannian $\OGr(2,7)$. We give these varieties several interpretations in terms of octonions. Contributed to this chapter: G. Comaschi (homogeneous $G_2$-varieties, the Cayley Grassmannian), V. Benedetti (double Cayley, Lie and horospherical Grassmannian).

\noindent
The third chapter presents two constructions of exceptional Lie groups via Jordan algebras and Triality. Our main Jordan algebras are spaces of 
$3\times 3$ Hermitian matrices with coefficients in a normed algebra $\AA$, and we show how to construct from these "projective planes" over $\AA$. In particular 
the octonions yield the sixteen-dimensional Cayley plane, in which "projective 
lines" are eight-dimensional quadrics. Starting from a pair of normed algebra, we explain how to define a Lie algebra and construct in particular  the exceptional series of Lie algebras in a unified way. The result of this construction is the  celebrated magic square of Freudenthal and Tits. Contributed to the exercises: V. Benedetti, G. Comaschi, A. Muniz, P. Schwahn.

\noindent
The last chapter explores the varieties that appear in the 
geometric version of the magic square, mainly due to Freudenthal.  We will explain how varieties from a same row share very similar geometric properties, and  how varieties from a same column are related to each other. It turns out that all these varieties, although not all exceptional, exhibit some special behaviour that we discuss and  connect to a series of  open problems of general interest in complex projective geometry. Contributed to the exercises: V. Benedetti, G. Comaschi, A. Muniz, P. Schwahn.

\noindent
{\bf Appendices.}
Appendix \ref{app:proj-duality} covers basic facts of projective duality, drawn mostly from the books \cite{GKZ} and \cite{Tevelev-PD}. Appendix \ref{sec_gradings} covers graded Lie algebras, following \cite{Vinberg_graded}; this part is not strictly necessary but helps understanding many of the constructions in a uniform way. Appendix \ref{sec_tits} covers Tits' shadows, i.e. how to cover homogeneous spaces with homogeneous families of subvarieties; this is part of the theory of Tits geometries, which is very helpful when trying to understand the geometry of homogeneous spaces. Contributed to the Appendices: A. Muniz (projective duality), V. Benedetti (gradings of Lie algebras, Tits' shadows).

\noindent
{\bf Acknowledgements.}
L. Manivel gratefully acknowledges the essential contributions of his team's members V. Benedetti, G. Comaschi, A. Muniz, P. Schwahn. Thanks are also due to E. Legendre, D. Faenzi, E. Loubeau, H. S\'a Earp  for the invitation to the BRIDGES Summer School, for the friendly atmosphere in Cargese, and for the opportunity to prepare and expand these notes. 

The BRIDGES meeting was supported by the \emph{BRIDGES collaboration: Brazil-France interactions in gauge theory, extremal structures and stability}, grants $\#$2021/04065-6, São Paulo Research Foundation (FAPESP) and ANR-21-CE40-0017, Agence Nationale de la Recherche (ANR).

\chapter{Three-forms and the octonions}

{In 1887, W. Killing was attempting to show that the only simple complex Lie algebras were $\fsl(n,\CC)$, $\fso(n,\CC)$ and $\fsp(n,\CC)$ when he discovered a $14$-dimensional simple algebra of rank two, which we now call $\fg_2$: the first exceptional Lie algebra. Later, F. Engel realized a Lie group of type $G_2$ as the stabilizer of a generic exterior three-form in seven variables. Another realization came from E. Cartan, who showed that the automorphism group of Cayley's octonions is of type $G_2$. This first lecture aims to explore the classification of three-forms, the structure of composition algebras and the octonions, and show how a group of type $G_2$ can be realized. For more on the historical background, see \cite{agricolag2, hawkins} and references therein. 
}


To set up notation: $V_n$ denotes a complex vector space of complex dimension $\dim_{\CC}V_n = n$; its dual is $V_n^*$; a skew-symmetric (or exterior) $k$-form on $V_n$ is an element of $\extp^k V_n^*$; for $k=n$, we denote by $\det(V_n^*) = \extp^n V_n^* \simeq \CC$ the top exterior power. Furthermore, if we fix a basis $\{e_1, \dots, e_n\}$ for $V_n$, then $\{e^1, \dots, e^n\}$ is the dual basis for $V_n^*$;  the vectors $e_{i_1\cdots i_k} = e_{i_1} \wedge \dots \wedge e_{i_k}$ with $i_1<\cdots <i_k$ form a basis for $\extp^kV_n$, and similarly for $e^{i_1\cdots i_k}$. For any $k$-form $\omega\in\extp^kV_n^\ast$, the interior product with a vector $v\in V_n$ is the $(k-1)$-form $v\intprod\omega=\omega(v,\ldots\,)$. The \emph{support} of a $k$-form $\omega \in \extp^kV_n^*$ is the smallest subspace $W\subset V_n^*$ such that $\omega \in \extp^kW$, we will denote it by $S_\omega$; the dimension of $S_\omega$ is the \emph{rank} of $\omega$. When $k=2$, $\extp^2V_n^*$ is isomorphic to the space of skew-symmetric matrices and this notion of rank coincides with the usual rank for matrices. 

A desirable property of this rank is that it defines a good stratification of $\extp^k V_n^*$. Let $\Sigma^k_{n,r}$ denote the set of $k$-forms of rank $r$ on $V_n$. If non-empty, it is a smooth algebraic variety of (complex) dimension $\binom{r}{k} + r(n-r)$ whose closure is the union of forms of rank $r$ or lower 
\cite[Chap. I]{Mar-sings}.  Moreover, the natural action of $\GL_n = \GL_n(\CC)$ on $\extp^kV_n^*$ preserves the rank, hence $\Sigma^k_{n,r}$ decomposes into orbits. Each $\GL_r$-orbit in $\Sigma^k_{r,r}$ corresponds to a $\GL_n$-orbit in $\Sigma^k_{n,r}$ of the same codimension. Therefore, before describing three-forms in seven variables, we start with lower dimensions.

\section{Three-forms in few variables}\label{sec:fewvars}
Three-forms in three variables are volume forms spanning the one-dimensional space $\det(V_3^*)$. In four variables, we have an isomorphism given by contraction:
$$\extp^3V_4^* \simeq V_4 \otimes \det(V_4^*) \simeq V_4, \qquad  v\otimes e^{1234} \mapsto v\intprod e^{1234}.$$ Note that $v\intprod e^{1234} \in \extp^3(\ann(v))$, where $\ann(v) = \{\, f\in V_4^* \mid f(v) = 0 \,\}$; hence, every non-zero element of $\extp^3V_4^*$ is decomposable, i.e., a product of one-forms; thus of rank three. 

In five variables, three-forms have either rank three or five; from the above discussion, we see that rank four is not possible. Rank-three forms are decomposable, so let $\omega \in\extp^3V_5^*$ be of rank five. Consider the skew-symmetric bilinear pairing 
\[
V_5^* \otimes V_5^* \longrightarrow \det(V_5^*); \quad (\alpha, \beta) \longmapsto \alpha \wedge \beta \wedge \omega. 
\]
Since the dimension of $V_5$ is odd, this pairing must degenerate, i.e., there exists $\alpha \in V_5^*$ such that $\alpha \wedge \omega = 0$; consequently, $\omega = \alpha \wedge \theta$, with $\theta$ a two-form of rank four. Acting with $\GL_5$, we can assume that $\alpha = e^1$ and $\theta \in \extp^2\spann{e^2, \dots, e^5}$. From the classification of two-forms (or skew-symmetric matrices), we conclude that every rank-five three-form is in the orbit of
\[
\omega = e^1\wedge(e^{23} + e^{45}) = e^{123} +e^{145}.
\]
In six variables, three-forms of maximal rank are described by the following result.

\begin{prop}\label{prop:orbits-66}
The set $\Sigma_{6,6}^3$ decomposes into two $\GL_6$-orbits, represented by
    $$\omega_1 = e^{123}+e^{456} \quad  and \quad 
    \omega_2 = e^{124}+e^{135}+e^{236}.$$
Moreover, $\GL_6\cdot \omega_1$ is an open subset of $\extp^3V_6^*$, and the closure of $\GL_6\cdot \omega_2$ is an open subset in a irreducible quartic hypersurface. 
\end{prop}

We will give a partly geometric proof. We refer the reader to \cite{R-rg6} for a more elementary algebraic proof. Let us begin with a useful lemma.

\begin{lemma} \label{lem:comp-irred}
Let $V$ be a complex vector space with an action of $\GL_n$ such that there is a Zariski-open orbit. Then any hypersurface in the complement of this orbit is irreducible.
\end{lemma}

\begin{proof}
Let $\cO$ be the open orbit and $D \subset V\setminus \cO$ be a hypersurface in the complement. If $D = D_1\cup D_2$ with $D_1\neq D_2$ then $D_1 = \{f_1 = 0\}$ and $D_2 = \{f_2=0\}$ for two polynomials $f_1,f_2$ which must be \emph{semi-invariant} (i.e., invariant modulo some character) under the action of $\GL_n$:
\[
f_i(g\cdot x) = \det(g)^{t_i}f_i(x)
\]
for some $t_i\in \ZZ$, $i=1,2$. Let $x_0 \in \cO$ be a point in the open orbit, and consider the polynomial function $F(x) = f_1(x_0)^{t_2}f_2(x)^{t_1} -f_2(x_0)^{t_1}f_1(x)^{t_2} $. Then, $F$ is semi-invariant and $F(x_0) = 0$. Since the orbit of $x_0$ is an open subset, $F$ is identically zero, hence $f_1^{t_1}$ and $f_2^{t_2}$ are proportional, which contradicts $D_1\neq D_2$. 
\end{proof}

\begin{proof}
First, we observe that the stabilizer of $\omega_1 = e^{123} + e^{456}$ under the action of $\GL_6$ must preserve the pair $\{P_1,P_2\}$ of subspaces defined by $P_1=\spann{e^1,e^2,e^3}$ and $P_2=\spann{e^4,e^5,e^6}\subset V_6^*$. Indeed, the set of vectors $v$ whose contraction with $\omega$ is a two-form of rank at most two is the union of $P_1^\perp$ with $P_2^\perp$. 
We readily deduce that this stabilizer must be isomorphic to $(\SL_3\times \SL_3)\rtimes \ZZ_2$,  of dimension $16$. Then $\dim (\GL_6\cdot \omega_1) = 36-16 = 20 = \dim \extp^3V_6^*$, hence
the orbit of $\omega_1$ must be open. 

We proceed to describe  the complement.
Consider $X \subset \extp^3V_6$ the set of decomposable three-vectors. It is an affine variety, the affine cone over the Grassmannian $\Gr(3, V_6)$ in its Pl\"ucker embedding. The dual variety $X^\vee \subset \extp^3V_6^*$ is the Zariski closure of the set of forms $\omega \in \extp^3V_6^*$ such that $\ker \omega $ contains the embedded tangent space $T_wX$ for some $w\in X\setminus \{0\}$. See Appendix~\ref{app:proj-duality} for basics on Projective Duality. 

Let $W\subset V_6$ be the subspace associated with $w\in X$. The tangent space to $T_wX$ is spanned by three-vectors of the form $v_1\wedge v_2 \wedge v_3$ where $v_1, v_2 \in W$, as can be checked by computing first-order deformations. Therefore, $\omega \in X^\vee$ if and only if there exists $W\in V_6$ of dimension three such that 
\[
\omega \in \extp^2\ann(W) \wedge V_6^*=\extp^3\ann(W)\oplus\extp^2\ann(W) \otimes W',
\]
for any complement $W'$ of $W$. Consider the component in $\extp^2\ann(W) \otimes W'\simeq  \CC^3 \otimes \CC^3$.
Let $M_\omega$ be the associated $3\times 3$ matrix. Up to the action of $\GL_3\times\GL_3$ it is only determined by its rank.  But $\rk M_\omega = 3$, since otherwise $\omega$ would have rank at most $5$.  So we can choose basis  $\{e^1,e^2,e^3\}$ of $\ann(W)$ and $\{e^4, e^5, e^6\}$ of $W'$ for which  $M_\omega$ is the identity, and then
\[
\omega = ze^{123}+e^{12}\wedge e^4 + e^{13}\wedge e^5 + e^{23}\wedge e^6 = e^{12}\wedge (ze^3+e^4)+e^{135}+e^{236} = \omega_2.
\]

Now, given $\omega \in \extp^3V_6^*$, consider the equivariant map $\Psi_{\omega} \colon V_6 \to V_6 \otimes \det(V_6^*)$ given by $\Psi_{\omega}(v) = u \otimes e^{123456}$, where $u$ is the unique vector such that $(v\intprod \omega)\wedge\omega = u\intprod e^{123456}$. Then define
\[
\lambda\colon \extp^3V_6^* \lra \det(V_6^*)^{\otimes 2}; \quad \omega \longmapsto  \frac{1}{6}\tr(\Psi_{\omega} \circ \Psi_{\omega}).
\]
Note that $\Psi_\omega$ is quadratic in the coefficients of $\omega$, hence $V(\lambda)$, the zero locus of $\lambda$, is a quartic hypersurface. A simple computation shows that $\lambda(\omega_1) = 1$ and $\lambda(\omega_2) = 0$; hence $V(\lambda)$ is contained in the complement of $\GL_6\cdot \omega_1$, and contains the closure of  $\GL_6\cdot \omega_2$. By Lemma~\ref{lem:comp-irred}, $V(\lambda)$ is irreducible. Since forms of lower rank are contained in $\Sigma^3_{6,5}$, which has dimension $15 < 19 = \dim V(\lambda)$, this hypersurface must be the closure of $\GL_6\cdot \omega_2$.
\end{proof}

\begin{remark}\label{rem:oadp}
As we noticed, the decomposition of $\omega_1$ in Proposition \ref{prop:orbits-66} is essentially unique. Geometrically, this can be interpreted as the property, for a general point in $\PP(\wedge^3V_6^*)$, to be on a unique secant line joining two points of the Grassmannian; this is called the {\it one apparent double point property}. Note that $\wedge^3V_6$ is self-dual (choose a volume form and use wedge product); this allows to identify the dual variety of the Grassmannian with its tangent variety (the union of its tangent spaces). The decomposition of $\omega_2$ in Proposition \ref{prop:orbits-66} is also  essentially unique (by the same argument as for $\omega_1$), and this can be rephrased by saying that a generic point on the tangent variety belongs to a unique tangent line. Note finally that the Grassmannian is cut-out by the  quadratic equations $\Psi_{\omega}=0$, parametrized by $SL(V_6)$. In particular the tangent quartic is singular along the Grassmannian, as expected in general. 
\end{remark}

\begin{remark}\label{rem:v6-reals}
The proof of Proposition \ref{prop:orbits-66} uses the complex numbers, or at least that the base field is quadratically closed.
Over $\RR$ there is  one more orbit with representative $e^{124}+e^{135}+e^{456}-e^{236}$, cf. \cite{R-rg6}.
\end{remark}



\section{Nondegenerate three-forms in seven variables}
\label{sec:nondeg7}

For a three-form $\omega$ on a seven-dimensional vector space $V_7$, define the quadratic form
\begin{equation}\label{eq:quadV7}
 q_\omega(v)=(v\intprod\omega)\wedge(v\intprod\omega)\wedge \omega\in\det(V_7^*).
\end{equation}
The three-form $\omega$ is called \emph{nondegenerate} if $q$ has maximal rank. We will prove that the set of nondegenerate three-forms is the unique open orbit for the action of $\GL_7$ on $\extp^3V_7^*$, which justifies calling such forms \emph{generic}. This will be a consequence of the following theorem due to Schouten, see \cite{schouten} or \cite{noui-revoy} for a modern exposition. 

\begin{theorem}[Schouten] \label{thm:schouten}
Let $\omega \in \Sigma^3_{7,7}$ be a three-form of maximal rank  
 on a seven-dimensional complex vector space. Up to the action of $\GL_7$, $\omega$ is one of the following:
\begin{enumerate}
    \item $\omega_1= e^{125}+e^{136}+e^{147}$;
    \item $\omega_2 =e^{125}+e^{136}+e^{147} + e^{234}$;
    \item $\omega_3 = e^{125}+e^{236}+e^{347}$;
    \item $\omega_4 = e^{125}+e^{147}+e^{346}+e^{327}$;
    \item $\omega_5 = e^{125}+e^{136}+e^{147} + e^{234} + e^{567}$.
\end{enumerate}
\end{theorem}


We will give here a geometric proof. For an elementary algebraic proof, see \cite{noui-revoy}. 

The finiteness of the number of orbits follows from general theory of graded Lie algebras, see Example~\ref{example:e7e8}. Let us only assume, at this point, that there exists an open orbit; we claim that is must coincide with the set of nondegenerate forms. Indeed, the set of degenerate forms is an invariant hypersurface $D$, and we know by Lemma \ref{lem:comp-irred} that there is a unique such hypersurface: the complement of the open orbit!   
One can check that the ranks of the quadratic forms $q_{\omega_i}$ are respectively $1,1,2,4,$ and $7$ (\necessary Exercise \ref{ex:rank-forms}). Thus, the open orbit is that of $\omega_5$.

To prove Schouten's theorem, we follow the same line of arguments as in the proof of Proposition~\ref{prop:orbits-66}. We need the following lemma.

\begin{lemma}\label{lem:auxV7}
Let $H\subset \extp^2V_4^*$ be a three-dimensional subspace such that the supports of its elements span $V_4^*$. Then $H$ has a basis of one of the following forms, 
for some basis $\{e^1, e^2, e^3, e^4\}$ of $V_4^*$:
\[
\begin{cases}
\gamma_1 = e^{12} \\
\gamma_2 = e^{13}\\
\gamma_3 = e^{14}\\
\end{cases}\!\!\!\!\!, \quad
\begin{cases}
\gamma_1 = e^{12} \\
\gamma_2 = e^{13}\\
\gamma_3 = e^{14}+ e^{23}\\
\end{cases}\!\!\!\!\!, \quad
\begin{cases}
\gamma_1 = e^{12}\\
\gamma_2 = e^{23}\\
\gamma_3 = e^{34}\\
\end{cases}\!\!\!\!\!, \quad \text{or} \quad
\begin{cases}
\gamma_1 = e^{12}\\
\gamma_2 = e^{34}\\
\gamma_3 = (e^1+e^3)\wedge(e^2+e^4).\\
\end{cases}\!\!\!\!\!
\]
\end{lemma}

\begin{proof}
Consider the quadratic form $q(\eta) = \eta\wedge \eta \in \extp^4V_4^* \simeq \CC$. The proof is a case-by-case analysis in terms of the rank of the restriction $q|_H$, see \cite[\S3.2]{R-rg6}. 
\end{proof}

\begin{proof}[Proof of Theorem \ref{thm:schouten}]
From the discussion below, we just need to determine the orbits of degenerate forms.
As in the proof of Proposition \ref{prop:orbits-66}, consider $X$ the affine cone over 
$\Gr(3,V_7)$, and its dual $X^\vee\subset \extp^3V_7^*$. This is a hypersurface (see Example \ref{ex:deg dual G(3,7)}), $GL(V_7)$-invariant, so it has to coincide with $D$.

So let $\omega \in D=X^\vee$; there exists $W\in \Gr(3,V_7)$ such that 
\[
\omega \in \extp^2\ann(W)\wedge V_7^*=\extp^3\ann(W) \oplus \extp^2\ann(W)\otimes W'
\]
if $W'$ is any complement. 
(Notice that, given any $u\in W$, $(u\intprod \omega)\wedge \omega \in \extp^4\ann(W)\otimes W'$ hence $q_{\omega}(u,v) = 0$ for every $v\in V_7$, confirming that $\omega$ is degenerate.)

Consider the decomposition $\omega = \eta_0 + \eta_1$ where $\eta_0 \in \extp^3\ann(W)$ and $\eta_1 \in \extp^2\ann(W)\otimes W'$. Notice that $\eta_0$ is decomposable for $\dim \ann(W) = 4$. On the other hand, we determine $\eta_1$ via the action of $\SL_4\times \SL_3 \subset \SL_7$ preserving $\ann(W)$ and $W'$. Fixing a basis $\{e^5, e^6, e^7\}$ for $W'$ we write $\eta_1 = \gamma_1\wedge e^5 + \gamma_2\wedge e^6 + \gamma_3\wedge e^7$. Consider the subspace $H = \spann{\gamma_1,\gamma_2,\gamma_3} \subset \extp^2\ann(W)$. Note that $\dim H = 3$, otherwise, $\omega$ would have rank at most $6$. For a similar reason, the support of $H$ spans $\ann(W)$; then, we may assume that $\{\gamma_1,\gamma_2,\gamma_3\}$ is as in Lemma \ref{lem:auxV7} for $\{e^1, e^2,e^3,e^4\}$ a basis of $\ann(W)$.

For the first case of Lemma \ref{lem:auxV7}, if $\eta_0 =0$ we get 
\[
\omega = e^{12}\wedge e^5 + e^{13}\wedge e^6 + e^{14}\wedge e^7 = e^1\wedge (e^{25}+e^{36}+e^{47}) = \omega_1.
\]
If $\eta_0\neq 0$ then, up to changing $W'$, we get $\eta_0 = e^{234}$, hence 
\[
\omega = e^1\wedge (e^{25}+e^{36}+e^{47}) + e^{234} = \omega_2.
\]
In the second case, up to changing $W'$ we can assume $\eta_0 = 0$. Then
\[
\omega = e^{125} + e^{136} + e^{147} + e^{237} = e^{1}\wedge( e^{25} + e^{36} + e^{47}) + e^{237},
\]
which is in the orbit of $\omega_2$. In the third case, we can assume again that $\eta_0 =0$ and then 
\[
\omega = e^{125}+e^{236}+e^{347} = \omega_3.
\]
In the fourth case, once more we can assume $\eta_0=0$. Then 
   $$ \omega  = e^{125} + e^{346} + e^{127} + e^{147} -e^{237}+e^{347}
    = e^1\wedge(e^2\wedge(e^5+e^7) + e^{47}) + e^3\wedge(e^{27}+e^4\wedge (e^6+e^7)),$$
so that it belongs to the orbit of $\omega_4$.
\end{proof}

By semi-continuity of the rank of $q_\omega$, the irreducible hypersurface $D$ must 
be the closure of the orbit of $\omega_4$. An equation is 
$\det (q_\omega) = 0$, a polynomial of degree $21$ in $\omega$. So it would seem that the hypersurface $D$ is of this degree, but we will see it has in fact degree $7$ (so that $\det(q_\omega)$ is a cube). 

There are several ways to construct an equation of $D$: since it is the unique invariant hypersurface, we just need to construct a semi-invariant. By the First Fundamental Theorem of Invariant Theory, such a semi-invariant can necessarily be constructed in terms of permutations in tensor products, see \cite{procesi}. There are many ways to do that. We could consider the map 
\[
\extp^2V_7^*\stackrel{\wedge\omega}{\lra}\extp^5V_7^*\simeq \extp^2V_7.
\]
The resulting determinant is also a polynomial of degree $21$. The caveat is that this polynomial is zero! 


To find a degree seven semi-invariant defining $D$, consider the following equivariant morphism:
\begin{align*}
   \Phi\colon S^2(\extp^3V^*_7) \otimes V_7^\ast &\longrightarrow \extp^3V_7^* \otimes \extp^4V_7^* \simeq \End(\extp^3V_7)\otimes \det(V_7^*),\\
    (\alpha \cdot \beta, \gamma) &\longmapsto \frac{1}{2}(\alpha \otimes \beta \wedge \gamma + \beta\otimes \alpha \wedge \gamma).
\end{align*}
The natural action of $\GL_7$ defines an equivariant morphism $\End(V_7) \to \End(\extp^3V_7)$. Transposing and composing with $\Phi$, we get a morphism $S^2(\extp^3V^*_7) \otimes V_7^\ast \to \End(V_7)\otimes \det(V_7^*)$, hence,
\[
S^2(\extp^3V^*_7)  \longrightarrow \End(V_7)\otimes V_7\otimes \det(V_7^*).
\]
Taking the symmetric square of this morphism, and  natural inclusions and projections, we get 
\[
S^4(\extp^3V^*_7) \hookrightarrow S^2(S^2(\extp^3V^*_7))  \to S^2(\End(V_7)\otimes V_7\otimes \det(V_7^*)) \twoheadrightarrow S^2\End(V_7)\otimes S^2V_7\otimes \det(V_7^*)^{\otimes 2};
\]
composing in the first factor with the trace form $S^2\End(V_7)\to\CC$, we deduce a map
\begin{equation}\label{eq:mapv71}
    S^4(\extp^3V^*_7)  \longrightarrow S^2\End(V_7)\otimes S^2V_7\otimes \det(V_7^*)^{\otimes 2} \longrightarrow S^2V_7\otimes \det(V_7^*)^{\otimes 2}
\end{equation}
On the other hand, the map $\omega\mapsto q_\omega$ from \eqref{eq:quadV7} defines, via polarization, an equivariant morphism 
$$S^3(\extp^3V_7^*) \to S^2V_7^*\otimes\det(V_7^*).$$ 
Contracting it with \eqref{eq:mapv71} and symmetrizing yields an equivariant map $S^7(\extp^3V^*_7) \to \det(V_7^*)^{\otimes 2}$; this is our desired degree-seven semi-invariant. 

\noindent\textit{Claim}. This semi-invariant is nonzero. Consequently, it gives an equation of 
$D$, and it cannot be a power. 
This degree-seven semi-invariant was first constructed in \cite{Kim_inv}. Moreover, it can also be constructed through a procedure called Matrix Factorization (see \cite{AbMan}).

The full classification of complex three-forms in seven variables can be summarized in the following diagram. The numbers are the dimensions of the orbits and the arrows show how their closures are contained in each other \cite{abo-ott}:
\[
\begin{tikzcd}
             &              &                          & 28 \arrow[r] \arrow[rdd] & 21 \arrow[rd] &              &              &   \\
35 \arrow[r] & 34 \arrow[r] & 31 \arrow[ru] \arrow[rd] &                          &               & 20 \arrow[r] & 13 \arrow[r] & 0 \\
             &              &                          & 26 \arrow[r]             & 25 \arrow[ru] &              &              &
\end{tikzcd}
\]
The orbits of rank-seven forms are those of dimensions $21, 28,31,34$, and $35$.


\begin{remark}\label{rem:v7-real-case}
Again the situation is more complicated over the real numbers. For example, there are two open orbits, represented by the canonical three-forms for the real octonions and their split version (see Remark~ \ref{rem:generic}). We refer to \cite[\S5]{noui-revoy} for a classification over an arbitrary field.
\end{remark}

\section{Composition algebras and the octonions}

Many interesting objects of study in geometry, both classical and exceptional, can be traced back to simple algebraic structures permeating our everyday life: the so-called composition algebras.

\begin{definition}
A \emph{composition algebra} (or \emph{Hurwitz algebra}) is a unital, not necessarily associative or commutative algebra $(\AA,\cdot\,)$ over a field $\KK$, together with a nondegenerate symmetric bilinear form (the \emph{inner product})
\[\langle\,\cdot\,,\cdot\,\rangle \colon \AA\times\AA\to\KK\]
such that its associated quadratic form $q\colon \AA\to\KK$, $q(x)=\langle x,x\rangle$ (the \emph{norm}), is \emph{multiplicative}, i.e.
\[q(u\cdot v)=q(u)q(v) \quad \forall u,v\in\AA.\]
\end{definition}

The \emph{real part} of a composition algebra is $\Re\AA:=\KK \, 1_\AA$, and the \emph{imaginary part} is its orthogonal complement $\Im\AA:=(\Re\AA)^\perp$. We always have $q(1_\AA)=1$, thus $\AA=\Re\AA\oplus\Im\AA$. Any composition algebra comes with an involution $x\mapsto\bar x$ called \emph{conjugation}, defined by reflection along $\Re\AA$:
\[\bar{x}=2\langle x,1_\AA\rangle 1_\AA-x,\qquad x\in\AA.\]
Many important properties and calculational rules for the conjugation and bilinear form on composition algebras are contained in \cite{springerveldkamp}.\footnote{Note that this book uses the notation $\langle\,\cdot\,,\cdot\,\rangle$ for \emph{twice} the inner product, that is, $q(x)=\frac{1}{2}\langle x,x\rangle$.} For example, conjugation is an antiautomorphism:
\begin{equation}
    \overline{(xy)}=\bar{y}\bar{x}\qquad\forall x,y\in\AA,\label{eq:conjantiauto}
\end{equation}
and it recovers the norm:
\begin{equation}
    x\bar{x}=\bar{x}x=q(x)1_\AA\qquad\forall x\in\AA.\label{eq:norm}
\end{equation}

Let us for now focus on algebras over the field $\KK=\RR$. A real composition algebra $\AA$ whose quadratic form $q$ is positive-definite  is called a \emph{Euclidean Hurwitz algebra}. Such an $\AA$ is automatically a \emph{division algebra} -- that is, every nonzero element $u\in\AA$ has a two-sided multiplicative inverse $u^{-1}$ \cite[Prop.~1.3.4]{springerveldkamp}. As a consequence, there are no zero divisors:
\[x\cdot y=0\quad\Longrightarrow\quad x=0\text{ or }y=0,\qquad x,y\in\AA.\]
We hence refer to $\AA$ also as a \emph{normed division algebra}.

On the other hand, if $q$ is indefinite (or, equivalently, if there exist zero divisors), the algebra $\AA$ is called \emph{split}. A vector subspace on which $q$ vanishes identically is called \emph{isotropic}.

\begin{definition}
Let $\AA$ be a composition algebra over $\RR$. Its \emph{complexification} $\AA^\CC$ is the algebra over $\CC$ given by $\AA\otimes_{\RR}\CC$, with complex-bilinearly extended product.
\end{definition}

The complex algebra $\AA^\CC$ is again a composition algebra. Complexification is a useful tool: it may simplify the representation theory of groups and algebras, the classification of bilinear forms (signature disappears), and as we have seen, also three-forms in low dimensions (there is only one open orbit of $\GL_7(\CC)$ on $\extp^3\CC^7$, instead of two  in the real case).

Let us review two  normed division algebras whose existence has far-reaching consequences.

\subsection*{Quaternions} The quaternions $\HH$ are a four-dimensional normed division algebra discovered by Hamilton in 1843. As a vector space over $\RR$, it is spanned by the basis vectors $1,I,J,K$, which satisfy the relations
\[I^2=J^2=K^2=IJK=-1.\]
Any quaternion $x\in\HH$ can be written as a linear combination
\[x=x_0+x_1I+x_2J+x_3K,\qquad x_0,x_1,x_2,x_3\in\RR.\]
The \emph{real part} is $\Re(x)=x_0$.

The quaternions include the algebra of complex numbers $\CC=\mathrm{span}\{1,I\}$ as a subalgebra. Complex conjugation extends to $\HH$ as
\[\bar x=x_0-x_1I-x_2J-x_3K,\]
and the inner product is given by
\[\langle x,y\rangle=\frac{1}{2}(x\bar y+y\bar x)=\Re(x\bar y).\]
This gives rise to the norm
\[q(x)=|x|^2=x\bar x=\bar x x.\]
The norm is multiplicative and thus gives $\HH$ the structure of a normed division algebra. It is associative, but not commutative.

The quaternions can be built from the complex numbers via a construction known as the \emph{Cayley-Dickson doubling process} (\necessary Exercise \ref{ex_CD_H}). This process can be iterated also over the quaternions, and this begs the question: what is the result of doubling the quaternions?

\subsection*{Octonions}

Discovered by Graves in 1843, and independently by Cayley in 1845, 
the octonions are a real algebra $\OO$ defined by doubling the quaternions via the Cayley--Dickson rule \eqref{eq:cayleydickson}. It is thus an eight-dimensional algebra over $\RR$, spanned by the basis vectors $e_0=1,e_1,\ldots,e_7$. Their multiplication is encoded in the \emph{Fano plane} (the projective plane over the field with two elements), which has seven points (labeled by the basis vectors $e_1,\ldots,e_7$) and seven lines:
\small 
\begin{center}
\setlength{\unitlength}{4mm}
\begin{picture}(30,11)(-15,-.6)
\put(-5.5,0){ \resizebox{!}{1.5in}{\includegraphics{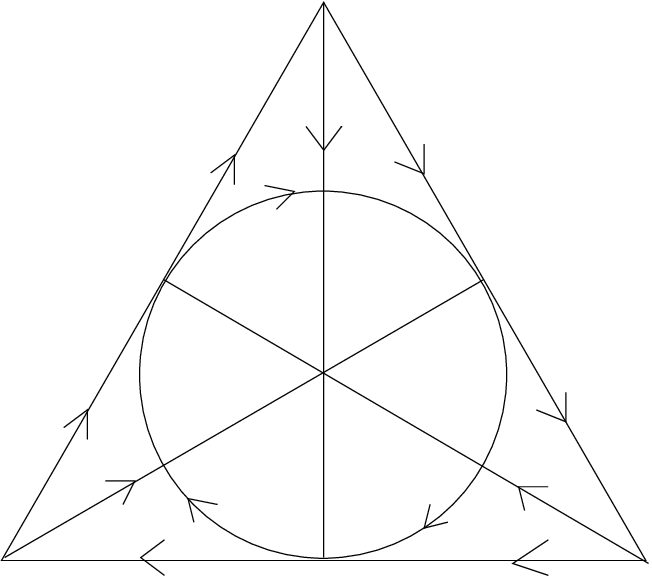}}}
\put(-3.8,4.7){$e_2$}\put(3.35,4.7){$e_6$}
\put(-.3,-.8){$e_4$}\put(-6.1,-.8){$e_1$}
\put(5.5,-.8){$e_5$}\put(-.2,9.9){$e_3$}\put(.9,3.2){$e_7$}
\end{picture} 
\end{center}
\normalsize

The elements $e_1,\ldots,e_7$ are imaginary units: we have
\[e_1^2=\cdots=e_7^2=-1.\]
 Each line $(e_i\to e_j\to e_k)$ in the Fano plane, cyclically oriented as above, gives the rule
\[e_i\cdot e_j=e_k=-e_j\cdot e_i.\]
The conjugation of $\OO$ is defined by
\[\bar{1}=1,\qquad\overline{e_i}=-e_i\quad\text{for }i=1,\ldots,7.\]
Again, $\langle x,y\rangle=\frac{1}{2}(x\bar y+y\bar x)=\Re(x\bar y)$ defines an inner product (for which $e_0:=1,e_1,\ldots,e_7$ form an orthonormal basis) and a norm, making $\OO$ into a normed division algebra. It is neither commutative nor associative! However it is both alternative (Exercise \ref{assocsubalg}) and satisfies the so-called Moufang identities (Exercise \ref{ex_moufang}). For more information, see e.g. \cite{baez}.

\subsection*{Hurwitz's Theorem}

\begin{theorem}[Hurwitz 1898]
\label{hurwitz}
There exist only four real normed division algebras: $\RR$, $\CC$, $\HH$, and $\OO$.
\end{theorem}

As a consequence of this theorem and Exercise~\ref{assocsubalg}, any two elements in $\OO$ are contained in a subalgebra isomorphic to $\HH$. Hurwitz's Theorem, as well as the concept of quaternions and octonions, has been generalized to other fields \cite{jacobson58}:

\begin{theorem}[Hurwitz; Jacobson 1958]
\label{hurwitzgeneral}
Let $\AA$ be a composition algebra over a field $\KK$ of characteristic $\neq2$. Then one of the following holds:
\begin{enumerate}
    \item $\dim_\KK\AA=1$ and $\AA\simeq\KK$,
    \item $\dim_\KK\AA=2$ and either $\AA\simeq\KK\oplus\KK$ or $\AA$ is a quadratic field extension of $\KK$,
    \item $\dim_\KK\AA=4$ (and $\AA$ is called a \emph{quaternion algebra} over $\KK$),
    \item $\dim_\KK\AA=8$ (and $\AA$ is called an \emph{octonion algebra} over $\KK$).
\end{enumerate}
Moreover, two composition algebras are isomorphic if and only if their underlying quadratic vector spaces $(\AA,q)$ are isomorphic.
\end{theorem}

Examples of non-Euclidean quaternion algebras can be found in Exercise \ref{ex_noneucl}.

\subsection*{Cross products}

Normed division algebras are related to so-called vector cross products.

\begin{definition}
Let $(V,\langle\cdot\,,\cdot\,\rangle)$ be an Euclidean vector space. A \emph{vector cross product} is a bilinear map
\[\times\colon V\times V\to V\]
which is totally skew-symmetric, i.e.,
\[\langle u\times v,w\rangle=-\langle v\times u,w\rangle=\langle v\times w,u\rangle\qquad\forall u,v,w\in V,\]
and satisfies the \emph{magnitude} condition
\begin{equation}
    |u\times v|^2=|u\wedge v|^2:=|u|^2|v|^2-\langle u,v\rangle^2\qquad\forall u,v\in V.
    \label{eq:magnitude}
\end{equation}
\end{definition}

Any vector cross product gives rise to a three-form $\omega\in\extp^3V^\ast$ via
\[\omega(u,v,w)=\langle u\times v,w\rangle, \qquad u,v,w\in V.\]
Equation \eqref{eq:magnitude} is equivalent to the condition that for every unit vector $u\in V$, the skew-symmetric endomorphism $u\times \colon V\to V$ restricted to $u^\perp\subset V$ is orthogonal (and thus a complex structure). In particular, $\dim V$ has to be odd.

In fact, from any vector space with vector cross product one can construct a normed division algebra (\necessary Exercise \ref{ex:divalgfromveccross}). Conversely:

\begin{theorem}
Let $\AA$ be a real normed division algebra with imaginary part $V:=\Im\AA$. Then
\[u\times v:=\Im(uv),\qquad u,v\in V,\]
defines a vector cross product on $V$.
\end{theorem}

\begin{proof}
As a consequence of Theorem~\ref{hurwitz}, $\AA$ is isomorphic to a subalgebra of the octonions and thus alternative. The definition of octonionic multiplication implies that
\[x\times y=\Im(xy)=-\Im(yx)=-y\times x\]
for all $x,y\in\Im\AA$. Additionally, we have
\[\langle x\times y,z\rangle=\Re((x\times y)\bar z)=-\Re((xy)z)\]
for $x,y,z\in\Im\AA$, since $\bar z=-z$. By alternativity, and since the square of any imaginary octonion is real, we deduce that
\[\langle x\times y,y\rangle=-\Re((xy)y)=-\Re(xy^2)=|y|^2\Re(x)=0,\]
and the total skew-symmetry of $\times$ follows by polarization. To check the magnitude condition, we use the fact that $\overline{xy}=\bar y\bar x$ together with the Moufang identities and compute
\begin{align*}
    |x\times y|^2&=|\Im(xy)|^2=|xy|^2-|\Re(xy)|^2=\Re((xy)(\overline{xy}))-\langle x,y\rangle^2=\Re((xy)(yx))-\langle x,y\rangle^2\\
    &=\Re((xy^2)x)-\langle x,y\rangle^2=-|y|^2\Re(x^2)-\langle x,y\rangle^2=|x|^2|y|^2-\langle x,y\rangle^2. \qedhere
\end{align*}
\end{proof}

Together with Theorem~\ref{hurwitz}, these results imply that nontrivial cross products exist only in three or seven dimensions, and are related to the quaternions or octonions, respectively \cite{browngray}.

\subsection*{The associative three-form}

We return to the study of the octonion algebra $\OO$, whose imaginary part is  $\Im\OO=\spann{e_1,\ldots,e_7}\simeq\RR^7$. As in Exercise~\ref{ex:divalgfromveccross}, octonionic multiplication gives rise to a vector cross product on $\RR^7$, and thus to a canonical three-form.

\begin{definition}\label{def:3-form}
The \emph{associative} three-form $\omega\in\extp^3(\RR^7)^\ast$ is defined by
\begin{equation}\label{eq:inv-form}
    \omega(x,y,z)=\langle\Im(xy),z\rangle=-\Re((xy)z),\qquad x,y,z\in\OO.
\end{equation}
If $(e^1,\ldots,e^7)$ is the dual basis of $(e_1,\ldots,e_7)$, we may also write
\[\omega=\sum_{(e_i\to e_j\to e_k)}e^i\wedge e^j\wedge e^k,\]
where the sum ranges over all oriented lines of the Fano plane. Explicitly,
\begin{equation}
\omega=e^{123}+e^{365}+e^{541}+e^{264}+e^{176}+e^{572}+e^{374}.\label{eq:omegainbasis}
\end{equation}
\end{definition}

Taking $(e_1,\ldots,e_7)$ as positively oriented, the information of $\omega$ is equivalently included in its Hodge dual $\star\omega\in\extp^4(\RR^7)^\ast$, given by
\[\star\omega=e^{4567}-e^{1247}-e^{2367}-e^{1357}-e^{2345}+e^{1346}-e^{1256}.\]
This is also called the \emph{coassociative four-form}. For more information, see \cite{harveylawson}, \cite{harvey}.

\begin{remark} \label{rem:generic}
Later, we will have to complexify the octonions and consider the complex-trilinear extension of $\omega$ to $\extp^3(\CC^7)^\ast$. By abuse of notation, we will still call it $\omega$. In the complex setting, one may check that $\omega$ is obtained from $\omega_5$ of Theorem~\ref{thm:schouten} by the following change of basis: 
\[
\begingroup 
\setlength\arraycolsep{10pt}
\begin{matrix*}[l]
       e^1 \mapsto \lambda \,e^1, & e^2 \mapsto \lambda (e^5 -\ii\,e^2), & e^3 \mapsto \lambda^2(e^3 + \ii\, e^6),  & e^4 \mapsto e^4 -\ii\,e^7, \\
   e^5 \mapsto -\lambda (e^5+\ii\,e^2), &  e^6 \mapsto e^3 - \ii\,e^6, & e^7 \mapsto \lambda^2(e^4+\ii\,e^7), &
\end{matrix*}
\endgroup
\]
where $\ii^2 = -1$ and $\lambda^3 = -\ii/2$. In particular, the three-form $\omega\in\extp^3(\CC^7)^\ast$ is \emph{generic} in the sense that its $\GL_7(\CC)$-orbit is open (see Section~\ref{sec:nondeg7}).
\end{remark}

\subsection*{Composition algebras with zero divisors}

It is possible to modify the Cayley--Dickson construction to obtain split composition algebras. Starting with a composition algebra $\AA$, let the product on $\AA\oplus\AA$ be defined by
\begin{equation}
(a,b)\cdot(c,d)=(ac+\overline{d}b,b\overline{c}+da)
\label{eq:modcayleydickson}
\end{equation}
for all $a,b,c,d\in\AA$ (notice the difference in sign to \eqref{eq:cayleydickson}). If $\AA=\RR$ or $\AA$ is associative and split, then $\AA\oplus\AA$ will again be a split composition algebra. Starting from $\RR$, we successively obtain the algebras of \emph{split-complex numbers} $\CC'$, \emph{split quaternions} $\HH'$, and \emph{split octonions} $\OO'$. These all have zero divisors, and their inner product is not positive definite.

As a consequence of Theorem~\ref{hurwitzgeneral}, both $\OO$ and $\OO'$ have isomorphic complexification $\OO^\CC$ (since over $\CC$ nondegenerate quadratic forms are equivalent). These are called \emph{complex octonions} or \emph{bioctonions}. In split or complex composition algebras, zero divisors are of importance. Therefore, we state two results characterizing zero divisors.

\begin{lemma}
\label{zerodivnorm}
Zero divisors in a composition algebra are precisely the elements of norm zero.
\end{lemma}
\begin{proof}
Let $\AA$ be a composition algebra, and $x\in\AA$ such that $q(x)=0$. Then, by \eqref{eq:norm}, $x\bar{x}=0$, hence $x$ is a zero divisor. Conversely, let $x,y\in\AA$ be nonzero with $xy=0$. It follows from the properties of composition algebras \cite[Lem.~1.3.3]{springerveldkamp} that
\[
    0=(xy)\bar{y}=q(y)x, \quad \mathrm{and} \quad 0=\bar{x}(xy)=q(x)y,
\]
and thus $q(x)=q(y)=0$.
\end{proof}

This immediately implies that if $x\in\AA$ is a zero divisor, then $x(\bar{x}z)=q(x)z=0$ for any $z\in\AA$. There is also a converse:

\begin{theorem}\label{kerofleftmul}
Let $\AA$ be a composition algebra, and $x,y\in\AA$ such that $xy=0$. Then there exists $z\in\AA$ such that $y=\bar{x}z$, and $\bar{x}\AA$ is a maximal isotropic subspace for $q$.
\end{theorem}
\begin{proof}
For $x\in\AA$, let $L_x\colon \AA\to\AA$ denote left-multiplication, $ L_x(y):=xy$. Suppose that $x$ is a zero divisor. We show that $\ker L_x$ and $\im L_x=x\AA$ are isotropic subspaces.

First, any element $y\in\ker L_x$ satisfies $xy=0$ and is thus a zero divisor. So $q(y)=0$ by Lemma~\ref{zerodivnorm}. Thus $\ker L_x$ is isotropic.

Second, let $a\in\AA$. Applying the calculational rules in \cite[Lem.~1.3.2, 1.3.3]{springerveldkamp}, we find
\[q(xa)1_\AA=\langle xa,xa\rangle=\langle a,\bar{x}(xa)\rangle=\langle a,q(x)a\rangle=0.\]
Thus $x\AA$ is isotropic.

Since $\ker L_x=(\bar x\AA)^\perp$ and $q$ is nondegenerate, the dimension formula
\[\dim\bar x\AA+\dim\ker L_x=\dim\AA\]
holds. Recalling that isotropic subspaces are of dimension at most $\frac{1}{2}\dim\AA$, it follows that both $\bar x\AA$ and $\ker L_x$ have this dimension and are thus maximal isotropic. Furthermore, since it has already been remarked above that $\bar x\AA\subseteq\ker L_x$, it follows that $\bar x\AA=\ker L_x$.
\end{proof}

\begin{cor}
\label{kerofleftmulim}
Let $x\in\Im\AA$ be a zero divisor, and $K_x=\ker(L_x|_{\Im\AA})$ the kernel of left-multiplication by $x$ on $\Im\AA$. Then $K_x=x\AA\cap\Im\AA$ is an isotropic subspace of dimension $\frac{\dim\AA}{2}-1$ containing $x$.
\end{cor}

We may also ask about subspaces of a composition algebra on which the multiplication vanishes identically. In the complex octonions, it turns out that the maximal such subspaces are two-dimensional (see Lemma \ref{lem:dim-null-planes}).

\begin{definition}
\label{def:null-plane}
A two-dimensional subspace $N\subset\OO^\CC$ is called a \emph{null-plane} if $xy=0$ for all $x,y\in N$.
\end{definition}

\begin{lemma}
\label{lem:null-isotropic}
Null-planes are isotropic and purely imaginary.
\end{lemma}
\begin{proof}
Let $N$ be a null-plane. Then for every $x\in N$, we have $x^2=0$, hence  $q(x)=0$ by  Lemma~\ref{zerodivnorm}. Thus $q$ vanishes when restricted to $N$. Moreover
\begin{align*}
0&=x^2+\bar{x}x=2\Re(x)x,
\end{align*}
so necessarily $\Re(x)=0$.
\end{proof}

We will see in the second lecture that the set of null-planes is homogeneous under the action of the complex Lie group $G_2$, see Corollary \ref{G2G27_hom}.


\section{Automorphisms} 

Let us finally turn to our desired object of study: the exceptional Lie group $G_2$, which arises as the automorphism group of the octonions. For any composition algebra $\AA$ over $\KK=\RR,\CC$, its automorphism group is the Lie group defined by
\[\Aut(\AA)=\left\{U\in\GL(\AA)\,|\,U(xy)=U(x)U(y)\quad\forall x,y\in\AA\right\}.\]
The group $\GL(\AA)$ is to be understood over the same field $\KK$.

By \cite[Cor.~1.2.4]{springerveldkamp}, the norm of $\AA$ is determined by its algebra structure. Thus, $\Aut(\AA)$ is a closed subgroup of the \emph{orthogonal group} $\rmO(\AA,q)$. 
Since automorphisms preserve $1_\AA$ and the norm, they preserve the splitting into real and imaginary parts and are determined by their action on the subspace $\Im\AA$. Thus, we may view $\Aut(\AA)$ as a subgroup of $\rmO(\Im\AA,q)$.

The Lie algebra of $\Aut(\AA)$ is the set of \emph{derivations} of $\AA$, that is,
\[\Der(\AA)=\left\{U\in\fgl(\AA)\,|\,U(xy)=U(x)y+xU(y)\right\}.\]

From now on, we work over  $\KK=\CC$, and take Lie groups and Lie algebras to be \emph{complex} by default.

\begin{theorem}
For the associative complex composition algebras, we have
\begin{align*}
    \Aut(\RR^\CC)&=1,&\Aut(\CC^\CC)&=\ZZ_2,&\Aut(\HH^\CC)&=\SO_3,\\
    \Der(\RR^\CC)&=0,&\Der(\CC^\CC)&=0,&\Der(\HH^\CC)&=\fso_3.
\end{align*}
\end{theorem}

\begin{proof}
Since $\Aut(\AA)\subset\rmO(\Im\AA,q)$, we have
\[\Aut(\RR^\CC)\subset\rmO_0=1,\qquad\Aut(\CC^\CC)\subset\rmO_1=\ZZ_2,\qquad\Aut(\HH^\CC)\subset\rmO_3.\]
The first case is trivial. For the second case, we note that the conjugation $z\mapsto\bar z$ is a nontrivial automorphism by \eqref{eq:conjantiauto} and  commutativity, thus $\Aut(\CC^\CC)=\ZZ_2$. For the third case, we observe that $\Aut(\HH^\CC)$ is precisely the stabilizer in $\rmO_3(\CC)$ of the nontrivial three-form $\omega\in\extp^3(\Im\HH^\CC)^\ast$ given by
\[    
\omega(x,y,z)=\langle\Im(xy),z\rangle.
\]
Since $\extp^3(\CC^3)^\ast\simeq\CC$, $\omega$ is a multiple of the volume form, and its stabilizer is
\[\Aut(\HH^\CC)=\rmO_3\cap\SL_3=\SO_3. \qedhere \] 
\end{proof}

\begin{definition}
We define the following Lie groups as automorphism groups of the complex, real, and split octonions, respectively:
\[G_2=\Aut(\OO^\CC),\qquad G_2^{\mathrm{c}}=\Aut(\OO),\qquad G_2'=\Aut(\OO').\]
Moreover we denote the respective Lie algebras by $\fg_2$, $\fg_2^{\mathrm{c}}$, $\fg_2'$, respectively.
\end{definition}

\begin{definition}
A \emph{basic triple} is a triple $(x,y,z)$ of imaginary unit octonions such that $\langle x,y\rangle=0$, and $z$ is orthogonal to the (quaternion) subalgebra generated by $\{x,y\}$.
\end{definition}

In fact $G_2^{\mathrm{c}}$ acts simply transitively on  basic triples; moreover $G_2^{\mathrm{c}}$ is connected  of dimension $14$ (\necessary Exercise \ref{ex_dim14}). The analogous statement holds for $G_2$ and the complex octonions $\OO^\CC$.

\begin{cor}\label{g2trans}
The group $G_2^{\mathrm{c}}$ acts transitively on the unit sphere $S^6\subset\Im\OO$.
\end{cor}

In a sense, octonions are often responsible for the existence of exceptional objects in mathematics. Their derivations give the smallest exceptional Lie algebra in the Cartan--Killing classification:

\begin{theorem}
\label{g2simple}
$\fg_2$ is an exceptional simple Lie algebra.
\end{theorem}

The same statement holds true over $\RR$ for the octonions $\OO$ (Cartan 1914) and the split octonions $\OO'$, provided that $G_2$ is respectively replaced by its \emph{compact real form} $G_2^{\mathrm{c}}\subset\SO_7(\RR)$ or \emph{split real form} $G_2^\ast\subset\SO(3,4)$ \cite{agricolag2,draperg2}. The other exceptional Lie algebras will be addressed in Lecture 3. For now, let us focus on some representation-theoretic properties of $G_2$.

\section{Representations and invariants of \texorpdfstring{$G_2$}{G2}}
\label{sec:g2rep}

\begin{theorem}
\label{omegageneric}
$G_2^{\mathrm{c}}$ (resp.~$G_2$) is the stabilizer of the associative three-form $\omega\in\Lambda^3(\RR^7)^\ast$ (resp.~its complexification), which is generic in the sense of Section~\ref{sec:nondeg7}. More precisely, the associated quadratic form $q$ defined in \eqref{eq:quadV7} induces the octonionic norm in the sense that
\begin{equation}
    q_\omega(v)=6|v|^2\otimes e^{1234567}\qquad\forall v\in\RR^7.\label{eq:normfromomega}
\end{equation}
\end{theorem}
\begin{proof}
As noted above, $G_2$ preserves the octonionic multiplication, the norm, and the imaginary part. Thus, it also stabilizes the associative three-form \[\omega(x,y,z)=\langle\Im(xy),z\rangle.\]
This shows that $G_2\subseteq\Stab(\omega)$.

By Corollary~\ref{g2trans}, $G_2^{\mathrm{c}}$ acts transitively on the unit sphere while preserving $\omega$. Thus, we may take without restriction $v=e_1$ and with the help of \eqref{eq:omegainbasis} compute
\begin{align*}
    v\intprod\omega&=e^{23}+e^{54}+e^{76},\\
    (v\intprod\omega)^2&=2e^{2354}+2e^{2376}+2e^{5476},\\
    (v\intprod\omega)^2\wedge\omega&=6e^{1234567}.
\end{align*}
After homogenizing, we obtain \eqref{eq:normfromomega}. This means that $q_\omega$ is nondegenerate, that is, $\omega$ is a generic three-form.

In particular, \eqref{eq:normfromomega} implies that every element in $\GL_7(\RR)$ preserving $\omega$ also preserves the volume form $e^{1234567}$ and the inner product on $\RR^7$. Thus, it will also preserve the vector cross product on $\RR^7$, and by Exercise~\ref{ex:divalgfromveccross} (and Hurwitz's Theorem) also the multiplication on $\OO=\RR\oplus\RR^7$. This shows that $\Stab(\omega)\subseteq G_2$.
\end{proof}

 Let $V_7=\Im\OO^\CC$ be the \emph{standard representation} of $G_2$. As $G_2$ preserves a nondegenerate bilinear form on $V_7$, this representation is self-dual, i.e.,~$V_7\simeq V_7^\ast$. The same thus goes for exterior or symmetric powers of $V_7$. In the sequel, we will freely identify these representations with their dual.
For example, $S^2V_7$ identifies with the space of quadratic forms on $V_7$, so it contains an invariant line generated by $q_\omega$. An invariant complement is the kernel $S^2_0V_7$ of the  contraction with $q_\omega$. 

\begin{prop}\label{prop:wedge3}
The representation $\extp^3V_7$ of $G_2$ decomposes into invariant subspaces as
\[\extp^3V_7\simeq\CC\oplus V_7\oplus S^2_0V_7.\]
\end{prop}

\begin{proof}
First note that $\End(V_7)\simeq \CC\oplus S^2_0V_7\oplus\extp^2V_7$ under $\rmO_7$ and thus under $G_2$. Further, the Lie algebra $\fg_2\subset\fso_7\simeq\extp^2V_7$ is invariant, and irreducible by Theorem~\ref{g2simple}. The equivariant map
\[V_7\lra\extp^2V_7,\qquad v\longmapsto v\intprod\omega\]
is injective since $\omega$ is nondegenerate. For dimensional reasons, it follows that
\[\extp^2V_7\simeq V_7\oplus\fg_2.\]
The infinitesimal action on $\omega$ now gives a map
\[\End(V_7)\lra\extp^3V_7,\qquad A\longmapsto A_\ast\omega.\]
As a consequence of Theorem~\ref{omegageneric}, the kernel of this map is precisely $\fg_2$, and counting dimensions we see that it must be surjective. Using the decomposition of $\End(V_7)$, we arrive at the desired statement.
\end{proof}

\begin{remark}
The representations occurring in this decomposition are irreducible. This is clear for $V_7=\Im\OO^\CC$, since $G_2^{\mathrm{c}}$ acts transitively on $S^6\subset\Im\OO$ by Corollary~\ref{g2trans}. One may employ the theory of highest weights to see that there is a $27$-dimensional irreducible representation of $G_2$, necessarily contained in $S^2V_7$; hence it must coincide with the $27$-dimensional representation $S^2_0V_7$.
\end{remark}


The Cartan--Killing classification of simple Lie algebras is mainly based on root systems, which are finite combinatorial objects encoding the Lie algebra structure. The root system of type $G_2$ is made of twelve roots, described as vectors in a real plane. Up to symmetries one can choose
so-called simple roots, a short one $\alpha_1$ and a long one $\alpha_2$, with an angle of $5\pi/6$ between them; the full root system is depicted in Figure~\ref{fig:g2roots}. in terms of $\alpha, \beta, \gamma$, such that $\alpha+\beta+\gamma=0$,  the six short roots are $\pm\alpha,\pm\beta,\pm\gamma$, and the six long roots are $\pm(\alpha-\beta),\pm(\beta-\gamma),\pm(\gamma-\alpha)$.

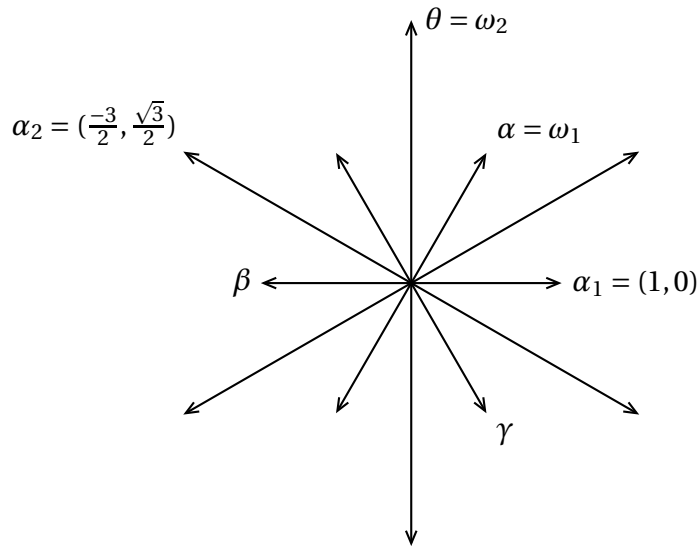
\begin{figure}[H]
\centering
\begin{tikzpicture}[
    -{Straight Barb[bend,
       width=\the\dimexpr10\pgflinewidth\relax,
       length=\the\dimexpr12\pgflinewidth\relax]},
  ]
    \foreach \ii in {0, 1, ..., 5} {
      \draw[thick] (0, 0) -- (\ii*60:2);
      \draw[thick] (0, 0) -- (30 + \ii*60:3.5);
    }
    \node{};
    \node[right, inner sep=.5em] at (90:3.5) {$\theta=\omega_2$};
    \node[above right] at (2*30:2) {$\alpha=\omega_1$};
    \node[right] at (2, 0) {$\alpha_1= (1,0)$};
    \node[above left, inner sep=.2em] at (5*30:3.5) {$\alpha_2=(\frac{-3}{2},\frac{\sqrt{3}}{2})$};
    \node[left] at (6*30:2) {$\beta$};
    \node[below right] at (10*30:2) {$\gamma$};
  \end{tikzpicture}
  \caption{The root system of type $G_2$}
  \label{fig:g2roots}
\end{figure}

Dual to the simple roots are the \emph{simple coroots} or \emph{fundamental weights}  $$\omega_1=2\alpha_1+\alpha_2=\alpha, \qquad \omega_2=3\alpha_1+2\alpha_2=\alpha-\gamma.$$ To every representation of a semisimple Lie algebra one can associate a set of weights, and any finite-dimensional irreducible representations is characterized (up to equivalence) by its \emph{highest weight}. The possible highest weights are precisely the \emph{dominant integral weights}, that is, nonnegative integer linear combinations of the fundamental weights. In the following, for every such dominant integral weight $\lambda$ we shall denote the irreducible representation of highest weight $\lambda$ by $V_\lambda$. In this notation, we have
\begin{align*}
    V_{\omega_1}&=V_7,&
    V_{\omega_2}&=\fg_2,&
    V_{2\omega_1}&=S^2_0V_7.
\end{align*}
For a thorough introduction to the theory of roots and weights, we refer to \cite{bourbaki,FH}.

\smallskip
The Lie algebra $\fg_2$ may be directly constructed from the root system as
\[\fg_2=\fh\oplus\bigoplus_{\lambda\text{ root}}\fg_\lambda,\]
where $\fh\simeq\CC^2$ is the Cartan subalgebra, and the $\fg_\lambda$ are the (one-dimensional) root spaces. A detailed discussion is available in \cite[\S22]{FH}. For now, let us mention two remarkable shortcuts via subalgebras of $\fg_2$.

\begin{prop}[\cite{FH}, \S22]\label{g2fromsl3}
The Lie algebra $\fg_2$ may be constructed as
\begin{align*}
\fg_2&=\fsl_3\oplus V_3\oplus V_3^\ast,&\fsl_3&=\fh\oplus\bigoplus_{\lambda\text{ long root}}\fg_\lambda,\\
V_3&=\fg_\alpha\oplus\fg_\beta\oplus\fg_\gamma,&V_3^\ast&=\fg_{-\alpha}\oplus\fg_{-\beta}\oplus\fg_{-\gamma}.
\end{align*}
\end{prop}

\begin{remark}\label{rmk:V_7-roots}
Accordingly, the seven-dimensional representation of $\fg_2$ is obtained as
\[V_7=\CC\oplus V_3\oplus V_3^\ast.\]
For a suitable basis of root vectors,  $\CC=\spann{e_0}$, $V_3=\spann{ x_\alpha,x_\beta,x_\gamma}$, and $V_3^\ast=\spann{x_{-\alpha},x_{-\beta},x_{-\gamma}}$, and there is an invariant three-form given by
\begin{equation}\label{eq:form-root}
    \omega= x^0\wedge x^{\alpha}\wedge x^{-\alpha}+ x^0\wedge x^{\beta}\wedge x^{-\beta}+ x^0\wedge x^{\gamma}\wedge x^{-\gamma} + x^{\alpha}\wedge x^{\beta}\wedge x^{\gamma} + x^{-\alpha}\wedge x^{-\beta}\wedge x^{-\gamma},
\end{equation}
where $x^0, x^{\alpha} \ldots x^{-\gamma}$ denotes the dual basis.
Note that this is precisely the form $\omega_5$ from Theorem~\ref{thm:schouten} by setting $e^1 = x^0$, $e^2 = x^\alpha$, $e^3 = x^{-\alpha}$, $e^4 = x^{\beta}$, $e^5 = x^{-\beta}$, $e^6 = x^{\gamma}$, and $e^7 = x^{-\gamma}$.

The corresponding quadratic form $q$ is 
\begin{equation}\label{eq:quadric-root} 
 q= 2(x^0)^2 -x^{\alpha}x^{-\alpha}-x^{\beta}x^{-\beta}-x^{\gamma}x^{-\gamma},
 \end{equation}
\end{remark}
%

\section{Exercises}

 \begin{exercise}\label{ex:rank-forms} \necessary
    The ranks of the quadratic forms $q_{\omega_i}$ associated to the three-forms $\omega_1$, $\omega_2$, $\omega_3$, $\omega_4$, $\omega_5$ appearing in Theorem \ref{thm:schouten} are respectively $1,1,2,4,$ and $7$.
\end{exercise}

\begin{exercise}
\label{ex_CD_H} \necessary
Since $K=IJ$, one can write $x= (x_0+x_1I)+(x_2+x_3I)J$ and see $\HH=\CC\oplus \CC$. Show that under this isomorphism, quaternion multiplication is given in terms of complex multiplication and conjugation by the rule
\begin{equation}
(a,b)\cdot(c,d)=(ac-\overline{d}b,b\overline{c}+da)
\label{eq:cayleydickson}
\end{equation}
for $a,b,c,d\in\CC$. This way to obtain $\HH$ from $\CC$ is called the \emph{Cayley--Dickson doubling process}. 
\end{exercise}

\begin{exercise}
\label{assocsubalg}
\computation Prove that $\OO$ has the \emph{alternativity} property: for any $x,y\in\OO$,
\[x(xy)=(x^2)y,\qquad(yx)x=y(x^2),\qquad (xy)x=x(yx).\]
Conclude that any two elements of $\OO$ are contained in an associative subalgebra.
\end{exercise}

\begin{exercise}
\label{ex_moufang} \computation Prove the Moufang identities: for all $x,y,z\in\OO$,
\[z(x(zy))=((zx)z)y,\qquad x(z(yz))=((xz)y)z,\qquad(zx)(yz)=(z(xy))z=z((xy)z).\]
\end{exercise}

\begin{exercise}
\curiosity Recall the definition of conjugation in a composition algebra:
\[u\mapsto\bar u:=2\langle u,1_\AA\rangle-u,\qquad u\in\AA.\]
Show that associativity of a normed division algebra $\AA$ is required for its Cayley--Dickson double to be normed. In particular, the \emph{sedenions}, obtained from doubling $\OO$, are not a normed division algebra.
\end{exercise}

\begin{exercise}
\label{ex_noneucl} \training
Show that the complexified quaternions $\HH^\CC$ are isomorphic to the matrix algebra $M_2(\CC)$ (\emph{Hint:} Search for Pauli matrices). In the same vein, show that $M_2(\RR)$ is a four-dimensional real split composition algebra. It is called the algebra of  \emph{split-quaternions}.
\end{exercise}

\begin{exercise}
\label{ex:divalgfromveccross} \necessary
 Show that if $(V,\langle\cdot\,,\cdot\,\rangle)$ carries a vector cross product $\times$, then $\AA:=\RR\oplus V$ together with the product
 \[(s,u)\cdot(t,v):=(st-\langle u,v\rangle,sv+tu+u\times v),\qquad s,t\in\RR,\ u,v\in V,\]
 is a normed division algebra.
\end{exercise}

\begin{exercise}\label{sextonions}[Sextonions, see \cite{LM-sext}]
\label{ex:sextonions} \curiosity
Let $N\subset\OO^\CC$ be a two-dimensional subspace. Prove that $S:=N^\perp$ is a six-dimensional subalgebra if and only if $N$ is a null-plane. Prove that such $S$ do exist, and that $S\simeq M_2(\CC)\oplus \CC^2$ with multiplication given by
\[(A,v)\cdot(B,w)=(AB,\tr(A)w-Aw+Bv),\qquad A,B\in M_2(\CC),\ v,w\in\CC^2.\]
(One can use the fact that $G_2$ acts transitively on the set of null planes.)
Since all null-planes are conjugate, we get a unique six-dimensional algebra of $\OO^\CC$, up to conjugation. We denote by $\SSS$ and call it the \emph{sextonion algebra}.
\end{exercise}

\begin{exercise}
\label{ex_dim14}
\necessary Show that $G_2^{\mathrm{c}}$ acts simply transitively on the set of basic triples. Conclude that $G_2^{\mathrm{c}}$ is connected and of dimension $14$.
\end{exercise}

\begin{exercise}
\training In Proposition \ref{g2fromsl3}, the action of $\fsl_3$ on itself, $V_3$, and $V_3^\ast$ is canonical. Construct the missing parts of the bracket
\[V_3\times V_3\lra V_3^\ast,\qquad V_3^\ast\times V_3^\ast\lra V_3,\qquad V_3\times V_3^\ast\lra\fsl_3\]
and verify the Jacobi identity.
\end{exercise}

\begin{exercise}
\training Define the action of $\fg_2$ on the $V_7$ of Remark \ref{rmk:V_7-roots}, and show that $\omega$ is invariant and generic.
\end{exercise}

\begin{exercise}
\curiosity Construct directly $\fg_2$ from $\fsl_2\times\fsl_2=\fsl(A)\times\fsl(B)$, as well as the seven-dimensional representation as follows:
\[ \fg_2= \fsl(A)\times\fsl(B)\oplus (A\otimes S^3B), \qquad V_7=B\oplus (A\otimes S^2B).\]
\end{exercise}

\medskip
\begin{sol}[Ex. \ref{ex:sextonions}]
    Let us prove the statement that: if $S$ is a six-dimensional subalgebra, then $N:=S^\perp$ is a null-plane. Notice that if $s\in S$ also $\overline{s}\in S$. Moreover $N\subset V_7$. Let $s,s'\in S$ and $n\in N$. Applying the calculational rules in \cite[Lem.~1.3.2, 1.3.3]{springerveldkamp} we get $0=\langle n,ss'\rangle=\langle \overline{s}n,s'\rangle $, and therefore $S\cdot N\subset N$. Since $S$ is six-dimensional, its orthogonal is two-dimensional, so the right multiplication by any element $n\in N$ cannot be invertible. This shows that $N$ is isotropic for the symmetric form. Let us look at the image $R_n(S)$ of the right multiplication operator $R_n$ by $n\in N$. The right multiplication version of Theorem \ref{kerofleftmul} tells us that $R_n(\OO^\CC)$ is four-dimensional, and since $S$ has codimension two, $R_n(S)$ has at most codimension two in $R_n(\OO^\CC)$; thus $R_n(S)=N$. So for any $n'\in N$, we have $n'=sn$ for a certain $s\in S$. But then $n'n=(sn)n=-q(n)s=0$. This shows that $N$ is a null-plane.

    Consider now $N$ a null-plane. By Corollary \ref{G2G27_hom} we know that $G_2$ acts transitively on the set of null-planes. So we can suppose that $N\subset\OO^\CC$ is the plane spanned by $e_4+\ii e_5$ and $e_6+\ii e_7$. Then 
\[S:=N^\perp=\HH^\CC\oplus N=\mathrm{span}\{1,e_1,e_2,e_3,e_4+\ii e_5,e_6+\ii e_7\}.\]
One can check directly that $S$ is a subalgebra. Moreover, it contains a copy of $\HH^\CC\simeq M_2(\CC)$; the motivated reader can check the multiplication rules in $S$. We remark that even though the sextonions $\SSS$ are a subalgebra of $\OO^\CC$, the restriction of the quadratic form $q$ to $\SSS$ is degenerate (since $N\subset N^\perp$). Hence, $S$ is not itself a composition algebra, and its existence does not contradict Theorem~\ref{hurwitzgeneral}.
\end{sol}

\begin{sol}[Ex. \ref{ex_dim14}] Choose $e_1=$ any vector of unit norm, then $e_2=$ any vector of unit norm and orthogonal to $e_1$. Then let $e_3=e_1e_2$, it has unit norm and is orthogonal to $e_1, e_2$. Finally choose $e_4$ of unit norm, orthogonal to $e_1,e_2,e_3$, and let $e_5=e_4e_1$, $e_6=e_4e_2$, $e_7=e_4e_3$: you get the correct multiplication table.

The vector $e_1\in S^6\subset\Im\OO$ can be chosen arbitrarily, while $e_2\in S^5\subset e_1^\perp$ and $e_4\in S^3\subset\{e_1,e_2,e_1e_2\}^\perp$. Thus $G_2^{\mathrm{c}}$ is diffeomorphic to a bundle over $S^6$ whose fiber is itself a bundle over $S^5$ with fiber $S^3$. In particular it is connected and of dimension $6+5+3=14$.
\end{sol}

\chapter{A few \texorpdfstring{$G_2$}{G2}-varieties}

In this lecture, we primarily focus on the complex group $G_2$. Our main objective is to study various interesting projective varieties acted on by $G_2$, with particular emphasis on highlighting how octonionic geometry and the representation theory of $G_2$ are deeply intertwined with their geometry. We will start by looking at the two homogeneous $G_2$-varieties, the 5-dimensional quadric $\QQ\subset \PP^6$ and the $G_2$-Grassmannian $\GtwoGr(2,7)\subset \PP(\mathfrak{g}_2^{\CC})$.
We will then move on to the study of specific families of linear spaces that are acted on by  $G_2$: the Cayley Grassmannian $\CGr\subset \OGr(3,7)$, its "doubled version" $\DGr\subset \OGr(6,14)$, the Lie Grassmannian $\LieGr(2,7)\subset \OGr(2,7)$ and the horospherical $G_2$-variety $HG\subset \PP(\Im(\OO^{\CC})\oplus \mathfrak{g_2}^{\CC})$. All have special properties that we will discuss for they own sake, but also place in a more general context. 

\section{Homogeneous varieties associated to \texorpdfstring{$G_2$}{G2}}

\subsection*{The six-sphere}

Over $\RR$, the group $G_2^{\mathrm{c}}$ acts transitively on imaginary octonions of norm one. Such octonions  are parametrized by the 6-dimensional sphere $\SSS^6\subset \Im \OO\simeq \RR^7$, therefore $\SSS^6$ is a homogeneous space of $G_2^{\mathrm{c}}$ (Corollary~\ref{g2trans}). As it turns out:
\[ \SSS^6\simeq G_2^{\mathrm{c}}/\SU_3.\]
To prove this isomorphism, it is sufficient to compute the isotropy subgroup of an arbitrary point $x\in \SSS^6$. Since $G_2^{\mathrm{c}}\simeq \Aut(\OO)$, each $\phi\in G_2^{\mathrm{c}}$ preserves both the real and imaginary parts of the octonionic product. If we then take an element $\phi\in \Stab_{G_2^{\mathrm{c}}}(x)$, $\phi$ must also fix the entire linear space $V_{x}:=x^{\perp}\cap \Im(\OO)$. This also tells us that $\phi$ is uniquely determined by its restriction to $V_x$ which, as $G_2^{\mathrm{c}}\subset \SO_7(\RR)$, must lie in $\SO_6(\RR)$. 
In other words, the restriction to $V_x$ defines an injective morphism $\rho_x \colon\Stab_{G_2^{\mathrm{c}}}(x)\to \SO_6(\RR)$. Since $x^2=-1$, the linear space $V_{x}\simeq \RR^6$ admits a complex structure defined by left-multiplication by $x$. By definition, given $\phi\in \Stab_{G_2^{\mathrm{c}}}(x)$ and $v\in V_x$, $\phi(x\cdot v)=x\cdot \phi(v)$, hence the image of $\rho_x$ preserves the complex structure, that is, lies in $\SL_3(\CC)$.
Since, moreover, $\phi$ preserves both (nondegenerate) bilinear forms $\langle\,\cdot\,,\cdot\,\rangle |_{V_x}$ and $(x\intprod \omega)|_{V_x}$, it must also preserve the Hermitian form $h:= \langle\,\cdot\,,\cdot\,\rangle|_{V_{x}} + \ii(x\intprod \omega)|_{V_{x}}$. Hence, the image of $\rho_x$ is contained in $\SU_3=\SO_6(\RR)\cap\SL_3(\CC)$. Counting dimensions, this allows us to conclude that $\Stab_{G_2^{\mathrm{c}}}(x)\simeq\SU_3$.

Since $G_2^{\mathrm{c}}$ acts linearly on $\Im(\OO)$, we can extend its action to $\PP(\Im(\OO))\simeq \PP_{\RR}^6$, the projective space of lines in $\Im(\OO)$ through the origin. This action is still transitive and, as a result: $\PP_{\RR}^6\simeq G_2^{\mathrm{c}}/\Stab_{G_2^{\mathrm{c}}}(x)$ for an arbitrary point $x\in \PP_{\RR}^6$. Let us consider the projection $\pi\colon \SSS^6\to \PP_{\RR}^6$ (mapping a point $x\in \SSS^6$ to the line spanned by $x$ and the origin) and take an arbitrary point $y\in \PP_{\RR}^6$. The preimage $\pi^{-1}(y)$ consists of two antipodal points $\pm x$ in $\SSS^6$; therefore, $\Stab_{G_2^{\mathrm{c}}}(y)=\{\pm\phi\mid \phi\in \Stab_{G_2^{\mathrm{c}}}(x)\}$. Accordingly $\Stab_{G_2^{\mathrm{c}}}(y)=\ZZ_2\ltimes \SU_3$ and \mbox{$\PP_{\RR}^6\simeq G_2^{\mathrm{c}}/(\ZZ_2\ltimes \SU_3)$}.

\subsection*{The quadric and the \texorpdfstring{$G_2$}{G2}-Grassmannian}

Let us discuss what this result becomes over the
complex octonions $\OO^{\CC}=\OO\otimes_{\RR} \CC$. The complexification of the quadratic form $q$ on $\OO$ to a quadratic form on $\OO^\CC$, as well as its restriction to $V_7=\Im\OO^\CC$, shall still be denoted by $q$. Moreover,  all linear spaces $\PP^{n}$ will implicitly be complex projective spaces.
 
As already mentioned, the first noticeable difference when we work over $\CC$, rather than $\RR$, is that there exist non-zero imaginary octonions whose ``norm" is zero. The quadratic form $q$ indeed defines a quadric hypersurface $\QQ^5$ in $\PP(V_7)=\PP(\Im(\OO^\CC))$, though without real points. In fact, there are entire linear spaces of octonions all having zero norm, as we already saw in Theorem~\ref{kerofleftmul}. For what concerns the action of $G_2$ on $\PP(\OO^\CC)$, the complement of the quadric is a $G_2$-orbit, and one can prove that $\PP(\Im(\OO^{\CC}))\setminus\QQ^5\simeq G_2/(\ZZ_2\ltimes\SL_3)$ (\necessary Exercise \ref{ex:SL3}). This implies the existence of an embedding $\SL_3 \hookrightarrow G_2$. We already met such an embedding in the last section (cf. Proposition \ref{g2fromsl3}), where we saw that the subalgebra $\mathfrak{g}_0=\mathfrak{h}\oplus_{\lambda\in \Delta_{l}}\mathfrak{g}_{\lambda}$, generated by the set $\Delta_l$ of long roots, is isomorphic to $\fsl_3$.

Let us recall some further features of $\fg_2$ that will be needed later. We have two fundamental representations of $\fg_2$: the adjoint representation and the seven-dimensional standard representation. The highest weight of the adjoint representation is the highest root \mbox{$\omega_2=3\alpha_1+2\alpha_2$}. The fundamental seven-dimensional representation is generated by the short roots; its highest weight is $\omega_1=2\alpha_1+\alpha_2$, that we also denote by $\alpha$; the weights are $0$ and the short roots (see Remark \ref{rmk:V_7-roots}).

As previously stated, when we pass to the complex octonions, there appear two-dimensional  linear subspaces of $\Im(\OO^{\CC})$ on which all octonionic products are zero. Such linear spaces are called \emph {null-planes} (Definition~\ref{def:null-plane}).
They are parametrized by a closed subvariety $\GtwoGr(2,7)\subset \Gr(2,7)$, on which $G_2$ acts transitively.
As it turns out, it is possible to obtain $\GtwoGr(2,7)$ as the zero locus of a global section of a homogeneous vector bundle on the Grassmannian $\Gr(2,7)$. 
Recall that on $\Gr(2,7)$ there is a \emph{tautological vector bundle}
$\cU$ of rank two (this is simply the vector bundle whose fiber over a point $[V]$ is the space $V$ itself), which fits into an ``Euler type" short exact sequence:
\[ 0\lra \cU\lra \cO_{\Gr}\otimes V_7\lra \cQ\lra 0.\]
Here $\cO_{\Gr}$ is the trivial bundle on $\Gr(2,7)$ and $\cQ$ is the tautological quotient bundle. We also let $\cO_{\Gr}(1)=\det(\cQ)$,  $\cO_{\Gr}(-1)=\cO_{\Gr}(1)^*$, and for any vector bundle $\cV$ and $n\in \ZZ$, we denote $\cV(n):=\cV\otimes \cO_{\Gr}(1)^{\otimes n}$.

The invariant three-form $\omega$ defines a global section of $\cQ^*(1)$ as follows. For any pair of linearly independent vectors $x,y$ in $V_7$ spanning $U :=\spann{x, y}$, the linear form $\omega(x,y,\cdot\,)$
clearly annihilates $U$, hence defines an element in $(V_7/U)^*$. Thus $\omega$ defines a section of $(\extp^2\cU)^\ast\otimes\cQ^\ast= \cQ^\ast(1)$.

\begin{prop}\label{prop:zero-locus} $\GtwoGr(2,7)\subset \Gr(2,7)$ is the zero-locus of $\omega$ regarded as a section of the rank-five 
vector bundle $\cQ^*(1)$.
\end{prop} 

\begin{proof} For a pair of linearly independent vectors $x, y$ in $\Im(\OO^{\CC})$, the linear form $\omega(x,y, \cdot\,)$ on  $\Im(\OO^{\CC})$ is identically zero if and only if the octonionic product $xy$ belongs to $\CC\, 1=\Re(\OO^\CC)$, see \eqref{eq:inv-form}. 
But if $xy=t1$ with $t\in \CC^*$, then  since $x \in \Im(\OO^{\CC})$ we deduce that $y=\frac{t}{|x|^2} \overline{x}= \frac{-t}{|x|^2} x$,
contradicting the linear independence of $x$ and $y$.\end{proof}

\begin{cor}
$\GtwoGr(2,7)$ is smooth of dimension five, homogeneous under the action of $G_2$. 
\label{G2G27_hom}
\end{cor}
\begin{proof}
The zero locus $\GtwoGr(2,7)$ of $\omega\in H^0(G(2,7),\cQ^{*}(1))$ is stable under the action of  $\Stab_{\SL_7}(\omega)$, which is $G_2$. By generic smoothness, 
the genericity of $\omega$ implies that its zero locus is smooth of the expected codimension $5=\rank(\cQ^{*}(1))$.
By Lemma \ref{lem_stable_contains_hom}, there exists a homogeneous space $G_2/P$, with $P$ parabolic, that is contained in $\GtwoGr(2,7)$. Now, parabolic subgroups of $G_2$ are classified, and up to conjugation their dimension can only be equal to $9$ or $8$. This shows that $G_2/P$ in fact must coincide with $\GtwoGr(2,7)$, which is thus homogeneous.
\end{proof}

\begin{definition}
    Given a semisimple Lie group $G$, a \emph{parabolic subgroup} of $G$ is a closed subgroup that contains a  Borel subgroup of $G$ (a maximal, closed connected solvable subgroup). A \emph{homogeneous space} will be in these lectures a quotient $G/P$ of a semisimple group by a proper parabolic subgroup.
\end{definition}

\begin{remark}
By \cite[Theorem 7.5]{Otta}, parabolic subgroups of a semisimple  group $G$ are those subgroups $P$ whose quotient $G/P$ is projective. Then homogeneous spaces are rational projective varieties.
    In the literature, our \emph{homogeneous spaces} are often referred to as \emph{generalized flag varieties} (or \emph{generalized Grassmannians} when the parabolic subgroup is maximal),  while a homogeneous space is simply a variety with a transitive group action.
\end{remark}

\begin{fact}
    Since parabolic subgroups are classified, the same is true for homogeneous spaces. Moreover, homogeneous vector bundles (vector bundles with a compatible action of $G$) on homogeneous spaces $G/P$ are in $1:1$ correspondence with representations of $P$. In particular, irreducible homogeneous vector bundles are classified by their highest weight. See \cite{Otta} for details. 
\end{fact}

\begin{lemma}
\label{lem_stable_contains_hom}
Let $G$ be a semisimple group and $V$ be a $G$-representation with no trivial factor. Then every $G$-stable projective variety $X\subset \PP(V)$ contains a homogeneous space.
\end{lemma}
\begin{proof}
    Let $B$ be a Borel subgroup of $G$. By Borel's fixed point Theorem, there exists a point $x\in X$ which is fixed by $B$. Then the stabilizer $P$ of $x$ in $G$ is a closed subgroup containing $B$, that is, a parabolic subgroup. Moreover, the point $x$ cannot be fixed by the whole 
    $G$, since otherwise it would define a one-dimensional representation of $G$ in $V$, i.e., a trivial factor. Hence, the $G$-orbit $G\cdot x$ is closed and projective, and of positive dimension, thus a homogeneous space.
\end{proof}

\begin{lemma}
    Let $V$ be a non-trivial irreducible representation of a semisimple group $G$. Then $\PP(V)$ contains a unique homogeneous space.
\end{lemma}
\begin{proof}
    From the proof of the previous lemma, the result follows if we can show that $\PP(V)$ contains a unique $B$-fixed point. However, this is equivalent to the fact that, since $V$ is irreducible, it has a unique highest weight vector up to scalar (we let the reader fill up the details with the help of any representation theory textbook, for instance \cite{FH}).
\end{proof}

We keep focusing our attention on the projective space $\PP(\extp^2 V_7)$; in particular, we will establish what the other $G_2$ homogeneous subvarieties of this space are. Recall that the vector space $\extp^2 V_7$ decomposes as the direct sum of two irreducible $G_2$-representations (see \cite[\S22]{FH} or the proof of Proposition \ref{prop:wedge3}):
\[\extp^2 V_7= V_7\oplus \fg_2\]
From this decomposition we deduce that there are two closed $G_2$-orbits in $\PP(\extp^2 V_7)$, one in $\PP(V_7)$ and the other in $\PP(\fg_2)$. From the general theory of rational homogeneous varieties, since $V_7=V_{\omega_1}$ and $\fg_2=V_{\omega_2}$  these two orbits must be isomorphic with  $G_2/P_1$ and $G_2/P_2$, where $P_1, P_2$ are the two maximal parabolic subgroups associated with the simple roots $\alpha_1$ and $\alpha_2$ (see \cite{Otta} for a complete treatment of homogeneous varieties). Moreover we know that their dimensions are both equal to $5$. This immediately implies  that $G_2/P_1=\QQ^5$ and   $G_2/P_2=\GtwoGr(2,7)$.



Since $V_7$ embeds in $\extp^2 V_7$ via $v\mapsto v\intprod\omega$. 
and any $[v]\in\QQ^5$ is in the $G_2$-orbit of $[x_\alpha]$, 
a direct computation shows that 
\[ [v]\in \QQ^5\quad \implies  \quad \,\,\,\,\, \rank(v\intprod\omega)=4. \]
In particular the image of  $\QQ^5$ does not meet the Grassmannian $\Gr(2,7)$. 

Since $B=P_1\cap P_2$ is a Borel subgroup of $G_2$, there is an incidence diagram: 
\begin{equation}\label{eq:diag-A}
\xymatrix{ & G_2/B\ar[dr]\ar[dl] & \\
\PP(V_7)\supset\QQ^5=G_2/P_1 & & G_2/P_2=\GtwoGr(2,7)\subset \PP(\fg_2). 
}
\end{equation} 
where the fibers of both projections are isomorphic to $\PP^1$.
From a geometric perspective, the diagram can be interpreted as follows.
Each null-plane $\spann{u, v}\subset V_7$ is fixed by a subgroup of $G_2$ isomorphic to $P_2$, and the stabilizer (in $P_2$) of an arbitrary point in this null-plane is isomorphic to a Borel subgroup $B\subset P_2\subset G_2$. In other words, we can identify $P_2/B$ with the projective line $\PP(\spann{ u, v})$, the full flag manifold $G_2/B$ with the incidence variety
\[ G_2/B\simeq \left\{(u,l)\in \QQ^5\times \GtwoGr(2,7)\mid u\subset l\right\},\]
and the fibration $G_2/B\to G_2/P_2$ with the tautological $\PP^1$-fibration over $\GtwoGr(2,7)$. 
The fiber over a point $u\in \QQ^5$ is the family of null-planes through $u$; since the latter must be isomorphic to $P_1/B\simeq \PP^1$, we deduce that each point of the quadric belongs to a family of null-planes parametrized by $\PP^1$. 

\begin{remark}
Since any point $u\in\QQ^5$ defines a line in $G_2/P_2\subset \PP(\fg_2)$, we get an embedding 
$$\QQ_5\hookrightarrow \Gr(2,\fg_2).$$
Pulling-back the tautological bundle on the Grassmannian, we deduce a special $G_2$-homogeneous vector bundle on $\QQ^5$, called the \emph{Cayley bundle} \cite{ott-cayley}. 
\end{remark}

Diagram \eqref{eq:diag-A} is  a subdiagram of the usual incidence diagram for lines and planes,
 \[\xymatrix{ & \Fl(1,2,7)\ar[dr]^{\pi_2}\ar[dl]_{\pi_1} & \\
\PP(V_7)=\PP^6 & & \Gr(2,7)\subset\PP(\extp^2 V_7)\simeq\PP(\fso_7)}\]
where $\Fl(1,2,7)$ denotes the variety of $(1,2)$-flags in $V_7$. The fibers of $\pi_1$ are isomorphic to $\PP^5$ and the fibers of $\pi_2$ are isomorphic to $\PP^1$.

\begin{remark}
    The homogeneous varieties $G_2/P_1$, $G_2/P_2$ and $G_2/B$ discussed above may also be expressed as homogeneous spaces (in the group-theoretic sense) of the compact group $G_2^{\mathrm{c}}$. For any complex semisimple Lie group with compact real form $G^{\mathrm{c}}$, the generalized flag manifolds (also called \emph{Kähler C-spaces}) are described equivalently as
    \begin{enumerate}
    \item quotients $G/P$ by a parabolic subgroup,
    \item quotients $G^{\mathrm{c}}/K$ by the centralizer $K=G^{\mathrm{c}}\cap P$ of a torus,
    \end{enumerate}
    see for example \cite{Arva,Chrys}. Under this correspondence, the full flag manifold is given by $G_2/B\simeq G_2^{\mathrm{c}}/T^2$, with $T^2\subset G_2^{\mathrm{c}}$ a maximal torus. The flag manifolds from viewpoint (2) are explicitly known -- see \cite[\S8.H]{Besse}
    for the classical groups. For the flag manifolds of $G_2$, we have
    \[\QQ^5=G_2/P_1=G_2^{\mathrm{c}}/\U_2^{-}\]
    with $\U_2^{-}$ embedded by $\U_2\subset\SO_4(\RR)\subset G_2^{\mathrm{c}}$, while
    \[\GtwoGr(2,7)=G_2^{\ad}=G_2/P_2=G_2^{\mathrm{c}}/\U_2^{+}\]
    with $\U_2^{+}$ embedded by $\U_2\subset\SU_3\subset G_2^{\mathrm{c}}$, cf.~\cite{Chrys}. Moreover the quadric, defined only in terms of $q$, is naturally a homogeneous space of $\SO_7$, indeed also a generalized flag manifold; it can be written as
    \[\QQ^5\simeq\SO_7(\RR)/(\SO_2(\RR)\times\SO_5(\RR)),\]
    which is also its presentation as a compact symmetric space \cite{helgason}.
\end{remark}

\subsection*{The orthogonal Grassmannian}

\begin{definition}
The \emph{orthogonal Grassmannians} of $V_7$ are defined as
$$    \OGr(k,V_7)=\OGr(k,7):=\left\{\,q\text{-isotropic} k\text{-planes in }V_7\,\right\}=\left\{\,[P]\in \Gr(k,7)\mid q|_P\equiv 0\,\right\}.$$
\end{definition}

They are non-empty only for $k\le3$, and  $\OGr(1,7)\simeq\QQ^5$. The action of $\SO_7$ preserving the quadratic form $q$ makes any $\OGr(k,7)$ a homogeneous space.
Also notice that $\OGr(2,7)$ is isomorphic to the $\SO_7$-adjoint variety (see Definition~\ref{def_adjoint}), as it is a closed orbit inside $\PP(\fso_7)$:
\[\OGr(2,7)\simeq \SO_7^{\ad}\subset \PP(\extp^2 V_7)\simeq\PP( \mathfrak{so}_7).\]

\begin{remark}
As we will see, $\OGr(k,7)$ is not in general a homogeneous space of $G_2$. However, it is a generalized flag manifold of $\SO_7$, and we can write $\OGr(k,7)=\SO_7/P_k$, $1\leq k\leq 3$, where $P_k$ is the parabolic corresponding to the $k$-th simple root of $\SO_7$.
We can also describe it  as a homogeneous space of the compact group $\SO_7(\RR)$. The Dynkin type of the isotropy group is determined by deleting the $k$-th root from the Dynkin diagram of $\SO_7$; this leads (cf.~\cite[Ex.~8.115]{Besse}) to
\[\OGr(k,7)\simeq\SO_7(\RR)/(\U(k)\times\SO_{7-2k}(\RR)),\qquad 1\leq k\leq 3.\]

\end{remark}

\begin{definition}
For any $[x]\in\QQ^5$, let $K_x:=\ker(L_x|_{\Im\OO^\CC})$be  the kernel of the left-multiplica\-tion by $x$ on $V_7=\Im\OO^\CC$.
By Lemma~\ref{zerodivnorm}, $x$ is a zero divisor, and by Corollary~\ref{kerofleftmulim},   $K_x$ is a three-dimensional isotropic subspace containing $x$, namely
\[K_x=x\OO^\CC\cap \Im(\OO^{\CC}).\]
\end{definition}
In particular, $\PP(K_x)\subset\QQ^5$. We can also write $K_x$ in terms of the associative three-form (Exercise \ref{ex_Kxker}):
\[K_x=\ker(x\intprod\omega).\]

\begin{lemma}
\label{nullplanes}
The two-planes $N\subset K_x$  through $x$ are precisely the null planes containing $x$.
\end{lemma}
\begin{proof}
If $N$ is a null-plane through $x$, then  $N\subset\Im\OO^\CC$ by Lemma~\ref{lem:null-isotropic}, so $N\subset K_x$ by definition.

Conversely, let $N=\spann{x,y}$ for some $y\in K_x$. Then, by definition, $xy=0$. On the other hand $y^2=-q(y)=0$ since $y$ is imaginary and $K_x$ is isotropic, and the same goes for $x$. By bilinearity, all octonionic products of elements of $N$ vanish, so $N$ is a null-plane.
\end{proof}

Projectivizing, these null-planes give lines in the quadric which we will also call \emph{$G_2$-lines}. It is crucial to note that not all lines in $\QQ^5$ are $G_2$-lines (\necessary Exercise \ref{ex_OG}); in fact, the union of all lines in $\QQ^5$ is $\OGr(2,7)$.

Information about the family of lines in $\QQ^5$ can also be deduced from the Dynkin diagram of $G_2$:
$$ \resizebox{0.1\hsize}{!}{
\dynkin[labels={ \alpha_2,\alpha_1 }]G2} $$
Following \cite[Section 4]{LM03} (see also Remark \ref{rmk_lines_short_roots}), since $\alpha_1$ is a short root, the family of lines in $\QQ^5\simeq G_2/P_1$ is irreducible with two $G_2$-orbits, the closed one being $G_2/P_2\simeq \GtwoGr(2,7)$ whilst the open one must be the complement  $\OGr(2,7)\setminus \GtwoGr(2,7)$. We will see later on  how to identify $\GtwoGr(2,V_7)$ with a zero locus of a section of a vector bundle on $\OGr(2,7)$ (Remark \ref{rem_G2G_in_OG}).

From Exercise~\ref{ex_OG} we will learn that the image of the embedding
\[
\QQ^5\lra \OGr(3,7), \qquad [x]\longmapsto [K_x]
\]
is a family of planes in $\QQ^5$ which is isomorphic to $\QQ^5$ itself. From the definition of this family of isotropic planes $\QQ^5\subset \OGr(3,7)$, we also deduce the following.

\begin{lemma}\label{lem:dim-null-planes}
The maximal dimension of a linear subspace of $\OO^{\CC}$ on which all octonionic products vanish is two; any such maximal linear subspace is a null-plane.
\end{lemma}
\begin{proof}
A linear subspace $N\subset\OO^{\CC}$ on which all octonionic products are zero must necessarily be contained in $\Im(\OO^{\CC})$ and be isotropic with respect to the quadratic form (Lemma~\ref{lem:null-isotropic}). Moreover, it must satisfy $N\subset K_x$ for any $x\in N$, thus $\dim N\leq \dim K_x=3$. But we will prove in Exercise~\ref{ex_OG} that the general point in $\Gr(2, K_x)$ is not a null-plane; as a consequence $\dim(N)\le 2$ with equality holding if and only if $N$ is a null-plane.
\end{proof}

\subsection*{Lines in the \texorpdfstring{$G_2$}{G2}-varieties}

\begin{definition}
For a projective variety $X\subset \PP^n$, we denote by 
\[
\Hil_{\PP^1}(X)=\left\{\PP^1\subset \PP^n \text{ linear} \,|\, \PP^1\subset X\right\}
\]
the \emph{Fano variety of lines} in $X$ (see also Appendix \ref{sec_tits}).
\end{definition}

Inside the Grassmannian $X=\Gr(2,7)\subset\PP(\extp^2V_7)$, lines are parametrized by $\Hil_{\PP^1}(\Gr(2,7))\simeq\Fl(1,3,7)$, the variety  of partial flags $(L_1\subset L_3)$, where $L_k\subset V_7$ is a subspace with $\dim L_k=k$ (see Example~\ref{lines_in_grassmannian}); any line in $\Gr(2,7)$ is of the form
\[ \left\{[P]\in\Gr(2,7)\,|\, L_1\subset P\subset L_3\right\}\simeq \PP(L_1\wedge L_3)\subset\PP(\Lambda^2V_7)\]
for some partial flag $(L_1,L_3)$. Let us see what happens for the lines in the orthogonal and $G_2$-Grass\-mann\-ians.

Some terminology: we call a $\PP^1$ embedded smoothly in $\PP^3$ as a subvariety of degree three a \emph{twisted cubic curve}. Moreover, by \emph{pencil} we will mean a family of objects of some type, which are (algebraically) parametrized by a $\PP^1$; e.g.~a pencil of lines, a pencil of planes, etc.

\begin{prop}
\label{prop_lines_G2}
Lines  in the $G_2$-Grassmannian $\GtwoGr(2,7) \subset \Gr(2,7)$ are pa\-ra\-me\-tri\-zed by $$\Hil_{\PP^1}(\GtwoGr(2,7))\simeq\QQ^5.$$ 
Lines in $\GtwoGr(2,7)$ passing through a fixed null-plane $[N]$ are pa\-ra\-me\-tri\-zed by a twisted cubic curve $\PP(N)=\PP^1 \hookrightarrow \PP(T_{[N]}\GtwoGr(2,7))$.
\end{prop}


\begin{proof}
Let us consider a point $(L_1=[x],L_3)\in \Fl(1,3,7)$. By Lemma~\ref{nullplanes}, the line $\PP(x\wedge L_3)\simeq \PP^1$ in $\Gr(2,7)$ gives rise to a pencil of null-planes if and only if $[x]\in \QQ^5$ and $L_3=K_x$. Hence lines of $\GtwoGr(2,7)$ are parametrized by the map
\[\QQ^5\longrightarrow\Hil_{\PP^1}(\GtwoGr(2,7)),\qquad [x]\mapsto\PP(x\wedge K_x)\subset\PP(\extp^2V_7).\]
In the following, we identify the line $\PP(x\wedge K_x)$ with the partial flag $([x],K_x)$.


Let us consider now the universal line $\cL\subset \Hil_{\PP^1}(\GtwoGr(2,7))\times \GtwoGr(2,7)$; this is the subvariety whose fiber over a point $L\in \Hil_{\PP^1}(\GtwoGr(2,7))$ is the corresponding pencil of null-planes. We have a diagram
 \begin{equation}\label{eq:diag-A'}
       \xymatrix{ & \cL\ar[dr]^{p_2}\ar[dl]_{p_1} & \\
\QQ^5 & & \GtwoGr(2,7)}
\end{equation}
where $p_1$ is the composition of the projection onto the first factor with the isomorphism above, and $p_2$ is the projection onto the second factor.

Given a null-plane $[N]=\spann{u,v}\in \GtwoGr(2,7)$, its pre-image by $p_2$ 
is the pencil of partial flags
\[ p_2^{-1}([N])=\{ ([t_0u+t_1v], K_{t_0u+t_1v}) \,|\, [t_0:t_1]\in \PP^1\}.\]
As $p_1(p_2^{-1}([N]))=\PP(N)$, we get 
an embedding 
\[\tau_{[N]}:\quad\PP(N) \to \PP(T_{[N]}\GtwoGr(2,7))\subset\PP(\extp^2V_7)\]
which sends a one-dimensional subspace of $N$ to the corresponding tangent direction at $[N]$. 


Let us show that $\tau_{[N]}$ is a degree three Veronese embedding. To ease our task, we choose to work with $N:=\spann{x_{\alpha},x_{-\beta}}$ -- the general case can be reduced to this one by the transitivity of the $G_2$ action on $\GtwoGr(2,7)$. 
To begin, consider the pencil of two-forms $\{[x\intprod \omega]\,|\, x\in N\}$. Since $x$ is a null vector, these forms all have rank 4 and define an embedding 
\[ \PP (N)\simeq \PP^1 \lra \PP(\extp^2 V_7^*), \qquad [x]\longmapsto x\intprod \omega. \]
Since $N$ is a null-plane, the image  lies in $\PP(\extp^2 \ann(N))$. For our particular choice of $N$ we get:
\begin{align*}
(t_0 x_{\alpha}+ t_1 x_{-\beta})\intprod\omega &= t_0 (x^0\wedge x^{-\alpha} + x^{\beta}\wedge x^{\gamma}) + t_1(x^0\wedge x^{\beta}+ x^{-\alpha}\wedge x^{-\gamma}),\\
\ann(N)&=\spann{x^0,x^{-\alpha},x^{\beta},x^{\gamma},x^{-\gamma}}.
\end{align*}
Let us denote by $\Pf(4,V_7)\subset \PP(\extp^2 V_7^*)$ the locus of two-forms of rank at most $4$. For any form $\alpha$ of rank $4$, its square $\alpha\wedge\alpha$ is decomposable and represents the four-dimensional support of $\alpha$. We thus have a rational map
\[\Pf(4,V_7)\dashrightarrow \Gr(4,V_7^*)\simeq \Gr(3,V_7), \qquad  \alpha\mapsto\alpha\wedge\alpha\]
(rational since only defined on the open subset of forms of rank $4$). This is a quadratic map whose restriction to the pencil $\{[x\intprod \omega]\,|\, x\in N\}$ is a degree two Veronese embedding
\[\PP(N)\xhookrightarrow{\nu_2} \Gr(4,\ann(N))\simeq \PP(\ann(N)^\ast).\]
For our choice of $N$, we get:
 \begin{equation}\label{eq:deg2-ver}
 \nu_2([t_0x_{\alpha}+t_1x_{-\beta}])= [t_0^2x_{-\gamma}+ 2t_0t_1 x_0+t_1^2 x_{\gamma}].
 \end{equation}
The pencil of lines through $[N]$ therefore embeds into $\PP(T_{[N]}\Gr(2,7))\subset\PP(\extp^2V_7)$ 
via:
\begin{align*}
 \PP^1\rightarrow \PP(T_{[N]} \Gr(2,7)),\quad [t_0:t_1]\mapsto& [(t_0x_{\alpha}+ t_1 x_{-\beta})\wedge \nu_2(t_0x_{\alpha}+ t_1 x_{-\beta})]\\ =\,&[t_0^3(x_{\alpha}\wedge x_{-\gamma}) +2{t_0}^2t_1(x_{\alpha}\wedge x_0)+ 2t_0t_1^2 (x_0\wedge x_{-\beta})+ t_1^3(x_{-\beta}\wedge x_{\gamma})]
 \end{align*}
which is indeed  a twisted cubic curve.
\end{proof}

\begin{remark}
To determine that the variety $\Sigma_N\subset \PP(T_{[N]}\GtwoGr(2,7))$ of lines in $\GtwoGr(2,7)$ through a null-plane $N$ is a rational curve, one could also simply look at the Dynkin diagram of $G_2$. The semisimple part of $P_2$ is simply $\SL_2$ (the Dynkin diagram obtained by erasing the node corresponding to $P_2$ is the Dynkin diagram of $\SL_2$), and as $\alpha_2$ is a long root, $\Sigma_{N}$ is isomorphic to $SL_2/B\simeq \PP^1$ for $B\subset \SL_2$ a Borel subgroup. That we get a twisted cubic reflects the triple bond in the Dynkin diagram of $G_2$ (see \cite{LM03} for further details, or Appendix \ref{sec_tits}).
\end{remark}

\begin{remark}\label{vmrt}
The image of $\tau_{[N]}$ in  $\PP(T_{[N]}\GtwoGr(2,7))$ is called the \emph{variety of minimal rational tangents} (VMRT) at the point $[N]$. In the general study of Fano manifolds, VMRT's have been a tremendously useful tool, as they turn out to give extremely strong information on the global geometry on the variety. This is specially true for homogenous spaces, which are essentially characterized by their VMRT; a beautiful consequence is that homogeneous spaces of Picard number one admit no non-trivial K\"ahler deformation. Actually, there is one single exception to this statement: the orthogonal Grassmannian $\OGr(2,7)$! We refer to \cite{hwang-mok2, hwang-contact, hwang-icm} for further details on this topic.
\end{remark}

Let us come back to the universal line $\cL\subset \QQ^5\times \GtwoGr(2,V_7)$. The previous discussion shows that $\cL$ is a conic bundle over the $G_2$-Grassmannian.
Indeed, the pencil of $G_2$-lines through a given null-plane $N$ is nothing but the pencil of isotropic planes in $\QQ^5\subset\OGr(3,7)$ containing $N$; this is isomorphic to a smooth conic on $\QQ^5$, i.e., the image of the quadratic embedding $\mu_2: \PP(N)\to \Gr(3,V_7), \ x\mapsto K_x$.
Observe now that the planes in $\QQ^5\subset \OGr(3,7)$ containing $\PP(N)$ are contained in the linear space $\PP(N^{\perp})$; each one of them is therefore parametrized by a point in $\PP(N^{\perp}/N)\simeq \PP^2$. 
Taking therefore a general 3-dimensional subspace $U\subset N^{\perp}$, disjoint from $N$ we get $\PP(N^{\perp}/N)\simeq \PP (U)$; the set of isotropic planes containing $N$ therefore also identifies with the smooth conic $\QQ^5\cap \PP (U)$, i.e., the image of a degree-2 Veronese embedding $\PP(N)\to \PP (U)$. 
Choosing $N=\spann{x_{\alpha},x_{\beta}}$ as in the previous exercise, we get that, up to projective equivalence, $\PP U=\PP(\spann{x_0,x_{\gamma}, x_{-\gamma}})$ and the conic in $\PP U$ is therefore nothing but the image of the morphism $\nu_2$ in Equation~\eqref{eq:deg2-ver}.

The families of linear spaces $N, \ N^{\perp}, \ N^{\perp}/N$, for $[N]$ varying in $\GtwoGr(2,7)$, define vector bundles $\cU, \ \cU^{\perp}, \cU^{\perp}/\cU$ over the $G_2$-Grassmannian. The bundle 
$\cU$ is simply the restriction of the rank two tautological bundle on $\Gr(2,7)$, $\cU^{\perp}$ is the pullback of the rank 5 tautological bundle on $\Gr(5,V_7)$ under the isomorphism $\Gr(2,V_7)\simeq \Gr(2,V_7^*)\simeq \Gr(5,V_7)$ induced by $q$. Finally, the composition
\[\cU\to \cO_{\Gr}\otimes V_7\to (\cO_{\Gr}\otimes V_7)/\cU^{\perp}\]
is zero on $\OGr(2,7)$ (by the very definition of $\OGr(2,7)$), therefore $\cU\hookrightarrow \cU^{\perp}$. The quotient of this inclusion defines a rank-3 vector bundle $\cQ$ (we will adopt the same notation for the restriction to $\GtwoGr(2,7)$).
We therefore have a $\PP^2$-bundle $ \PP(\cQ)\xrightarrow{\pi} \GtwoGr(2,7) $, and the quadratic form $q$ defines a global section of $\pi^* S^2\cQ^*$. The previous proof  shows that $\cL$ is the zero locus of this global section. 

In fact it is easy to check that $\cL$ is homogeneous under $G_2$, and conclude that 
the incidence diagram \eqref{eq:diag-A'}  coincides with \eqref{eq:diag-A}. It contains a lot of geometric information about $\fg_2$ and $G_2$. 
\begin{enumerate}
    \item $\Im\OO^{\CC}$ and $\fg_2$ are the two fundamental representations (from which all the others can be constructed).  
    \item $\QQ^5$ and $\GtwoGr(2,7)$ are the two generalized Grassmannians: the two minimal orbits  in the (projectivized) fundamental representations. They have the same dimension and Betti numbers as $\PP^5$. In particular, the Picard group is $\ZZ H$, but with $H^5=2$ for  $\QQ^5$ and $H^5=18$ for $\GtwoGr(2,7)$. 
    \item It is possible to recover $G_2$ from $\GtwoGr(2,7)$ in the sense that $G_2=\Aut(\GtwoGr(2,7))$ (see \cite[page 118]{Tits},  or \cite[Theorem 1]{Dema}).

\end{enumerate}

\begin{remark}\label{Fano}
$\QQ^5$ and $\GtwoGr(2,7)$ are both Fano varieties, of index $5$ and $3$ respectively (the index of a Fano manifold $X$ is the largest integer $m$ such that the canonical divisor $K_X=-mH$ for some divisor $H$). By adjunction, a codimension two linear section of $\GtwoGr(2,7)$ is a Fano threefold of index one, and a codimension three  linear section is a surface with trivial canonical bundle, in fact a K3 surface.
(The fact that they easily provide families of Fano manifolds or K3 surfaces is yet another reason to be interested in homogeneous spaces!) 

It was proved by Mukai \cite{Mukai} that any prime Fano threefold of degree $18$ is a section of the $G_2$-Grassmannian. The same statement holds for a \emph{general} polarized K3 surface of degree $18$ (a polarized K3 surface consists of a pair $(S,L)$ where $S$ is a smooth simply connected surface with trivial canonical bundle and $L$ is an ample line bundle of degree $d=L^2$). 
   \end{remark}

\section{The Cayley Grassmannian}\label{sec_CG}

Inside the Grassmannian $\Gr(4,V_7),$
The \textit{Cayley Grassmannian} $\CGr$ is defined as the closed subvariety parameterizing 4-planes that are isotropic with respect to $\omega$:  $[V]\in \CGr$ if and only if $\omega|_{V}\equiv 0$. Recall that the index of a Fano variety was defined in Remark \ref{Fano}.



\begin{lemma}
$\CGr$ is a smooth $8$-dimensional Fano variety of index $4$. 
\end{lemma}

\begin{proof}
Denote by $\mathcal{T}$ the rank 4 tautological bundle on $\Gr(4,V_7)$. The invariant three-form 
$\omega$ defines a global section $s_{\omega}$ of $\extp^3 \mathcal{T}^*$. This is a rank 4 globally generated vector bundle and, as $s_{\omega}$ is generic, its zero locus is smooth of codimension $4$; so  $\CGr$ is smooth of dimension $8$. The generality assumption on $\omega$ also ensures that $\CGr$ has Picard rank one (see \cite{Ein}, Thm 2.2), and its index can be computed by the \emph{adjunction formula} for zero loci of sections of vector bundles: since $K_{\Gr(4,7)}\simeq \cO_{\Gr(4,7)}(-7)$ and $\det(\extp ^3 \mathcal{T})=\cO_{\Gr(4,7)}(3)$, we get $$K_{\CGr}\simeq K_{\Gr(4,7)}\otimes\det(\extp ^3 \mathcal{T})\simeq \cO_{\CGr}(-4),$$ whose dual is ample on $\CGr$.
This ensures that the index is a multiple of four; it were bigger than four,  general results would imply that it is equal to $8$ and that $\CGr$ is a quadric, which is not the case. 
\end{proof}


\begin{prop}\label{lem:cg-subalgebras}
A codimension $3$ subspace $B\subset V_7$ belongs to $\CGr$ if and only if its orthogonal $B^{\perp}$ is a subalgebra of $\OO^{\CC}.$
\end{prop}


\begin{proof}
Recall that $\omega$ is given in terms of octonionic multiplication by $\omega(x,y,z)=Re((xy)z)$. 
The condition that $\omega$ vanishes identically on $B$ is therefore equivalent to $B.B\subset B^\perp.$ 

Suppose this is the case. For $x\in B$, the left multiplication operator by $x$ is thus an 
endomorphism of $\OO^\CC$ that maps $B$ to $B^\perp$. Note that since the invariant quadratic form $q$ is nondegenerate, $B$ cannot be isotropic; for $x$ general in $B$, $q(x)\ne 0$, $L_x$ is an isomorphism and proportional to its own inverse. This implies that $L_x$ also maps $B^\perp$ to $B$, and since this is true for $x$ general in $B$, it must be true for any $x\in B$. This means that $B.B^\perp\subset B$. But then, if $x\in B$ and $y,z\in B^\perp\cap Im(\OO)$, we get 
$(x,yz) = -(xz,y) =0$ since $xz\in B$. So  $B^\perp.B^\perp\subset B^\perp$, which means that $B^\perp$ is a subalgebra of $\OO^\CC$. 

The converse implication is similar and left to the reader. 
\end{proof}

So, the Cayley Grassmannian parametrizes the four-dimensional subalgebras of $\OO^{\CC}$. Over the real numbers, all subalgebras of $\OO$ are isomorphic to the quaternion algebra $\HH$. Over the complex numbers, this is no longer true, since there exist \emph{degenerate subalgebras}.
\begin{definition}
A subalgebra $\mathbb{A}\subset \OO^{\CC}$ is said to be nondegenerate if the restriction of the octonionic norm to $\mathbb{A}$ is a nondegenerate quadratic form. Otherwise, $\mathbb{A}$ is degenerate.
\end{definition}

\begin{lemma}\label{lem:nondeg-quaternion}
Any 4-dimensional nondegenerate subalgebra $\mathbb{A}\subset \OO^{\CC}$ is isomorphic to $\mathbb{H}^{\CC}$.
\end{lemma}

\begin{proof} Choose two orthogonal norm-one vectors $e_1, e_2$ in 
$\Im(\mathbb{A})$; then  ${e_i}^2=-e_i\overline{e_i}=-1$
for $i=1,2$. 
Let $e_3=e_1e_2=-e_2e_1;$ this an octonion of norm one and since $\OO^{\CC}$ is alternative, the subalgebra generated by $e_1$ and $e_2$ is associative, therefore $e_1 e_3=-e_2$ and $e_2 e_3=e_1$. To conclude the proof, it suffices to show that $e_3\not\in\spann{1, e_1,e_2}.$
From the choice of $e_1$ and $e_2$, $e_3$ is imaginary since 
$\langle e_3,1\rangle=\langle e_1e_2,1\rangle=\langle e_1,e_2\rangle=0$. Then suppose that $e_3=\alpha e_1+\beta e_2$ for some $\alpha, \beta\in \CC$; multiplying on the left by $e_1$ yields $-e_2=-\alpha 1+\beta e_3$, hence $\alpha=0$ by taking real parts; and similarly $\beta=0$.  
\end{proof}

Let us now pass to the study of degenerate subalgebras $\mathbb{A}$.
\begin{lemma}\label{lem:deg-null-plane}
Suppose that $\mathbb{A}$ is degenerate. Then $\Im(\mathbb{A})$ contains a null-plane $N$.
\end{lemma}
\begin{proof}
The degeneracy assumption on $\mathbb{A}$ is equivalent to the existence of a non-zero element $x\in \mathbb{A}$ such that $\mathbb{A}\subset x^\perp.$ Observe that such an $x$ must necessarily lie in $\Im(\mathbb{A})\cap \QQ^5\subset \Im(\OO^{\CC})$, as $\langle x, 1\rangle=0=\langle x,x\rangle$. Moreover the left multiplication operator $L_x$  preserves $\Im(\mathbb{A})$
since $\Re(xy)=-(x,\overline{y})=-(x,y)$. But $x$ being an imaginary octonion of norm zero, 
 $x^2=0$ and thus $L_x^2=0$ ($\OO^{\CC}$ being alternative). So $L_x$ has rank at most one as an endomorphism of $\Im(\mathbb{A})$, and there exists
 $y\in \Im(\mathbb{A})$, independent of $x$, such that $xy=0$ (and $yx=0$ by conjugating). 
Necessarily $y$ has zero norm, being a zero divisor, and therefore   $N= \spann{x,y}$ is a null-plane. 
\end{proof}

Using these preliminary results we can classify, up to the $G_2$ action, the 4-dimensional subalgebras $\mathbb{A}$ of $\OO^{\CC}$. Denote by $r$ the rank of the restriction of the octonionic norm to $\mathbb{A}.$

\begin{prop}\label{prop:classif-subalgebras}
Up to the action of $G_2$, there are only three possibilities:
\begin{enumerate}
    \item $r=4$, then $\mathbb{A}$ is a quaternion algebra;
    \item $r=2$, then there exists a unique null-plane $N$ such that $N\subset \Im(\mathbb{A})\subset N^\perp$ 
    \item $r=1$, and there exists a unique isotropic line $\ell$ such that $\Im(\mathbb{A})=\ell\OO\cap \Im(\OO)$. 
\end{enumerate}
\end{prop}

\begin{proof} If $r=4,$ $\mathbb{A}$ is nondegenerate, hence isomorphic to $\mathbb{H}^{\CC}$ by Lemma  \ref{lem:nondeg-quaternion}. Suppose  that $\mathbb{A}$ is degenerate; by Lemma \ref{lem:deg-null-plane}, $\Im(\mathbb{A})$ contains a null-plane $N$. Suppose there is another one, say $N'$, and let $\ell$ belong to $N\cap N'$. Then $\ell.\Im(\mathbb{A})=\ell.(N+N')=0$ and we can conclude that  $\Im(\mathbb{A})=\ell\OO\cap \Im(\OO)$ and $r=1$. 

It follows that if $r=2$,  $N$ is unique. What remains to prove is that any $z\in\Im(\mathbb{A})$ is orthogonal to $N=\spann{x,y}$, with the notations of Lemma \ref{lem:deg-null-plane}. For this we note that 
$xz$ belongs to $\AA$ and is killed by left multiplication by $x$, which implies that $xz\in N$. 
So $xz=\alpha x+\beta y$ for some scalars $\alpha, \beta$, that we can write as $\beta y=(x-\alpha 1)z$. Since $z\ne 0$ and $x-\alpha 1$ is invertible, we deduce that $\beta\ne 0$. Then $y$ is a linear combination of $xz$ and $x$, hence also $yz$. In particular $Re(yz)=Re(xz)=0$, hence $z\in N^\perp$.  
\end{proof}

The previous classification leads to the following
\begin{prop}\label{prop:orbits-CG}
The action of $G_2$ on $\CGr$ has three orbits: 
\begin{enumerate}
    \item the open one is isomorphic with $G_2/SL_2\times SL_2;$ 
    \item the intermediate one is a complement to a conic bundle in a $\PP^2$-bundle over $\GtwoGr(2,7)$; 
    \item the closed one is isomorphic to the quadric $\QQ^5.$
\end{enumerate}

\end{prop}

\begin{proof}
Given a four-dimensional nondegenerate subalgebra $\AA$, with a basis  
$e_1, \: e_2\: e_3$ of  $\Im(\mathbb{A})$  as in the proof of Lemma~\ref{lem:nondeg-quaternion}, we can choose a norm-one element $e_4$ orthogonal to $\mathbb{A}$ and define $e_5=e_3e_4, \ e_6=e_2e_4, \ e_7=e_4e_1$. The multiplication table of octonions in this basis is then the same as for the standard basis, proving that the corresponding change of basis defines an element of $G_2$; by construction it maps $\mathbb{A}$ to the standard quaternion algebra $\HH^\CC\subset \OO^{\CC}$, hence the claim that four-dimensional nondegenerate subalgebras are parametrized by a unique orbit of $G_2$.

To prove $(1)$, there remains to compute the stabilizer of  $\HH^\CC$. 
Consider the restriction map  $$r\colon \Stab_{G_2}(\mathbb{A})\to \Aut(\mathbb{A})\simeq \SL_2.$$ This is a surjection whose kernel is isomorphic to the unit vectors in $\mathbb{A}^{\perp}$ (see, e.g., the multiplication by $e_4$ described above). As these are parametrized by $\SL_2$, we get $\Stab_{G_2}(\mathbb{A})\simeq \SL_2\times \SL_2$.

For the cases $(2)$ and $(3)$, we first show that, given a null-plane $N\subset V_7$, any three-dimensional subspace $W$ of $N^{\perp}$ containing $N$ is the imaginary part of a degenerate subalgebra $\AA$ of $\OO^{\CC}$. 
By Lemma~\ref{lem:cg-subalgebras}, it suffices to check that $W^{\perp}\subset V_7$ is isotropic with respect to $\omega$. 
Since $N$ is a null-plane, this is equivalent to stating that, taking an arbitrary $3$ dimensional subspace $U$ of $N^{\perp}$ disjoint from $N$, 
we have $x\intprod (y\intprod \omega)|_{N}\equiv 0$. 
However, due to the transitivity of the action on the $G_2$ Grassmannian, this can be verified directly on a specific null-plane. Take, for instance, 
\[ N=\spann{x_{\alpha},x_{-\beta}}, \hspace{5mm} U=\spann{x_0, x_{\gamma}, x_{-\gamma}}. \]
One checks that the linear morphism $\Gr(2,U)\to \PP(V_7)^*, \ \spann{x, y} \mapsto x\intprod(y\intprod \omega)$ maps $\Gr(2,U)$ isomorphically to $\PP(q(U))\simeq \PP(U)^*$ and that the latter is contained in $\PP(\ann(N))$.
The family of 3-dimensional subspaces of $N^{\perp}$ containing $N$ identifies with $\PP(U)\simeq \PP(N^{\perp}/N)$, and a point $u\in \PP(U)$ defines an algebra $\mathbb{A}$ such that $r=2$ if and only if it lies outside the smooth conic $\QQ^5\cap \PP(U)$, hence the claim since $N$ is uniquely defined by $\AA$. 
Finally, the family of degenerate subalgebras such that $r=1$ identifies with the family of planes in $\QQ^5$ defined by
\[ \QQ^5\rightarrow \OGr(3,7),\qquad [x]\mapsto \PP(\ker(x\intprod \omega)), \]
which, by construction, is isomorphic to $\QQ^5$ itself.
\end{proof}

Over the real numbers, quaternion subalgebras of 
the octonions are parametrized by the compact homogeneous space $G_2^c/SO_3(\RR)\times SO_3(\RR)$.
Note that 
$\Aut(\HH)=SO_3(\RR)$, so there must exist automorphisms of $\OO$ that act trivially on $\HH$ (Exercise \ref{ex_aut_OO_HH}).

\section{The double Cayley Grassmannian}

A ``doubled'' version of the Cayley Grassmannian is provided by a special compactification of the group $G_2$: the \emph{double Cayley Grassmannian} $\DGr$, see \cite{Manivel_double}. This variety can be defined as parametrizing eight-dimensional subalgebras of the complex bioctonion algebra $\OO^\CC \otimes_\RR \CC$ (hence its name). We will describe the geometric construction of $\DGr$ and briefly hint at its connection with octonions.

The first observation is the following. Consider the group $\Spin_{14}$, which is the universal (degree two) cover of $\SO_{14}$. The projectivization of the spinor representation $\PP \Delta^+_{14}$ is prehomogeneous for the action of $\Spin_{14}$, meaning that it contains an open orbit (see Example \ref{ex_spinor_orbit_grading}). The stabilizer of an element $\delta$ of the open orbit has thus dimension equal to $\dim(\Spin_{14})-\dim(\PP\Delta^+_{14})=91-63=28=2\dim(G_2)$. The last equality is not just a coincidence, as shown by the following result:

\begin{prop}\cite{KimSato, AbMan}\label{prop_lie_double_G2}
$\Stab_{\Spin_{14}}(\delta)$ is isomorphic with $(G_2\times G_2)\rtimes \ZZ_2$.
\end{prop}

Before sketching the proof, a few words about spinor representations. First recall that, given a complex vector space $W$ of even dimension $2n$ (the odd-dimensional case is similar) endowed with a nondegenerate quadratic form $q\in S^2W^*$, one can construct two \emph{half-spin representations} $\Delta^{\pm}(W)=\Delta^{\pm}_{2n}$ (but a unique one in the odd-dimensional case, $\Delta(W)$). 
For this one chooses two transverse  maximal isotropic subspaces $E,F\subset W$: so  $W=E\oplus F$
and $q$ restricts to a perfect duality between $E$ and $F$.  As vector spaces, $\Delta^+(W)$ and  $\Delta^-(W)$ are isomorphic to the even and  odd exterior powers 
$$\Delta^+(W)\simeq \extp^+F=(\extp^0\oplus \extp^2\oplus\cdots\oplus \extp^{2\lfloor n/2\rfloor})F, \quad \Delta^-(W)\simeq \extp^-F=(\extp^1\oplus \extp^3\oplus\cdots\oplus \extp^{2\lfloor (n-1)/2\rfloor+1})F.$$ 
(In the odd-dimensional case, one considers the whole exterior algebra $\Delta(W)=\extp^\bullet F$.) Then $\Delta^\pm(W)$ is endowed with an action of $\fso(W,q)\simeq \wedge^2W$ by first defining an action of $W$ and its so-called \emph{Clifford algebra} on $\extp^\bullet F$  (see for instance \cite{chevalley,FH}). For $\delta\in \Delta^\pm(W)$ and $w=e+f\in E\oplus F \simeq W$,  the  action is simply defined by the formula $w \otimes \delta \mapsto e \intprod \delta + f\wedge \delta$, where $e \intprod$ is the contraction induced by the identification of $E$ with $F^*$. 
One can check that when $n$ is odd,  $\Delta^+_{2n}$ and $\Delta^-_{2n}$ are dual one to the other, while they are both self-dual when $n$ is even. In the odd-dimensional case,  $\Delta_{2n+1}$ is always self-dual.

\begin{lemma}
\label{lem_spin7}
The stabilizer in $\Spin(V_7)$ of a general element of $\PP(\Delta_7)$ is isomorphic to $G_2$.
\end{lemma}

\begin{proof}
We use the theory of $\ZZ$-gradings, see Appendix \ref{sec_gradings}, that we apply to 
$\ff_4$ as in  Example \ref{main_ex_grading_Z} by choosing $h$ dual to the $4$th simple root. Then $\fg_0\simeq \fso(V_7)\oplus \CC$ and $\fg_1\simeq \Delta_7$. Moreover, the group $G_0$ is the semidirect product $\Spin(V_7)\rtimes \CC^*$, where $\CC^*$ acts by homotheties. Since there are only finitely many $G_0$-orbits in $\fg_1$ (Theorem \ref{thm_finite_parabolic}), a general element $\chi\in\Delta_7$ has a dense $G_0$-orbit, and $\Stab_{\Spin(V_7)}(\chi)$ has dimension equal to $\dim(\Spin(V_7))+\dim(\CC^*)-\dim(\Delta_7)=21-8+1=14$. We already know that $G_2\subset \Spin(V_7)$, and that it is connected (and irreducible) of dimension $14$; if we can show that $\Stab_{\Spin(V_7)}(\chi)\subset G_2$, we are done.

Let us consider the Clifford action $V_7\otimes \Delta_7 \to \Delta_7$. The element $\chi$ defines a three-form $w_\chi$ on $V_7$ by the formula $w_\chi(x\wedge y\wedge z)=\langle x\cdot( y\cdot ( z\cdot \chi)), \chi \rangle$, where we have used both the Clifford action and self-duality of $\Delta_7$. One can show that, if $\chi$ is in the dense orbit (which in this case is just the complement of a hypersurface, namely the cone over $\OGr(3,7)$), then $w_\chi$ is a generic form in $\extp^3 V_7$ (Exercise \ref{ex_spinors}). By Theorem~\ref{omegageneric}, we deduce that $\Stab_{\Spin(V_7)}(\chi)\subset \Stab_{\GL(V_7)}(w_\chi)\simeq G_2$, and the proof is complete.
\end{proof}

\begin{remark}
\label{rem_G2G_in_OG}
We have seen that a null-plane is isotropic with respect to the symmetric quadratic form. This means that we have a codimension two embedding $\GtwoGr(2,V_7)\subset \OGr(2,V_7)$ (see also Exercise \ref{ex_OG}). By the classification of homogeneous vector bundles on homogeneous spaces (refer to \cite{Otta}), there exists a rank two vector bundle $\cS$ on $\OGr(2,V_7)$, called the spinor bundle, whose space of global sections is $\Delta_7$. For a general section $\chi\in \Delta_7\simeq H^0(\OGr(2,V_7), \cS)$, the zero locus of $\chi$ is a smooth five dimensional subvariety of $\OGr(2,V_7)$ which, by the previous lemma, is stabilized by $G_2$: it is in fact exactly $\GtwoGr(2,V_7)$.
\end{remark}

\begin{proof}[Sketch of proof of Proposition \ref{prop_lie_double_G2}]
    Let  $V_{14}$ be the natural representation of $\Spin_{14}$. One proves that given $\delta$ generic in $\Delta_{14}^+$, there exists a uniquely defined orthogonal decomposition $V_{14}=V_7\oplus V'_7$ with the following properties. Since by restriction $V_7$ and $V'_7$ are endowed with nondegenerate quadratic forms, one can define their spin representations  $\Delta_7$ and  $\Delta'_7$, as well as an isomorphism  $\Delta_{14}\simeq \Delta_7\otimes \Delta'_7$. Then the astonishing fact is that there exist  $\chi\in \Delta_7$ and $\chi'\in \Delta'_7$ such that $\delta\simeq\chi\otimes \chi'$. As a consequence, the stabilizer of $\delta$ in $\Spin_{14}$ has to fix the pair $(V_7,V'_7)$ (hence the factor $\ZZ_2$ in the stabilizer), and the index two subgroup that fixes $V_7$ and $V'_7$ can be deduced from the stabilizers of $\chi$ in $\Spin(V_7)$ and $\chi'$) in $\Spin(V_7')$), hence the two copies of $G_2$ by Lemma \ref{lem_spin7}.
\end{proof}

Having obtained this, the motto we will follow is: ``Any variety constructed from the only datum of $\delta$ (a generic element in $\Delta^+_{14}$) will inherit an action of its stabilizer 
$(G_2\times G_2) \rtimes \ZZ_2$''. In order to construct such a variety, we
start with the spinor variety $\SSS_{14}$, which is one of the two isomorphic connected components of the orthogonal Grassmannian $\OGr(7,V_{14})$:
\[ \SSS_{14}\subset \OGr(7,V_{14}):= \left\{[U_7]\in \Gr(7,V_{14}) \mid q|_{U_7}=0\right\}= \SSS_{14} \cup \SSS'_{14}\subset \Gr(7,V_{14}). \]
Let $\cU$ be the restriction to $\SSS_{14}$ of the tautological vector bundle on $\Gr(7,V_{14})$. 
It is a standard fact that $\det(\cU)=\cL^{-2}$, where $\cL$  is the positive generator of the Picard group of $\SSS_{14}$. In particular $\cL$ is ample, in fact very ample, and defines an embedding $\SSS_{14} \subset \PP (\Delta^+_{14})$. The key property of this embedding is that it 
sends $x=[U_7]\in \SSS_{14}$ to $[\delta_x]$, where $\delta_x$ is a {\it pure spinor} from which we can recover $U_7\subset V_{14}$ as the kernel of the Clifford mutiplication $V_{14}\otimes \delta_x\subset V_{14}\otimes \Delta_{14}^+\lra \Delta_{14}^-$.  

\begin{lemma}
$ H^0(\SSS_{14},\cU\otimes \cL) \simeq \Delta^+_{14}. $
\end{lemma}

\begin{proof}
This is a special case of the Borel-Weil Theorem \cite{serre-bw}.
Let us be more explicit about how  $\delta\in \Delta^+_{14}$ defines 
a section $s_\delta$ of $\cU\otimes \cL$. For any point $x:=[U_7]\in \SSS_{14}$, $s_\delta(x)$ will belong to the fiber $(\cU\otimes \cL)|_{x}$, hence will define an element of $U_7$, at least up to scalar. In order to construct this element, consider the linear form on $V_{14}$ sending $v$ to 
$\langle v.\delta_x,\delta\rangle$. Since $v.\delta_x=0$ for $v\in U_7$, the kernel of this linear form is (in general) a hyperplane $H$ containing $V_7$, whose orthogonal $h=H^\perp$ is contained in $U_7^\perp=U_7$; this is the subspace we were looking for. 
\end{proof}

\begin{remark}
Both $\cL$ and $\cU\otimes \cL$ are  referred to as {\it spinor bundles} since their spaces of sections are spin representations.
\end{remark}

Now, being equivariant and irreducible $\cU\otimes \cL$ is globally generated.  By Bertini's theorem, since the rank is $7$,  for $\delta\in \Delta^+_{14}$ generic, the zero locus $Y_\delta:=\left\{ x\in \SSS_{14}\mid s_\delta(x)=0 \right\}$ is a smooth codimension $7$ subvariety of $\SSS_{14}$. Since the latter has dimension  $21$,  $Y_\delta$ is a $14$-dimensional variety acted on by $(G_2\times G_2)\rtimes \ZZ_2$; can you guess which one?

\begin{prop}\cite{Manivel_double}
$Y_\delta$ is a compactification of $G_2$ and identifies with $\DGr$. 
\end{prop}

More precisely, $Y$ admits an action of $G_2\times G_2$ that compactifies the left-right action! As for the Cayley Grassmannian,
there are only three orbits in the double  Cayley Grassmannian : the open one, an open subset of a special hyperplane section, and the closed orbit isomorphic to $\QQ^5\times \QQ^5$. The blowup of the closed orbit is the so-called wonderful compactification of $G_2$, see \cite{dCP}. 

\begin{proof}[Sketch of proof]
As mentioned, $\DGr$ parametrizes $8$-dimensional subalgebras of $\OO^\CC\otimes_\RR \CC$; the latter is isomorphic (as an algebra) to $\OO^\CC \oplus \OO^\CC$ (Exercise \ref{ex_productofC}). A general subalgebra in $\DGr$ will be the graph of an element $g\in \GL(\OO^\CC)$; this element must satisfy $g(1)=1$ and, for its graph to be a subalgebra, one needs that $g$ belongs to $G_2$. Thus, the graph of $g$ defines a $7$-dimensional subspace $L_g$ of $V_{14}:=\Im(\OO^\CC)\oplus Im(\OO^\CC)$, which is isotropic with respect to the difference of the octonionic norms on the two copies of $\OO^\CC$. The map $g\mapsto L_g$ defines an embedding $G_2 \subset \DGr \to \SSS_{14} $. One can show that its image  is contained in $Y_\delta$ (when $\delta$ is chosen so as to identify the associated $V_7$ and $V_7'$ with the two copies of $\Im(\OO^\CC)$). Since being a subalgebra is a closed condition and $G_2$ is open both inside $\DGr$ and $Y$, we deduce that $\DGr\simeq Y$.
\end{proof}

\section{The Lie Grassmannian}

The canonical three-form $\omega$ yields a map from $\extp^2V_7$ to $V_7$, similar to 
a Lie bracket. This is just the skew-symmetrized octonionic multiplication, that is  $$ \omega(x,y,\cdot\,)=\frac{1}{2}\langle xy-yx,\cdot\,\rangle = Im(xy).$$

\begin{remark}
This map is not quite a Lie bracket: the Jacobi identity does not hold. Indeed, the map $\phi:\extp^2V_7
\lra V_7$ can be read on the Fano plane, in the form $[e_i,e_j]=e_{ij}$. One directly checks that the Jacobi identity gives zero 
for $i,j,k$ aligned, and otherwise gives $3e_{(ij)k}$.
\end{remark}  

Although we do not get a Lie algebra structure on $V_7$, it may happen that some subspace $\fs\subset V_7$ is stable under the fake Lie bracket, and that the Jacobi identity holds on $\fs$. We improperly call such subspaces  \emph{Lie subalgebras} of $V_7$.

\begin{definition}
    We denote by $\LieGr(2,V_7)$ the subvariety of $\Gr(2,V_7)$ parametrizing two-dimensi\-on\-al Lie subalgebras of $V_7$.
\end{definition}

\begin{prop} \label{LieGr}
The subvariety $\LieGr(2,V_7)$ of $\Gr(2,V_7)$ is smooth 
of dimension seven. It can be described as the locus where the section of $\extp^2\cQ^*(1)$ 
defined by $\star\omega$ drops rank. It is acted on by $G_2$ with only two orbits, the closed 
one being $\GtwoGr(2,V_7)$. 
\end{prop}

\begin{remark}
 Recall from Definition \ref{def:3-form} the coassociative 4-form $\star\omega\in\extp^4V_7^*$. It defines a section of $\extp^2 \cQ^*(1)$ as follows. Consider $P\in \Gr(2,V_7)$ so that $P$ is generated by two vectors $x_1,x_2$; then $\star\omega(P)=(x_1\wedge x_2) \lrcorner \star\omega \mod \spann{x_1,x_2} \in \extp^2(V_7/P)^*$.
\end{remark}

\begin{remark}
Let $s$ be a section of $\extp^2 \cQ^*(1)$. Since $\cQ$ has rank five, the fiber over a point of $\Gr(2,V_7)$ of $\extp^2 \cQ^*(1)$ is isomorphic to $\extp^2 \CC^5$. Thus, for $P\in \Gr(2,V_7)$, $s(P)$ can be seen as a $5\times 5$ skew-symmetric matrix. The locus where the section $s$ drops rank is the locus of points $P$ such that $s(P)$ is a matrix of rank two (recall that skew-symmetric matrices always have even rank).
\end{remark}

\begin{proof} A plane $P=\spann{x,y}$ is a Lie subalgebra if and only if $xy$ (or $yx$)
belongs to $\spann{1,x,y}$, say $xy=ax+by+c$. We can rewrite this a $(x-b1)(y-a1)=ab+c$. 
If $ab+c=0$, $X=x-b1$ and $Y=y+a1$ and inverse to the other, up to constant and conjugation; so they must be proportional, hence also their imaginary parts $x$ and $y$, a contradiction!
So $ab+c=0$ and $X\bar{Y}=0$. In particular $X$ and $Y$ are zero divisors, hence isotropic, and orthogonal since $Re(X\bar{Y})=0$, so that $P'=\spann{X,Y}$ is isotropic. If $P'\subset V_7$, hence $P'=P$, we conclude that $P$ is a null-plane. 

Otherwise, we may suppose that $Y$ is imaginary, but not $X$. Up to the action of $G_2$, we can let 
$X=1+ie_1$. Then $Y=v+iw$ has to verify the only condition that $w=e_1v$, with $v$ orthogonal to $e_1$ for the latter to be imaginary. Since we know that  $G_2^{\mathrm{c}}$ acts transitively on pairs of orthogonal real octonions, up to the action of $G_2$ we can suppose that  $Y=y=e_2+ie_3$. 

This means that  $G_2$ acts transitively on the complement of $\GtwoGr(2,V_7)$ in  $\LieGr(2,V_7)$, which is the orbit if the plane $P_0=\spann{e_1,e_2+ie_3}$. Note that the restriction of $q$ to $P_0$ has rank one, with kernel generated by $y$. By the previous discussion, Lie subalgebras of $V_7$ containing $y$ are the imaginary parts of planes $P'$ such that $y\subset P'\subset R_y$; in particular they are parametrized by a projective plane. We can deduce  that  $\LieGr(2,V_7)$ is irreducible of  dimension $5+2=7$. On the other hand, the locus $D$ where the section of $\extp^2\cQ^*(1)$ defined by $\star\omega$ drops rank has codimension $3$, hence also dimension $7$, 
and it is preserved by the action of $G_2$. So we just need to check that it contains 
$P_0$, which is a straightforward computation; then $D$ contains $\LieGr(2,V_7)$, which must be one irreducible component. A little extra argument shows that in fact $D$ is also irreducible, and we are done. \end{proof}

\begin{remark}
In the previous proof, we were looking at points $P\in \Gr(2,V_7)$ such that $\star\omega(P)$ is a $5\times 5$ skew-symmetric matrix of rank $2$, and it was stated that this is a codimension $3$ locus. This is an application of a Bertini-type Theorem for degeneracy loci \cite[Proposition 2.3]{bfmt} together with the observation that the locus of $5\times 5$ skew-symmetric matrices of rank at most $2$ has codimension three inside the space of skew-symmetric matrices $\extp^2 \CC^5$; indeed it can be identified with the cone over $\Gr(2,5)$.
\end{remark}


Although $\LieGr(2,V_7)$ seems to be a mysterious object, it is in fact familiar:

\begin{prop}
$\LieGr(2,V_7)\simeq \OGr(2,V_7)$. 
\end{prop}

\begin{proof}[Sketch of proofs]
This follows from the classification of $2$-orbit varieties given in \cite{Cupit-Foutou}.

Alternatively, one can show that the Picard number of $\LieGr(2,V_7)$ is one, which allows to apply \cite[Theorem 0.1]{Pa1}. We deduce that $\LieGr(2,V_7)$ has a bigger automorphism group than $G_2$, and that this automorphism group acts transitively. Finally, one checks that the only seven-dimensional homogeneous variety containing $\GtwoGr(2,V_7)$ is $\OGr(2,V_7)$.

A more explicit proof was given in \cite{beri-man}, in relation with the decomposition 
$$\wedge^2V_7 = \fg_2\oplus V_7.$$
This decomposition induces an equivariant involution $\iota$ of $\wedge^2V_7$, which does not preserve $G(2,V_7)$ but exchanges $\LieGr(2,V_7)\simeq \OGr(2,V_7)$. 
\end{proof}


\section{The horospherical \texorpdfstring{$G_2$}{G2}-Grassmannian}

Another interesting variety that we can define from $G_2$ is the following: let $N=\spann{e,f}$ be a null-plane, and let $\HG$ denote the closure of the $G_2$-orbit of the point $\spann{e+e\wedge f}$ inside $\PP(V_7\oplus\fg_2)$ \cite{Pa1}.


\begin{prop}
$\HG$ is a smooth Fano variety of dimension $7$ and Picard number one, acted on by $G_2$ with 
three orbits, but  only two orbits of its automorphism group. 
\end{prop}

\begin{proof} Consider the variety $Y$ parametrizing the pairs of planes $N\subset V_7=\Im\OO^\CC, P\subset\OO^\CC$
such that $N$ is a null-plane and $\Im P\subset N$. This is the subvariety of $\GtwoGr(2,7)\times \Gr(2,8)$
defined by the tautological section of $Q\boxtimes U^*$, a bundle of rank $10$ on a variety of 
dimension $5+12$; we deduce by Bertini's Theorem that $Y$ is smooth of dimension $7$. 

On the other hand, by forgetting $N$, we get a map to  $$\pi : Y\lra \Gr(2,\OO^\CC)\subset\PP(\extp^2\OO^\CC)=\PP(\extp^2(\CC\oplus V_7))=\PP(V_7\oplus \extp^2V_7),$$ sending the plane $P=\spann{1-f, e}$ to  $\spann{e+e\wedge f}$. So $\pi (Y)=\HG$.
Moreover, from $P$ we can recover $N$ by projection to the imaginary octonions, 
except when $1\in P$. In the latter case, $P=\spann{1, x}$
for some isotropic $x$ in $\Im\OO^\CC$, and we know there is a $\PP^1$ of null-planes containing $x$. This means that  $\pi: Y\lra \HG$ is birational and contracts a divisor $E$ to $\QQ^5$, 
with fibers isomorphic to $\PP^1$. In such a situation, $X$ must be smooth, and $\pi$ must be the blow-up of  $\QQ^5$. \end{proof}

\begin{remark}
This birational projection is very close to the description we gave of the open orbit of 
$ \LieGr(2,V_7)$ in the proof of Proposition \ref{LieGr}. In fact one can show that 
$\OGr(2,V_7)\simeq \LieGr(2,V_7)$ deforms to $\HG$. As we already mentionned, it is a very remarkable and unexpected fact that $\OGr(2,V_7)$ is the only rational homogeneous variety with Picard number one which is not globally rigid \cite{PPa, kuznetsov-horo, hwang-li-horo}.
\end{remark}

\begin{remark}
 Note that $\HG-\QQ^5\simeq Y-E$ is an $\AA^2$-fibration over $\GtwoGr(2,7)$. 
Over the open $G_2$-orbit, this restricts to a $\GG_m^2$-fibration over $\GtwoGr(2,7)$. In fact, 
$\HG$ is an instance of a {\it horospherical variety}, which is roughly a relative toric variety
over a flag manifold. This is precisely this structure that explains the fact that the automorphism group is bigger than $G_2$, and is not semisimple. 
\end{remark}
 
\begin{remark}
In this lecture we have encountered several varieties on which $G_2$ acts with an open orbit, isomorphic with $G_2/H$ for some subgroup $H$. There is  a whole theory of equivariant compactifications of such quotients $G/H$ for $H$ a closed subgroup of a reductive group $G$. This includes and generalizes the famous theory of toric varieties, and behaves particularly nicely when $H$ is a so-called  
\emph{spherical subgroup} (see e.g. \cite{perrin-spherical}. 
\end{remark}

\section{Exercises}

\begin{exercise}\label{ex:SL3}
\necessary
Show that  $\PP(\Im(\OO^{\CC}))\setminus\QQ^5\simeq G_2/\SL_3\rtimes\ZZ_2.$
\end{exercise}

\begin{exercise}\label{ex_Kxker}
\computation
Show that $\ker(L_x|_{\Im\OO^\CC})=\ker(x\intprod \omega)$ for every $[x]\in\QQ^5$.
\end{exercise}


\begin{exercise}
\label{ex_OG}
\necessary
Show that not every line in $\QQ^5$ is a $G_2$-line. More precisely, show that the lines passing through a point $[x]\in\QQ^5$ are parametrized by a three-dimensional variety in which the set of $G_2$-lines through $[x]$ has codimension two.
\end{exercise}

\begin{exercise}
\training
    Similarly to Diagram \ref{eq:diag-A} and Proposition \ref{prop_lines_G2}, discuss the incidence diagrams for the other complex Lie groups of rank two.
\end{exercise} 

\begin{exercise}
\curiosity
\label{ex_aut_OO_HH}
    Describe the automorphisms of $\OO$ that act trivially on $\HH$. 
\end{exercise}

\begin{exercise}
\label{ex_spinors}
\computation
Let $\chi$ be a spinor in the dense orbit inside the spinor representation $\Delta_7$; show that the associated three-form $\omega_\chi$ defined by $\omega_\chi(x\wedge y\wedge z)=\langle x\cdot (y\cdot(z\cdot \chi)),\chi \rangle$ belongs to the open orbit in $\extp^3 V_7$ (\emph{Hint}: search in the literature for representative of orbits in $\Delta_7$ and use such a representative to do the explicit computation).
\end{exercise}

\begin{exercise}
\label{ex_productofC}
\computation
Show that $\OO^\CC \otimes_\RR \CC \simeq \OO^\CC \oplus \OO^\CC$ as algebras (\emph{Hint}: show and use the fact that $\CC\otimes_\RR \CC\simeq \CC\oplus \CC$).
\end{exercise}

\begin{exercise}
\label{ex_LieG}
\curiosity
Construct a diagram
$$\xymatrix{ & {I}\ar[ld]_{ p}\ar[rd]^{ q} \\ \LieGr(2,V_7) & & \CGr ,}$$
where the fibers of $ p$ are $\PP^2$'s, while the generic fiber of $ q$ are conics.
\end{exercise}

\begin{exercise}
\label{ex_genstab_bigger}
\training
Show that $\Aut(\HG)$ is bigger than $G_2$.
\end{exercise}  

\medskip
\begin{sol}[Hint Ex. \ref{ex:SL3}]
    Compute the stabilizer $\Stab_{G_2}([x])$ of a point $[x]\in \PP(V_7)\setminus \QQ^5$. Note that if $\phi\in \Stab_{G_2}(x)$, then $\phi$ must also preserve $V_x:= x^{\perp}\subset V_7$ and $\omega|_{V_x}\in \extp^3 V_x^*$, the restriction of the invariant three-form. $\omega|_{V_x}$ must be in the orbit of $\omega_1$ from Proposition~\ref{prop:orbits-66}, hence there is a pair of distinguished 3-dimensional subspaces $\Lambda_1$ and $\Lambda_2$ preserved by $\phi$ such that $V_x = \Lambda_1 \oplus \Lambda_2$. When each of the two spaces is preserved, use the quadratic form $q|_{V_x}$ to relate the actions of $\phi$ acts on 
    $\Lambda_1$  and  $\Lambda_2$.
\end{sol} 



\begin{sol}[Ex. \ref{ex_OG}]
The Grassmannian $\Gr(2,V_7)$ has dimension $10$; the subvariety $\OGr(2,V_7)$ has dimension $7$ since 
the condition that a quadratic form vanishes on a line is verified in codimension three. Therefore  the incidence variety parametrizing pairs $(x,\ell)\in \QQ^5\times OG(2,V_7)$ with $x\in \ell$ has 
dimension  $8$ and by homogeneity, all the fibers of its projection to $\QQ^5$ have the same dimension $3$. This means that there is a three-dimensional family of lines in $\QQ^5$ 
passing  through a given point $x$. The same analysis with  $\OGr(2,V_7)$ replaced by $\GtwoGr(2,7)$
shows that there is only a one-dimensional family of special lines through $x$.
\end{sol}

\begin{sol}[Ex. \ref{ex_LieG}]
We have seen that $\LieGr(2,V_7)$ coincides with the locus of points $P\in \Gr(2,V_7)$ where $\star\omega(P)\in \extp^2 (V_7/P)^*$ has rank two. This means that, if $P$ belongs to $\LieGr(2,V_7)$, the kernel of $\star\omega(P)$ is a three-dimensional subspace of $V_7/P$, i.e., a five-dimensional subspace $U$ of $V_7$ containing $P$. On the other hand, $\CGr$ is the subvariety of $\Gr(3,V_7)$ parametrizing subspaces $A$ such that $\omega|_{A^\perp}=0$, or equivalently $\star\omega(A,A,A,\cdot\,)=0$ (meaning that, whenever we contract $\star\omega$ with three vectors in $A$, we get  zero). 
This suggests to construct $ I$ inside the flag variety $Fl(2,3,5,V_7)$ as the locus of flags $P\subset A\subset U$ with $\dim(P)=2$, $\dim(A)=3$, $\dim(U)=5$ satisfying the condition that  $\star\omega(P,P,U,\cdot\,)=0$ (meaning that contracting $\star\omega$ with two vectors in $P$ and one in $U$ always yields zero). Notice that this condition automatically implies that $\star\omega(A,A,A,\cdot\,)=0$, so $A\in \CGr$. We can thus define ${ p}(P,A,U)=P$ and ${{q}}(P,A,U)=A$. 
$$\xymatrix{ & {I}\ar[ld]_{ p}\ar[rd]^{ q} \\ \LieGr(2,V_7) & & \CGr .}$$
The fibers of $ p$ are $\PP^2$'s, while the generic fiber of $ q$ are conics. Indeed, recall the interpretation of elements $A\in \CGr$ as subalgebras in $\OO^\CC$. The fiber of $ q$ over $A$ parametrizes Lie subalgebras contained in $A$. If $A$ is generic, then $A=Im \HH^\CC$ and the fiber is a conic, a point in the fiber being uniquely determined by the kernel of the restriction of $q$. 
\end{sol} 

 \begin{sol}[Ex. \ref{ex_genstab_bigger}]
The observation is that $V_7$ act on $X$ through 
$v.(e+f_1\wedge f_2)=e+q(v,f_1)f_2-q(v,f_2)f_1+f_1\wedge f_2$. If such an action of $v\in V_7$ were defined by an element $g\in G_2$, then $g\cdot (e+f_1\wedge f_2)=g\cdot e+ g\cdot (f_1\wedge f_2)=(e+q(v,f_1)f_2-q(v,f_2)f_1)+f_1\wedge f_2$. This would mean that $g\cdot f_1\wedge f_2 =f_1\wedge f_2$, i.e., that $g$ acts on $\fg$ as the identity; this is not possible because $v\in V_7$ does not act as the identity on the first factor. We deduce that $\Aut(\HG)$ is not contained in $G_2$. 
\end{sol}

\chapter{The Tits-Freudenthal magic square}

In this lecture, we will present two  constructions of the exceptional Lie algebras that use Jordan algebras and triality. Both constructions will depend on the choice of a pair of complexified normed algebras and will lead to the so-called Magic Square of Lie algebras. We will also begin to study the geometric incarnations of these constructions.

\section{The Jordan algebra \texorpdfstring{$H_3(\AA)$}{H3(A)}}

Let $\AA$ be one of the complexified normed algebras $\RR^\CC$, $\CC^\CC$, $\HH^\CC$, $\OO^\CC$. We let  $H_3(\AA)$ be the space of Hermitian matrices with coefficients in $\AA$:
$$H_3(\AA)=\left\{ \begin{pmatrix} x & u & v \\ \bar{u} & y & w \\
\bar{v} & \bar{w} & z\end{pmatrix}, \quad u,v,w\in\AA, x,y,z\in \RR^\CC\right\}.$$

\begin{definition}
A Jordan algebra is a non-associative commutative algebra over a field that satisfies the Jordan identity
$(x\circ y)\circ (x\circ x)=x\circ (y\circ (x\circ x))$ (where $\circ$ is the multiplication).
\end{definition}

The easiest way to obtain a Jordan algebra structure is to start from an associative algebra and modify the product as $x\circ y=(xy+yx)/2$. Therefore, if $\AA\neq \OO^\CC$ there is a {\it Jordan algebra} structure on $H_3(\AA)$ defined by the symmetric
product $A\circ B=(AB+BA)/2$.
This Jordan algebra has rank three in the sense that for any $A\in H_3(\AA)$
there is a cubic relation 
$$A^3-\tr(A)A^2+\sigma_2(A)A-Det(A)Id=0,$$
where $\tr(A)$ is the usual trace and $\sigma_2(A)=\frac{1}{2}(\tr(A)^2-\tr(A^2))$. No effort is needed to check that $H_3(\RR^\CC)$ identifies with symmetric $3\times 3$ complex matrices. The verification that $H_3(\CC^\CC)$ can be identified with the space of $3\times 3$ complex matrices is more involved, and even more so is that $H_3(\HH^\CC)$ is isomorphic to the space of skew-symmetric $6\times 6$ matrices (\necessary Exercise \ref{ex_jordan_CCHH}). However, the real challenge is to  understand that for $A=\OO^\CC$, one still gets a Jordan algebra, usually called the {\it  exceptional simple Jordan algebra}. 

\begin{remark}
When $\AA$ is not $\OO^\CC$, one can also define a Jordan algebra structure on $H_n(\AA)$ for any $n> 3$ in the same way. The exception is that one can define a Jordan algebra structure on $H_3(\OO^\CC)$ when $\OO^\CC$ is not even associative, but not on $H_n(\OO^\CC)$ for $n>3$. 
\end{remark}

\subsection*{Rank structure in $H_3(\AA)$}

Among Hermitian operators, one has Hermitian projections. Those of rank-one satisfy  $A^2=A$ and $\tr(A)=1$. Projectivizing, we get the variety
$$X_a^2:= \left\{ [A]\in \PP H_3(\AA), \quad A^2=\tr(A)A \right\}\subset \PP H_3(\AA).$$
Here $a$ is the complex dimension of $\AA$, while the meaning of the superscript $2$ will be clarified later on.
Now, Hermitian projections of rank one are in bijection with lines (even though over $\OO^\CC$, "lines" do not clearly make sense!). This allows to identify (at least formally for $\OO^\CC$) the varieties $X_a^2$ with projective planes over normed algebras:
$$X_a^2\simeq \AA\PP^2\subset  \PP H_3(\AA).$$
A dense open subset of $X_a^2$ is the projectivization of the image of the map
$$(x,y)\in\AA^2\mapsto \begin{pmatrix} 1&x&y \\
\overline{x}& x\overline{x}& y\overline{x} \\
\overline{y}& x\overline{y}& y\overline{y}\end{pmatrix}
\in H_3(\AA).$$

 The cone over $X_a^2$ parametrizes rank-one matrices inside $H_3(\AA)$. Summing two rank-one matrices yields a rank (at most) two matrix, characterized by the vanishing of the determinant
$$Det_a(M)=\frac{1}{3}tr(M^3)-\frac{1}{2}tr(M)tr(M^2)+\frac{1}{6}tr(M)^3.$$
This formula  makes sense even over $\OO^\CC$! Moreover, writing explicitly the determinant in coordinates, one checks that the derivatives of this cubic polynomial
vanish exactly on $\AA\PP^2$. Let  $C_a\subset \PP H_3(\AA)$ be the cubic hypersurface defined by $Det_a$. We get a stratification by the rank as follows:
$$ X_a^2=\Sing C_a \subset C_a \subset \PP H_3(\AA). $$
One also defines $SL(\AA)$ as the subgroup of $GL(H_3(\AA))$ that preserves the determinant $Det_a$. It preserves the hypersurface $C_a$, hence also its singular locus $X_a^2$, and one can show  that modding out by  the center, we get the full automorphism group of $X_a^2$.

Let us discuss what these varieties look like when $a=1,2,4$.

\begin{ex}[$a=1$: $\RR^\CC \PP^2$]
This is clearly the complex projective plane $\CC\PP^2$. What is interesting in this case is the embedding into the space of Hermitian matrices. We have seen that $H_3(\RR^\CC)$ is the space of symmetric matrices, and indeed the embedding $\CC\PP^2\subset \PP H_3(\RR^\CC)$ is the second Veronese embedding $v_2\colon \CC\PP^2 \to \CC\PP^5$. Moreover, $\CC\PP^2$ is the singular locus of $C_1$, which is defined by the usual determinant of symmetric matrices.
\end{ex}

Let us denote by $\AA\PP^{n-1}$ the variety parametrizing free rank-one $\AA$-modules inside $\AA^n$.

\begin{ex}[$a=2$: $\CC^\CC \PP^2$]
Recall that $H_3(\CC^\CC)$ is the space of complex $3\times 3$ matrices. In the projectivization of this space, rank-one matrices are parametrized by $\CC\PP^2\times \CC\PP^2$, embedded in $\CC\PP^8$ via the usual Segre embedding. Let us check that we can identify $\CC\PP^2\times \CC\PP^2$ with $\CC^\CC \PP^2$ (and more generally,  $\CC^\CC\PP^{n-1}$ with $\CC\PP^{n-1}\times \CC\PP^{n-1}$). 
Let $I$ be the complex structure on the real normed algebra $\CC$, and $i$ the complex structure coming from the complexification. Left multiplication $L_I$ by $I$ satisfies $L_I^2=-\id$, so we have two eigenspaces $(\CC^\CC)^n_i$ and $(\CC^\CC)^n_{-i}$ for $i$ and $-i$ of the same complex dimension $n$, and the algebra multiplication does not mix the two factors in the decomposition 
$(\CC^\CC)^n=(\CC^\CC)^n_{i} \oplus (\CC^\CC)^n_{-i}$. The first factor is the image of $L_{1-iI}$, while the second one is the image of $L_{1+iI}$. Given a free $\CC^\CC$-module $[E]\in \CC^\CC \PP^{n-1}$, one constructs two one dimensional subspaces $E_i:=E\cap (\CC^\CC)^n_i$ and $E_{-i}:=X\cap (\CC^\CC)^n_{-i}$. Vice versa, given any two $E_i$ and $E_{-i}$, one can reconstruct $E=E_i + E_{-i}$. This defines the isomorphism $\CC^\CC\PP^{n-1} \simeq \CC\PP^{n-1}\times \CC\PP^{n-1}$.  
\end{ex}

\begin{ex}[$a=4$: $\HH^\CC \PP^2$] Recall that $H_3(\HH^\CC)$ can be identified with $\extp^2 \CC^6$. We will now show that $\HH^\CC\PP^2\simeq \Gr(2,6)\subset \PP H_3(\HH^\CC)$ (and $C_4$ is the cubic hypersurface defined by the vanishing of the Pfaffian), more generally 
$\HH^\CC\PP^{n-1}\simeq \Gr(2,2n)\subset H_n(\HH^\CC)$. The idea  that since $\HH^\CC$ is associative, 
a Hermitian rank-one matrix inside $H_n(\HH^\CC)$ is of the form $v_2(a)=(\bar{a}_i a_j)$
for some $a\in \HH^n$ (note that $v_2(az)=|z|^2v_2(a)$ if $z\in\HH^\CC$). The image $L_a$ of $(\HH^\CC)^n$ by  right multiplication by $v_2(a)$ is $E(a):=\HH^\CC a\simeq \HH^\CC$ and has complex dimension four. In particular, it is stable by the left multiplication $L_I$ by $I$, whose eigenspaces $E_i(a)$
and $E_{-i}(a)$ are both two-dimensional, since they are exchanged by right 
multiplication by $J$. Notice in particular that $E_{-i}(a)$ is thus uniquely defined by the formula $E_{-i}(a)=JE_i(a)$ once $E_i(a)$ has been chosen. There is a  splitting  $(\HH^\CC)^n=(\HH^\CC)^n_i\oplus (\HH^\CC)^n_{-i}$, 
and the maps $[a] \mapsto E(a):=\HH^\CC a \mapsto E_i(a)$ define the isomorphisms 
$\HH^\CC\PP^{n-1}\simeq \Gr(2,2n)$. 
All of this is compatible with what happens for $\CC^\CC$, and we get a diagram
$$\xymatrix@1@=3pt@M=0pt{ \HH^\CC\PP^{n-1} & \simeq  & \Gr(2,2n) \\ \cup & & \cup \\ 
\CC^\CC\PP^{n-1} & \simeq  &\CC\PP^{n-1}\times \CC\PP^{n-1}.}$$
\end{ex}

\begin{ex}[$a=0$]
Let us include in the discussion the degenerate case $a=0$, obtained by letting $\AA=0$: we just get the space $H_3(0)$ of diagonal $3\times 3$  matrices with entries in $\RR^\CC\simeq \CC$. Then $\PP H_3(0)$ identifies with $\CC\PP^2$, and $X_0^2$  with the three points corresponding to diagonal matrices with only one non-zero entry.
\end{ex}

\subsection*{Playing around with Jordan algebras} 

One can think of the cubic $C_a$ as the projective dual hypersurface to $X_a^2$, parametrizing tangent hyperplanes to $X_a^2$ as points inside the 
dual projective space. Conversely the derivatives of the determinant yield a birational map 
\begin{equation}\label{eq_der_det}\xymatrix{  & I_a\ar[dl]_{Bl_{\AA\PP^2}}
\ar[dr]^{Bl_{\AA\PP^2_*}} & \\
\PP H_3(\AA) \ar@{-->}[rr]^{\partial Det_a} & & \PP H_3(\AA)^*.}\end{equation}

\begin{remark}
The derivatives of a polynomial $f$ over a vector space $V$ define a map $\partial f\colon V\to V^*$. Indeed, $\partial f$ takes a point $p\in V$ and a vector $x\in V$ (seen as a vector in the tangent space of $V$ at $p$) and gives a number $\partial_xf(p)$. So we get $0\in V^*$ when all the derivatives of $f$ vanish at $p$. In the projective version, this means 
that the indeterminacy locus 
of $\partial f\colon \PP(V)\dashrightarrow \PP(V^*)$ is the singular locus of the hypersurface $(f=0).$
 \end{remark}

\begin{ex}[$a=0$, continued]
The determinant on $H_3(0)$ is just $Det_0=xyz$, so that $\partial Det_0: \PP^2\dasharrow \check{\PP}^2$ is nothing else than the usual Cremona transformation $[x,y,z]\mapsto [yz,zx,xy]$.  In diagram (\ref{eq_der_det}), the triangle having the three indeterminacy points as vertices is blown-up to a hexagon, whose edges coming from the triangle are contracted  in  $\check{\PP}^2$. 
\end{ex}

Being generalizations of the Cremona transformation, it is no surprise that the other $\partial Det_a$ have very special properties: 
\begin{enumerate}
    \item Polynomials whose derivatives define birational maps are called \textit{homalo\"idal polynomials} and are very uncommon \cite{Dol_Cremona}. Here our homalo\"idal polynomial is a cubic, and by symmetry the inverse map is of the same type; so our birational transformation is quadro-quadric, and its indeterminacy locus $X_a^2$ is smooth and connected; it was proved in \cite{ein-shepherd-baron} that there is no other birational transformation with these properties.
    \item A fiber over $M\in\AA\PP^2_*$ is a copy of $\CC\PP^{a+1}$, that meets $\AA\PP^2$ along a quadric $Q_M\simeq\QQ^a$. These quadrics behave like lines (they are called $\AA$-lines) in a plane projective geometry:
    \begin{itemize}
    \item two (general) $\AA$-lines meet at a single point,
    \item two (general) points are joined by a unique $\AA$-line.
\end{itemize}
\item The four varieties $X_a^2=\AA\PP^2$ are the four {\it Severi varieties}: the only smooth subvarieties 
$X^{2m}\subset\PP^{3m+2}$
whose secant varieties are not the full space. According to Zak \cite{Zak_Severi}, this is impossible in $\CC\PP^n$ for $n<3m+2$. 
\end{enumerate}

\begin{remark}
In (2), the words "general" have to be added because we are working over the complex numbers, which makes some pathologies possible that would not exist over the real numbers (like isotropic octonions or null-planes). As a matter of fact, over the real numbers, "general" can be suppressed and one get a genuine plane projective geometry (with some subtle pathologies over the octonions) \cite{freudenthal-oktaven}.
 \end{remark}

\subsection*{Traceless Hermitian matrices}

There is a natural hyperplane of $H_3(\AA)$ given by traceless Hermitian matrices $ H_3(\AA)_0$. From $X_a^2=\AA\PP^2$ we can thus define the subvarieties $X_a^1:=X_a^2\cap  \PP H_3(\AA)_0$ by taking  hyperplane sections (again, the reason for the superscript $1$ will be explained later). 
Its automorphism group is $\Aut (H_3(\AA))$, that 
identifies with $SO_3(\AA)$. Indeed, $Aut (H_3(\AA))$ must fix the identity of the Jordan algebra $e\in H_3(\AA)$, hence also the quadratic form $Det_a(e,x,y)$, as well as the orthogonal $H_3(\AA)_0$ to $e$, and as a consequence also $X_a^1$. It turns out that $SO_3(\AA)$ and $SL(\AA)$ are semisimple Lie groups whose Dynkin diagrams can be read on the two following rows:
$$\begin{array}{cccc}
  \dynkin A2   & \dynkin A2 \vspace{5mm} \dynkin A2 &\dynkin A5& \dynkin E6 \\
   \dynkin A1   & \dynkin A2  & \dynkin C3& \dynkin F4
\end{array}$$
In terms of Lie algebras, the second row is deduced from the first one by  
{\it folding} the Dynkin diagram. Alternatively, the Lie algebras are related by 
a decomposition 
$$\fsl_3(\AA)=\fso_3(\AA)\oplus H_3(\AA)_0.$$

\begin{ex}[$a=1$]
If $a=1$ then $X_a^2$ is $\CC\PP^2$ embedded in $\PP^5$ by the second Veronese embedding $v_2$. So a general hyperplane section is a conic $v_2(\CC\PP^1)$. The Lie algebra of its automorphism group is $\fsl_2\simeq \fso_3=\fso_3(\RR^\CC)$ while $Lie (\Aut(X_a^2))=\fsl_3(\RR^\CC)$. The decomposition $\fsl_3(\RR^\CC)=\fso_3(\RR^\CC)\oplus H_3(\RR^\CC)_0$ is the decomposition of (traceless) matrices into skew-symmetric and symmetric parts.
\end{ex}

\begin{ex}[$a=2$]
If $a=2$ then $X_a^2$ is $\CC\PP^2\times \CC\PP^2$, so a hyperplane section is the flag variety $\Fl(1,2,3)$. The Lie algebra of the automorphism group of $\Fl(1,2,3)$ is $\fsl_3$ while $Lie (\Aut(X_a^2))=\fsl_3\oplus \fsl_3$ and we get a (non-diagonal!) decomposition $\fsl_3\times \fsl_3=\fsl_3\oplus \fsl_3$, where $H_3(\CC^\CC)_0$ is identified with $\fsl_3$, i.e., a hyperplane section of $\fgl_3$.
\end{ex}

\begin{ex}[$a=4$]
If $a=4$ then $X_a^2$ is $\Gr(2,6)\subset \extp^2 \CC^6$. A hyperplane section $\IGr(2,6)\subset \Gr(2,6)$ is defined by a general element $e\in \extp^2 (\CC^6)^*$, i.e., a symplectic form. The variety $\IGr(2,6)$ parametrizes those planes which are isotropic with respect to this symplectic form: it is a homogeneous variety under the action of $\Sp_6$. The Lie algebra of its automorphism group is $\fsp_6$ while $Lie (\Aut(X_a^2))=\fsl_6$ and we get a decomposition $\fsl_6=\fsp_6\oplus \extp^2 \CC^6$, i.e., again a decomposition of matrices in symmetric and skew-symmetric ones, but with inverted roles.
\end{ex}

\begin{remark}
One notices, maybe surprisingly, that in all the examples above, $\Aut(H_3(\AA))$ acts transitively on $X_a^1$. This is also the case for $a=8$, as we shall see.
\end{remark}

\begin{remark}
From what we have said one can deduce that $\Aut(H_3(\AA))$ can be seen as the stabilizer inside $\Aut(X_a^2)$ of a general element $e\in H_3(\AA)$. This implies that the complement of $X_a^2$ in $\PP H_3(\AA)$ is essentially the symmetric space $SL_3(\AA)/SO_3(\AA)$.
\end{remark}

\subsection*{The exceptional Jordan algebra}

We summarize in this short section what happens when $a=8$, i.e., when $\AA=\OO^\CC$. The first important result is the following:

\begin{theorem}\label{chevalley-F4}\cite{Chev_F4} $F_4=\Aut(H_3(\OO^\CC))$. 
\end{theorem} 


Note that $H_3(\OO^\CC)$ has dimension $27$ while, from the theory of simple Lie algebras, the Lie algebra  $\ff_4$  of $F_4$ has a unique minimal irreducible representation $V_{\omega_4}$ (with highest weight  $\omega_4$) of dimension $26$, all the non-trivial representations having higher dimension. This is in agreement with what we have seen so far! Indeed, $\Aut(H_3(\OO^\CC))$ must fix the identity in $H_3(\OO^\CC)$, and preserve the space $H_3(\OO^\CC)_0$ of traceless matrices, which identifies with the minimal irreducible representation of $F_4$. Moreover $X_8^1$ identifies with the minimal orbit $F_4/P_4$ inside $\PP V_{\omega_4}$: a $15$-dimensional complex variety, that we will also denote by $\OO^\CC\PP^2_0$.

Now, we know that $X_8^1$ is a hyperplane section of $X_8^2=\OO^\CC \PP^2$ and that $F_4=\Aut(H_3(\OO^\CC))$ is the stabilizer inside $\Aut(\OO^\CC \PP^2)$ of a general point $e\in H_3(\OO^\CC)$. What is this mysterious group $\Aut(\OO^\CC \PP^2)$? Well, it is a ``new", or rather exceptional, group, which turns out to be simple: the group $E_6$ \cite{Chev_F4}. Then one can show that $H_3(\OO^\CC)$ identifies with the irreducible $27$-dimensional $E_6$-representation $V_{\omega_1}$ whose highest weight is the fundamental weight $\omega_1$ (again, this is the non-trivial representation of $E_6$ of minimal dimension); this yields the interpretation of $\OO^\CC\PP^2$ as the minimal orbit $E_6/P_1$ inside $\PP (V_{\omega_1})$ (see Lemma \ref{lem_stable_contains_hom} and discussion that follows it). This variety is known as the (complex) \emph{Cayley plane}. 

\begin{remark}
The Weyl group of $E_6$ 
is (almost) the automorphism group of the configuration of $27$ lines on a smooth cubic surface, and this is no coincidence: J. Lurie \cite{Lurie} showed how to reconstruct the full
Lie algebra $\fe_6$, with its minimal representation, from the combinatorics of the  $27$ lines. In fact, let us identify the $27$ lines with coordinates $x_1,\dots,x_{27}$ on $\CC^{27}$. Let $\Delta$ be the set of triangles in the cubic surface and define $P(x)=\sum_{(ijk)\in \Delta}x_ix_jx_k$. Then in suitable coordinates,  $P$ is nothing else than $Det_8$!
\end{remark}

\begin{remark}
The action of $F_4$ on $\OO \PP^2$ is transitive over the
reals numbers (and the real Cayley plane can be identified with $F_4/\Spin_9$) but not over the 
complex numbers since it stabilizes $\OO\PP^2_0$. In fact the actions of $F_4$ on the latter and its complement are both transitive.
\end{remark}

\subsection*{The magic square from Jordan algebras}

Jordan algebras were studied a lot in the 1940-50's in connection with quantum mechanics
and its use of Hermitian operators. At this time, Chevalley's theorem \ref{chevalley-F4} urged the need to construct 
the other exceptional groups and algebras in terms of specific algebraic structures. This was accomplished in the  1950-60's by Tits, Freudenthal, Vinberg.

Consider a pair $(\AA, \BB)$ of (complexified, if you wish) normed algebras and let
$$\fg(\AA,\BB) :=  \der \AA\times   \der H_3(\BB) \oplus ({\mathrm Im}\AA\otimes H_3(\BB)_0), $$
where $ \der \AA$  and   $\der H_3(\BB)$ denote the spaces of derivations of these two algebras, i.e., linear operators satisfying the Leibniz rule. 
One can define a Lie algebra structure on $\fg(\AA,\BB)$, so that 
$\der\AA\times \der H_3(\BB)$ is  a Lie subalgebra acting on 
${\rm Im}\AA\otimes H_3(\BB)_0$ in a natural way.  
The result of this construction (over the complex numbers) is 
the {\it Tits-Freudenthal  magic square of Lie algebras}:

\begin{center}
\renewcommand{\arraystretch}{1.25}
\begin{tabular}{|c|cccc|} \hline 
 & $\RR^\CC$ & $\CC^\CC$ & $\HH^\CC$ & $\OO^\CC$ \\ \hline
$\RR^\CC$ &  $\fsl_2$ & $\fsl_3$ & $\fsp_6$ & $\ff_4$ \\  
$\CC^\CC$ & $\fsl_3$ & $\fsl_3\times \fsl_3$ & $\fsl_6$ & $\fe_6$ \\ 
$\HH^\CC$ & $\fsp_6$ & $\fsl_6$ & $\fso_{12}$ & $\fe_7$  \\
$\OO^\CC$ & $\ff_4$ & $\fe_6$ & $\fe_7$ & $\fe_8$ \\ \hline
\end{tabular}\end{center}

This provides a uniform description of the exceptional Lie algebras (except $\fg_2$!) over the complex numbers! Over the reals, one can use different forms of $\AA$ and $\BB$ (the split octonions, typically)  to obtain distinct real Lie algebras with the same complexification. The usual normed algebras always lead to compact forms.

\begin{remark}
Look carefully at the second line... can you see it? What about the first line?
\end{remark}

\section{Triality}

It  is amazing that  $\fg(\AA,\BB)$ has  a natural Lie
algebra structure, but also the symmetry of the square appears to be miraculous, since the definition of $\fg(\AA,\BB)$ is very non-symmetric. 
A more symmetric variant of the Tits-Freudenthal construction  was found by Vinberg 
in 1966, and rediscovered by several people (Allison, Dadok-Harvey, Barton-Sudbery, 
Landsberg-M.) using triality.
Triality groups generalize automorphism groups of normed algebras. They are defined as 
$$\begin{array}{lcl}
\Tri(\AA) & = & \left\{(U_1,U_2,U_3)\in SO(\AA)^3,\; U_1(xy)=U_2(x)U_3(y)\;\;\forall x,y\in\AA\right\}.
\end{array}$$
The three projections $\pi_i: \Tri(\AA)\lra SO(\AA)$ give three representations of the triality group, that we denote $\AA_1, \AA_2, \AA_3$. The triality Lie algebras
$$\ftri(\AA)\simeq  \left\{(u_1,u_2,u_3)\in \fso (\AA)^3,\; u_1(xy)=u_2(x)y+xu_3(y)\;\;\forall x,y\in\AA\right\}$$
for $a=1,2,4$ are $\ftri(\AA)=0, \;\mathfrak{ab}_2, \;\fsl_2^3$, where $\mathfrak{ab}_2$ is abelian and two-dimensional.

\begin{theorem}[Cartan 1925] 
The group $\Tri(\OO^\CC)\simeq \Spin_8$. Each projection 
$\pi_i$ is a twofold cover of $SO_8$. The three corresponding $8$-dimensional
representations of $\;\Spin_8$ are inequivalent. \end{theorem}

\begin{proof} See e.g. \cite[Theorem 14.19]{harvey}. \end{proof}

That we have (at least) double covers is clear since we can fix for example $U_1$ and change the 
signs of $U_2$ and $U_3$. 
If you don't know Spin groups, the Lie algebra version is also nice:
$$\fso_8\simeq  \left\{(u_1,u_2,u_3)\in \fso (\OO^\CC)^3,\; u_1(xy)=u_2(x)y+xu_3(y)\;\;\forall x,y\in\OO^\CC\right\}.$$

Triality is related to the fact that the  Dynkin diagram  $D_4$ is the only one with a 
threefold symmetry. The three exterior nodes correspond to the three $8$-dimensional representations 
$V_8$, $\Delta_+$, $\Delta_-$, and the symmetry means that these representations can be exchanged
by some outer automorphisms (that we can realize by permutations of $U_1,U_2,U_3$, and some 
conjugations.) In particular, these three representations are in fact indistinguishable.
(For higher spin groups only the two half-spin representations are indistinguishable.)

\setlength{\unitlength}{4mm}
\begin{picture}(30,7.2)(-10,3)
\put(3,5.8){$\circ$}
\put(.2,5.7){$\circ$}
\put(4.8,8.2){$\circ$}
\put(4.8,3.3){$\circ$}
\put(.5,6){\line(1,0){2.6}}
\put(3.35,6.1){\line(2,3){1.5}}
\put(3.35,5.9){\line(2,-3){1.5}}
\put(-.8,5.6){$\OO^\CC_1$}
\put(5.3,8.2){$\OO^\CC_2$}
\put(5.3,3.2){$\OO^\CC_3$}

\put(15,5.8){$\circ$}
\put(12.2,5.7){$\circ$}
\put(16.8,8.2){$\circ$}
\put(16.8,3.3){$\circ$}
\put(12.5,6){\line(1,0){2.6}}
\put(15.35,6.1){\line(2,3){1.5}}
\put(15.35,5.9){\line(2,-3){1.5}}
\put(11.2,5.6){$V_8$}
\put(17.3,8.2){$\Delta_+$}
\put(17.3,3.3){$\Delta_-$}
\put(14.4,5.3){$\fg_2$}
\end{picture}

\begin{remark}
Sometimes one starts with the 
construction of spin representations, proves the triality Theorem, and then deduces
that $G_2$ is the generic stabilizer for the transitive action of $\Spin_7$ on 
the seven dimensional sphere in $\Delta_+$.
\end{remark} 


\subsection*{The magic square from triality}
Let $(\AA, \BB)$ be normed algebras, with the three actions of their triality 
Lie algebras $\ftri(\AA)$ and $\ftri(\BB)$. There is a natural Lie algebra structure on 
$$\fg(\AA,\BB) = \ftri(\AA)\times \ftri(\BB)\oplus (\AA_1\otimes\BB_1)\oplus (\AA_2\otimes\BB_2)
 \oplus (\AA_3\otimes\BB_3).$$ 
Of course $\ftri(\AA)\times \ftri(\BB)$ is a Lie subalgebra, 
acting naturally on $\AA_i\otimes\BB_i$ \cite{LMTriality}. The bracket of 
$a_1\otimes b_1\in\AA_1\otimes\BB_1$ with $a_2\otimes b_2\in\AA_2\otimes\BB_2$ 
is simply
$a_1a_2\otimes b_1b_2$, considered as an element of $\AA_3\otimes\BB_3$. Now, by anti-linearity of the Lie bracket, this forces $[a_2\otimes b_2,a_1\otimes b_1]=-a_1a_2\otimes b_1 b_2$. For the other maps, one should be careful with signs: $[a_1\otimes b_1,a_3\otimes b_3]=-\overline{a}_1a_3\otimes \overline{b}_1 b_3\in \AA_2\otimes \BB_2$, $[a_2\otimes b_2,a_3\otimes b_3]=a_3\overline{a}_2\otimes b_3 \overline{b}_2\in \AA_1\otimes \BB_1$. Finally, if $c_i\in \AA_i$ and $d_i\in \BB_i$, then $[a_i\otimes b_i,c_i\otimes d_i]$ must belong to $\ftri(\AA)\times \ftri(\BB)$; to such an element we can associate $ q(b_i,d_i)a_i\wedge c_i +q(a_i,c_i) b_i\wedge d_i \in \extp^2 \AA_i\oplus \extp^2\BB_i$, where $q$ is the symmetric form. Then one obtains an element of $\ftri(\AA)\times \ftri(\BB)$ by applying the duals of the inclusion maps $\ftri(\AA)\to \fso(\AA_i)\simeq \extp^2\AA_i$. Note that 
by definition, $\fg(\AA,\BB)$ is graded over  $\ZZ_2\times\ZZ_2$, starting from 
$$\fg(\AA,\RR) = \ftri(\AA)\oplus \AA_1\oplus \AA_2 \oplus \AA_3.$$

\begin{remark}
Let $u,v\in \AA_1$ and denote by $\psi(u\wedge v)$ the corresponding element in $\ftri(\AA)$. Then the action on $\AA_1$ is given by the natural action of $\extp^2 \AA_1\simeq\fso(\AA_1)$: $\psi(u\wedge v)_1 x=4q(u,x)v-4q(v,x)u$. The action on $\AA_2$ and $\AA_3$ is instead given essentially by multiplication: $\psi(u\wedge v)_2= \overline{u}(vx)-\overline{v}(ux)$, $\psi(u\wedge v)_3= (xu)\overline{v}-(xv)\overline{u}$.
\end{remark}

\section{Exercises}

\begin{exercise}
    \necessary \label{ex_jordan_CCHH}
    Show that $H_3(\CC^\CC)$ is isomorphic to the space of $3\times 3$ complex matrices, and that $H_3(\HH^\CC)$ is isomorphic to the space of skew-symmetric $6\times 6$ matrices.
\end{exercise}

\begin{exercise}\training
\label{ex_der_det}
 Check that the first instance ($a=1$) of Diagram \eqref{eq_der_det} is well-known: 
$$\xymatrix{  & I_1\ar[dl]\ar[dr] & \\
 \CC\PP^5\ar@{-->}[rr] & & \CC\PP^5}$$
\end{exercise}

\begin{exercise}\curiosity
For $a=1,2,4$, observe that the quadrics that define "lines" in $X_a^2=\AA\PP^2$
can be seen as {\it entry loci}: any $P\in Sec(X_a^2)-X_a^2 $
belongs to infinity many secant lines, and the intersection points of these lines with $X_a^2$ are parametrized by an $a$-dimensional 
quadric $Q_P$. 
\end{exercise}

\begin{exercise}\training
\label{ex_tri_HH}
Compute $\ftri(\HH^\CC)$ and determine the three triality representations.
\end{exercise} 

\begin{exercise}\training
\label{ex_sl2_sl8}
Construct $\fso_8$ from four copies of $\fsl_2$ and decompose the triality representations. 
\end{exercise}  

\begin{exercise}\computation
\label{ex_geom_tri}
Check the geometric version of triality: any maximal isotropic subspace of $\OO^\CC$ is of the form $Im(R_p)$ or $Im(L_p)$ for some zero divisor $p\in\OO^\CC$. \end{exercise}

\begin{exercise}\computation
Check the Jacobi identity on $\fg(\AA,\RR)$. \end{exercise}

\begin{exercise}\training
\label{ex_root_e8}
Use this model to describe the root system of $\fe_8$. \end{exercise}

\medskip
\begin{sol}[Exercise \ref{ex_der_det}]
 $I_1$ is the blow-up of the Veronese surface in $\CC\PP^5$ and $\partial Det_1$ is given by the complete linear system of quadrics vanishing on this surface. The study of this blow-up was instrumental in Schubert's proof that there exist exactly $3264$ conics tangent to $5$ given general conics (see for instance \cite{EH-3264}). 
\end{sol}

\begin{sol}[Exercise \ref{ex_tri_HH}]
For $(0,u_2,u_3)\in \ftri(\HH^\CC)$, the triality equations become $u_2(x)y+xu_3(y)=0$, so $u_3$ is determined by $u_2$ (just let $x=1$). Thus elements of the form $(0,u_2,u_3)$ are contained in a copy of $\fso(\HH^\CC)\simeq \fso_4\simeq \fsl_2\times \fsl_2$. One checks that the triality equations define three supplementary equations on $\fso_4$ so that $\{(0,u_2,u_3)\in \ftri(\HH^\CC)\}$ form a (diagonal) subalgebra isomorphic to $\fsl_2$. Of course, putting to zero $u_2$ or $u_3$ instead of $u_1$ one gets two other copies of $\fsl_2$. These algebras commute since, for instance, $[(0,u_2,u_3),(v_1,0,v_3)]=(0,0,w_3)$ for some $w_3$, and the triality equations easily imply that $w_3=0$. Finally  $\ftri(\HH^\CC)=\fsl_2^3$.
\end{sol}

\begin{sol}[Exercise \ref{ex_sl2_sl8}]
(Sketch) One can construct the explicit decomposition from very natural morphisms, but one can also do it using graded Lie algebras. In particular, this example appears already in \ref{ex_graded_so8}. One obtains a decomposition $\fso_8=(\fso_8)_0\oplus (\fso)_1=(\fsl_2)^{\oplus 4} \oplus (\CC^2)^{\otimes 4}$. The space $(\fso)_1= (\CC^2)^{\otimes 4}$ is the tensor product of the four standard $\fsl_2$-representations corresponding to the four copies of $\fsl_2$. Each $\fsl_2$ corresponds to one of the four extremities of the affine Dynkin diagram of type $D_4$: three corresponding to simple roots at the extremes of the Dynkin diagram of $D_4$, and one for the highest root. The three triality representations are of the form $\CC^2_i\otimes\CC^2_j\oplus\CC^2_k\otimes\CC^2_l$ for $ijkl$ a permutation of $1234$.  
\end{sol}

\begin{sol}[Exercise \ref{ex_geom_tri}]
Recall that the zero divisors in a unital composition algebra are the elements of norm zero.  Let  $p\in\OO^\CC$ be us such a zero divisor. Then any element of the form $px$ is again a zero divisor, so  the image $Im(L_p)$ of left multiplication by $p$ is isotropic, in particular its dimension is at most $4$. The same holds for the kernel of $L_p$, which is again made of zero divisors and thus isotropic. So both are $4$-dimensional. We get a map from the quadric $\QQ^6\subset \PP V_8$ to the orthogonal Grassmannian $\OGr(4,V_8)\subset \PP \Delta_+$ given by $p\mapsto Im(L_p)$ (and similarly for right multiplication with $\Delta_-$ and $Im(R_p)$). This is an isomorphism since $\QQ^6$ and $\OGr(4,V_8)$ are both  six-dimensional quadrics.
\end{sol}

\begin{sol}[Exercise \ref{ex_root_e8}]
(Sketch) The Lie algebra of $\fe_8$ is obtained as
$$
\fe_8\simeq \fg(\OO^\CC,\OO^\CC)=\fso_8\times \fso_8 \oplus (V_8^{(1)}\otimes V_8^{(2)})\oplus (\Delta_+^{(1)}\otimes \Delta_+^{(2)})\oplus (\Delta_-^{(1)}\otimes \Delta_-^{(2)}).
$$
We deduce that two copies of the root system of type $D_4$ are contained in the root system of type $E_8$. The remaining roots come from the three remaining factors. For instance, let us denote by $\delta_1^{(i)}\dots \delta_8^{(i)}$ the weights of $V_8^{(i)}$ for $i=1,2$. Then $\delta_j^{(1)}+\delta_k^{(2)}$ are roots of $\fe_8$. All the roots of $E_8$ are obtained in this way ($48$ roots from the two copies of $\fso_8$ and $3\times 64=192$ from the three triality factors).
\end{sol}

\chapter{The magic square and open problems}

In the previous section, we have constructed the Magic Square from Jordan algebras and triality, together with some "natural" varieties $X_a^1$ and $X_a^2$. No we discuss the full geometric version of the magic square, making the varieties $X_a^1$ and $X_a^2$ appear in a quite natural fashion.

\section{Geometric version of the magic square}

This was developed mainly by Freudenthal in the 1950-60's, based on "points" 
in certain geometries \cite{freudenthal-oktaven, freudenthal-lie}, that are parametrized (over the complex numbers) by the following homogeneous spaces:

\begin{center}\begin{tabular}{|c|cccc|} \hline 
 & $\RR^\CC$ & $\CC^\CC$ & $\HH^\CC$ & $\OO^\CC$ \\ \hline
$\RR^\CC$ &  $v_4(\PP^1)$ & $Fl_3$ & $\IGr(2,6)$ & 
$\OO^\CC\PP^2_0$ \\  
$\CC^\CC$ & $v_2(\PP^2)$ & $\PP^2\times \PP^2$ & $\Gr(2,6)$ & $\OO^\CC\PP^2$ \\ 
$\HH^\CC$ & $\IGr(3,6)$ & $\Gr(3,6)$ & $\OGr(6,12)^+$ & $E_7^{hs}$  \\
$\OO^\CC$ & $F_4^{ad}$ & $E_6^{ad}$ & $E_7^{ad}$ & $E_8^{ad}$ \\ \hline
\end{tabular}\end{center}

The varieties in the first row are the varieties $X_a^1$, while the varieties $X_a^2$ are the varieties from the second row (mystery unveiled on the apices $1$ and $2$). The formers are hyperplane sections of the latters. Moreover, each variety is homogeneous for the action of a Lie group whose Lie algebra is $\fg(\AA,\BB)$. We can even complete it to a magic triangle \cite{deligne3} (the first line gives the parameter $a$):

\begin{center}\begin{tabular}{|cccccc|}  \hline 
  $ -\frac{2}{3}$  & 0 & 1 & 2 & 4 & 8 \\ \hline 
 & &  $v_4(\PP^1)$ & $Fl_3$ & $\IGr(2,6)$ & 
$\OO^\CC\PP^2_0$ \\  
 & 3pts& $v_2(\PP^2)$ & $\PP^2\times \PP^2$ & $\Gr(2,6)$ & $\OO^\CC\PP^2$ \\ 
 $v_3(\PP^1)$ & $(\PP^1)^3$ & $\IGr(3,6)$ & $\Gr(3,6)$ & $\OGr(6,12)^+$ & $E_7^{hs}$  \\
 $G_2^{ad}$ & $SO_8^{ad}$& $F_4^{ad}$ & $E_6^{ad}$ & $E_7^{ad}$ & $E_8^{ad}$ \\ \hline
\end{tabular}\end{center}

We will denote by $X_a^i$ the variety corresponding to the $i$th row. Note that each space $X_a^i=G/P$ from row $i$ and column $a$ comes with a homogeneous embedding inside $\PP V_a^i$ for some irreducible  $G$-representation $V_a^i$, whose dimension is given below: 
$$\begin{array}{lll}
\dim X_a^1=2a-1, &\qquad &  \dim V_a^1=3a+2, \\
\dim X_a^2=2a, & & \dim V_a^2=3a+3, \\
\dim X_a^3=3a+3, && \dim V_a^3=6a+8,\\
\dim X_a^4=6a+9, && \dim V_a^4=\dim\fg(\AA,\OO)=2\frac{(3a+7)(5a+8)}{a+4}. 
\end{array}$$
The last formula is the  most remarkable. To be able to express a dimension by a rational function is completely  unexpected! This formula was essentially discovered by Vogel \cite{vogel} from the construction of knot invariants, 
then expanded by Deligne and led to his conjectures on  {\it the exceptional series} of Lie algebras \cite{deligne1, deligne2}. 

We already know that $V_a^2=H_3(\AA)$, $V_a^1=H_3(\AA)_0$, $X_a^2=\AA\PP^2$ and $X_a^1=\AA\PP^2_0=\AA\PP^2\cap \PP H_3(\AA)_0$. There remains to understand the last two rows. For later use, let us explain some notation we are using.

\begin{definition}
\label{def_adjoint}
Let $G$ be a simple complex Lie group, acting on its Lie algebra $\fg$ via the adjoint representation. The image  inside $\PP\fg$ of the minimal non-trivial  nilpotent orbit  is called the \emph{adjoint variety} $G^{\ad}\subset \PP\fg$. Equivalently, it is the unique closed orbit inside $\PP\fg$.
\end{definition}

\begin{definition}
The variety $E_7^{hs}$ is the \emph{Freudenthal variety}, defined as the minimal $E_7$-orbit $E_7^{hs}=E_7/P_7\subset \PP V_{\omega_7}$ inside the projectivization of the representation $V_{\omega_7}$ of $E_7$ with highest weight  $\omega_7$ (its minimal representation, of dimension $56$).
\end{definition}


\section{The third row and varieties \texorpdfstring{$X_a^3$}{Xa3}}

Freudenthal defined for the third row a synthetic geometry modeled on symplectic geometry in five projective dimensions. The geometric elements are 
points, isotropic lines and isotropic  planes. Points are parametrized by adjoint varieties. Isotropic planes are parametrized by $X_a^3$  of dimension $3a+3$. 

$$\begin{array}{|l|c|c|c|c|} \hline
\mathrm{Points} &  \dynkin C{*oo}   & \dynkin A{*ooo*} & \dynkin   D{o*oooo} &\dynkin E{*oooooo} \\
\mathrm{Lines} &   \dynkin C{o*o}   & \dynkin A{o*o*o} & \dynkin D{ooo*oo} &\dynkin E{ooooo*o} \\
\mathrm{Planes} &  \dynkin C{oo*}   & \dynkin A{oo*oo} & \dynkin D{oooo*o} &\dynkin E{oooooo*} 
\\ \hline 
\end{array}$$

Up to isomorphisms, two different planes have only three possible relative positions. They are called {\it incident} when the corresponding points are joined by a line contained in $X_a^3$ (refer to Section \ref{sec_tits} for the description of projective linear spaces inside  homogeneous varieties in terms of Dynkin diagrams).

The connection between the second and the third row goes in both directions. The space of lines in $X_a^3$ passing
through a given point of $X_a^3$ is a copy of  $X_a^2$ (this is a nice application of the theory of Tits' shadows as presented in Section \ref{sec_tits}; this is also another instance of a VMRT, see Remark \ref{vmrt}). This also means that the Lie algebra of the stabilizer of any given point of $X_a^3$ contains a copy of $\fg(\AA,\CC)$, and that the isotropy representation on the tangent space is a copy of $V_a^2=H_3(\AA)$, containing $X_a^2$.

Conversely, we can reconstruct $X_a^3$ from $X_a^2$. 
Recall the latter is $\AA\PP^2$, 
the space of rank-one elements in 
$\PP H_3(\AA)$. This is also the singular locus of the space of elements of rank at most two, which is the cubic $C_a$ of equation $(Det_a=0)$. 

\begin{theorem}[Integrating the cubic, \cite{LMFreud}]
$X_a^3\subset \PP V_a^3$ is the  image of the cubic map 
$$M\in H_3(\AA) \mapsto [ 1, M, \partial Det_a(M), Det_a(M)] \in \PP (V_a^3),$$
where $V_a^3=\CC\oplus H_3(\AA)\oplus H_3(\AA)^*\oplus\CC^*$  admits a natural symplectic
structure, invariant under $\Aut(X_a^3)$ (\necessary Exercise \ref{ex_sympl_str}).
Moreover $X_a^3$ is Legendrian. 
\end{theorem}

\begin{remark}
Legendrian means that any  affine tangent space is maximal isotropic with respect to the symplectic structure  (i.e., Lagrangian).
Mukai coined these varieties {\it twisted cubics over Jordan algebras} \cite{mukai-sim}. 
\end{remark}

\begin{remark}
For $a=-\frac{2}{3}$, $M$ is just a scalar and we get a rational cubic in $\PP^3$:
$$X^3_{-\frac{2}{3}}=v_3(\PP^1), \hspace{2cm}  X^3_0=\PP^1\times\PP^1\times\PP^1.$$
\vspace{-5mm}
$$
\tikzset{/Dynkin diagram/fold style/.style={stealth-stealth,thick,shorten <=1mm,shorten >=1mm,}}
\dynkin[edge length=.75cm]G2 \hspace{3cm} \dynkin[ply=3,edge length=.75cm]D4
$$

This is a geometric incarnation of folding! Moreover the twisted cubic curve has the well-known property that its tangent developpable is a quartic surface in $\PP^3$. 
For $\PP^1\times\PP^1\times\PP^1\subset \PP (M_{2,2,2})$, where $M_{2,2,2}=\CC^2\otimes\CC^2\otimes\CC^2$ is the space of cubic matrices, we also have an invariant quartic polynomial given by {\it Cayley's hyperdeterminant}
$$\begin{array}{rcl}
HDet(A) & =&  a_{000}^2a_{111}^2 + a_{001}^2a_{110}^2 + a_{010}^2a_{101}^2 + a_{100}^2a_{011}^2 
 \\
 & & -2a_{000}a_{001}a_{110}a_{111}-2a_{000}a_{010}a_{101}a_{111}-2a_{000}a_{011}a_{100}a_{111} \\ 
 & & -2a_{001}a_{010}a_{101}a_{110}-2a_{001}a_{011}a_{110}a_{100} -2a_{010}a_{011}a_{101}a_{100} \\
 & & + 4a_{000}a_{011}a_{101}a_{110} + 4a_{001}a_{010}a_{100}a_{111}.
 \end{array}$$
 Hyperdeterminants are equations of dual varieties, but here this  is the same as the tangent variety,
 because of the Legendrian property. 
\end{remark}

Remarkably, these properties propagate to the whole row $X_a^3\subset \PP V_a^3$. 
The tangent variety of $X_a^3$ is always a quartic hypersurface, with a singular locus $W_a\supset X_a^3$ of dimension $5a+4$. 
The equation $HDet_a$ of the  quartic hypersurface admits a uniform expression, independently of $a$. 
The string 
\begin{equation}
\label{eq_strat}
    X_a^3\subset W_a\subset Tan(X_a^3)\subset \PP V_a^3
\end{equation}
is the full stratification into orbit closures of $\Aut(X_a^3)$ (Exercise \ref{ex_strat}). 
Let us conclude this section with  one last property that the varieties $X_a^3$ share. 

\begin{definition}
A projective variety $X\subset \PP (V)$ has the One Apparent Double Point property (OADP) if through any general point $x$ in $\PP (V)$ passes a unique secant line to $X$.
\end{definition}

\begin{theorem}\cite[Proposition 8.4]{clerc}
$X_a^3$ has the OADP property.\end{theorem}


\begin{remark}
In the first lecture we have already seen that $\Gr(3,6)\subset \PP (\extp^3 \CC^6)$ has the OADP property: a general trivector in six variables can be written as the sum of two pure trivectors in a unique way.
\end{remark}

\section{The fourth row and varieties \texorpdfstring{$X_a^4$}{Xa4}}

Freudenthal's so-called metaplectic geometries describe the fourth row. There are four types of elements in such a geometry, called points, lines, planes and symplecta, and they are parametrized by homogeneous spaces. (Recall that such a homogeneous space is defined by a subset of vertices of the Dynkin diagram - the black dots in the pictures below.) 
All the varieties in the fourth row are modeled on the geometry of $F_4$ and the four $F_4$-Grassmannians. 

$$\begin{array}{|l|c|c|c|c|} \hline
\mathrm{Points} &  \dynkin F{ooo*}   & \dynkin E{*oooo*} & \dynkin   E{ooooo*o} &\dynkin E{*ooooooo} \\
\mathrm{Lines} &   \dynkin F{oo*o}   & \dynkin E{oo*o*o} & \dynkin E{ooo*ooo} &\dynkin E{ooooo*oo} \\
\mathrm{Planes} &  \dynkin F{o*oo}   & \dynkin E{ooo*oo} & \dynkin E{oo*oooo} & \dynkin E{oooooo*o}\\
\mathrm{Symplecta} & \dynkin F{*ooo} & \dynkin E{o*oooo} & 
\dynkin E{*oooooo} & \dynkin E{ooooooo*} \\ \hline
\end{array}$$

\begin{remark}
Let $G$ be a simple Lie group and $P$ a parabolic subgroup. Since $P$ is its own normalizer, there is a canonical interpretation of the homogeneous space $G/P$ as a conjugation class of parabolic subgroups. Now, there is a natural incidence relation between parabolic subgroups: simply ask that their intersection contains a Borel subgroup. This point of view, that generalizes Freudenthal's 
constructions, was widely 
expanded by Tits, who discussed the properties of the incidence geometries one obtains this way \cite{tits-excgeom}. This ultimately led to the revolutionary notion of Tits building.
\end{remark}

By fixing a point and looking at lines-planes-symplecta through this point,  
one obtains the symplectic geometries from the third line. 
In particular, the space of lines in $X_a^4$ passing
through a given point is isomorphic to  $X_a^3$ (refer again to Section \ref{sec_tits}). This also means that $V_a^3$ can be seen as a subspace of the tangent space at a point of $X_a^4$. In fact, we will see that $V_a^3$ is a hyperplane in this tangent space. Notice also that each simplex is an adjoint variety for the Lie algebras $\fg(\AA,\HH)$ of the third line.

\subsection*{Contact structures}
Recall that the adjoint variety $G^{ad}$ can be defined as the unique closed $G$-orbit inside $\PP\fg$, 
where $\fg$ is a simple complex Lie algebra and $G=\Aut(\fg)$. 

 \noindent {\it Example.} $A_n^{ad}\simeq \Fl(1,n,n+1)$ is a variety
of partial flags. At a point $f=(V_1\subset V_n)$ of  $Fl:=A_n^{ad}$, the cotangent space is 
$$\Omega_{Fl,f}= \left\{X\in\fsl_{n+1}, \quad X(\CC^{n+1})\subset V_n, \quad X(V_n)\subset V_1, \quad X(V_1)=0\right\}.$$
It contains the line $L^*_f=\{X\in\fsl_{n+1}, \; X(\CC^{n+1})\subset V_1, \; X(V_n)=0\}.$ 

Dualizing, an letting $f$ vary in $Fl$, we get an exact sequence of vector bundles 
$$0\raar \cH \raar T_{Fl}\raar L\raar 0,$$
where $T_{Fl}$ is the tangent bundle. The subbundle $\cH$ defines a {\it contact distribution}, i.e., 
a distribution 
of tangent hyperplanes 
which is maximally non integrable, as in the following definition:

\begin{definition}
A (holomorphic) \emph{contact structure} on $X$ is an exact sequence of vector bundles $$ 0\to \cH\to T_X \stackrel{\theta}\to L\to 0,$$
where $L$ is a line bundle and the restriction of $d\theta$ to $\extp^2 \cH$ is everywhere nondegenerate.
\end{definition}
This means that the induced linear map
$$\extp^2\cH\hookrightarrow\extp^2T_{X}\stackrel{d\theta }{\lra} T_{X}\lra L$$
defines at every point a non degenerate skew-symmetric form on $\cH$. 
In particular $X$ must have odd dimension $2n+1$, and then $K_X=-(n+1)L$
\cite{beauville-contact}.

The contact structure on $A_n^{ad}$ extends to any adjoint variety $G^{ad}$. It comes from the famous Kostant-Kirillov-Souriau symplectic
form on a coadjoint orbit $\cO\subset\fg^*$ (see e.g. \cite{beauville-contact}). 
If $\lambda\in\cO$ has stabilizer $H\subset G$, then 
$T_\lambda\cO\simeq \fg/\fh$ and one can let 
$$\omega_\lambda(\overline{X}, \overline{Y})=\lambda([X,Y])$$
for $X,Y\in\fg$ and $\overline{X}, \overline{Y}$ the corresponding tangent vectors.
This is a well-defined two-form, closed and nondegenerate: coadjoint orbits are (open) symplectic manifolds.

When $\fg$ is complex semisimple, $\fg^*\simeq\fg$ has finitely many nilpotent orbits (see Section \ref{sec_nilp_orbits}), 
which are (open) holomorphic symplectic manifolds.  In particular the nilpotent cone is a symplectic
singular variety $\cN$, which can be resolved by the Springer resolution $\Tot(\Omega_{G/B})\raar\cN$. At the other extreme there is a unique {\it minimal} nilpotent orbit $\cO_{min}$, 
 and by definition $G^{ad}=\PP\cO_{min}.$ The contact structure on $G^{ad}$ is the projective version of the symplectic structure on  $\cO_{min}.$
 
One can show that any line in $G^{ad}$ is a contact line, i.e., its tangent vector is contained in $\cH$. This explains the discrepancy between 
$$\dim X_a^8=6a+9 \qquad \mathrm{and} \qquad \dim V_a^3=6a+8.$$

\begin{remark}[Going upwards in the magic square]
Through this remark we summarize the connection between the different rows of the magic square. Let us start from the last row and the variety $X_a^4=G^{ad}$. Since it is a contact manifold, the tangent bundle of $X_a^4$ contains a subbundle $\cH$, which is endowed with a symplectic form with values in $L$. At each point of $X_a^4$, the Lie algebra of the stabilizer group contains a copy of $\fg(\AA,\HH)$, and we can identify its actin on the fiber of $\cH$ with the symplectic vector space $V_a^3$. Moreover,  the lines contained in $X_a^4$ and passing through the point can be identified with $X_a^3\subset \PP V_a^3$; this is another example of a nice VMRT \ref{vmrt}. Because of this close connection with adjoint varieties, the varieties 
$X_a^3$ are sometimes called \emph{subadjoint} varieties (see e.g. \cite{hwang-legendrian}).  

In fact, for each point of $X_a^3$ there is a decomposition of  $V_a^3$ as a direct sum $\CC \oplus H_3(\AA)\oplus H_3(\AA)^* \oplus \CC^*$, with the tangent space of $X_a^3$ at the point identified with $H_3(\AA)$. Once again, looking at lines inside $X_a^3$ passing through the point one obtains the subvariety $X_a^2\subset \PP H_3(\AA)$, which is the space of rank-one Hermitian matrices $\AA\PP^2$. Finally, $X_a^1$ is the hyperplane section of $X_a^2$ defined by intersecting with the hyperplane of traceless matrices $H_3(\AA)_0\subset H_3(\AA)$.
\end{remark}

\begin{remark}  
One can add a further column to the Magic Square, with $a=6$. Indeed, recall from Exercise \ref{sextonions} that the orthogonal to any null-plane in the octonions is a six-dimensional subalgebra $\SSS$ of $\OO^\CC$, called the sextonion algebra. Then we can define 
$\Tri(\SSS)$ and $\fg(\AA,\SSS)$ as if $\SSS$ was a genuine normed algebra. 
In particular 
$\fg(\OO,\SSS)$ is a (non-semisimple) Lie algebra of dimension $190$, coined $E_{7\frac{1}{2}}$ \cite{LM-sext}; its existence had already been guessed by string theorists. Remarkably, Freudenthal geometries can also be constructed for $\SSS$, with the major difference that the parameter spaces are no longer homogeneous, and in fact become singular.  See also \cite{westbury}. 
\end{remark}

\section{Exercises} 

\begin{exercise}\necessary
\label{ex_sympl_str}
Guess the symplectic structure! And deduce that $X_a^3$ is {\it Legendrian}: each 
affine tangent space is a Lagrangian subspace of $V_a^3$. 
\end{exercise}

\begin{exercise}\training
\label{ex_strat}
Prove that stratification \eqref{eq_strat} is the full stratification into orbit closures of $\Aut(X_a^3)$ by observing that the tangent bundle is a bundle in Jordan algebras,
in which we know the orbits. For $a=2$ describe three desigularisations of the intermediate orbit.
\end{exercise} 

\begin{exercise}\training
\label{ex_cont_str}
When $\fg$ is classical, describe its adjoint variety with its
contact structure.
\end{exercise}

\medskip
\begin{sol}[Exercise \ref{ex_sympl_str}]
Let $V$ be a complex vector space. There is a natural symplectic structure $\omega$ on $V\oplus V^*$ given by $\omega((v_1+w_1,v_2+w_2)=w_1(v_2)-w_2(v_1)$. By applying this to $V=\CC\oplus H_3(\AA)$ one deduces a symplectic structure on $V_a^3$ for which $\CC\oplus H_3(\AA)$ is isotropic. One checks that this can be arranged so that the symplectic structure is invariant under $\fg(\AA,\HH)$. Then to ensure that $X_a^3$ is Legendrian, by homogeneity it suffices to check that the Zariski tangent space at one chosen point is isotropic. Of course one chooses this point to be the image of the matrix $0\in H_3(\AA)$ via the above cubic map, i.e., the point $[1,0,0,0]$. Since $Deta_a$ is a cubic map and $\partial Det_a$ is quadratic, both vanish on the Zariski tangent space of $X_a^3$ at $[1,0,0,0]$, which therefore can be identified with $\CC\oplus H_3(\AA)\subset V_a^3$. But then it is clear that this subspace is isotropic. 
\end{sol}

\begin{sol}[Exercise \ref{ex_cont_str}]
We have already seen the case of $A_n$, so we are left with the symplectic and orthogonal groups.
\begin{itemize}
    \item Let $G=Sp_{2n}=Sp(\omega)$ where $\omega\in \extp^2(\CC^{2n})^*$ is symplectic, and $\fg=\fsp_{2n}\simeq S^2 \CC^{2n}$. Any vector in $\CC^{2n}$ is isotropic with respect to $\omega$ and $\PP^{2n-1}$ is a $Sp_{2n}$-homogeneous space which embeds inside $\PP S^2 \CC^{2n}$ via the second Veronese embedding $v_2$; it is the adjoint variety: $Sp_{2n}^{ad}=v_2(\PP^{2n-1})$. The tangent bundle of $\PP^{2n-1}$ at $x$ is $Hom(x,\CC^{2n}/x)$ and contains the contact hyperplane $Hom(x,x^\perp/x)$. 
    This yields a short exact sequence 
    $$0\to Hom(\cO(-1), \cO(-1)^\perp/\cO(-1)) \to T_{\PP^{2n-1}}\to \cO(2)\to 0.$$
     The contact bundle  $ \cH=\cO(-1)^\perp/\cO(-1)$  is called the  null-correlation bundle. 
    \item Let $G=SO_{m}=SO(q)$ where $q\in S^2(\CC^{m})^*$ is symmetric nondegenerate, and $\fg=\fso_{m}\simeq \extp^2 \CC^{m}$. So we need to search for the adjoint variety of $SO_m$ a the minimal orbit inside $\PP(\extp^2 \CC^m)$. Clearly $\Gr(2,m)\subset \PP(\extp^2\CC^m)$ is $SO_m$-invariant, but it is not homogeneous. The adjoint variety is $SO_m^{ad}=\OGr(2,m)\subset \Gr(2,m)$, the Grassmannian of isotropic two-dimensional subspaces. It can as well be seen as the zero locus of $q$ seen (by restriction) as a section of $S^2\cU^*$ over $\Gr(2,m)$, if $\cU$ denotes the tautological rank two bundle. Since the tangent bundle of $\Gr(2,m)$ is $\cU^*\otimes  \cQ$, for $\cQ$ the quotient bundle, 
    the normal exact sequence on $\OGr(2,m)$ reads:
    $$ 0\to T_{\OGr(2,m)} \to \cU^*\otimes \cQ \to S^2 \cU^* \to 0.$$
    Now $q$ gives an isomorphism between $\CC^{m}$ and its dual, and this isomorphism sends an isotropic 2-dimensional subspace $x\subset \CC^m$ into its dual $x^*\simeq \CC^m/x^\perp$ (where $\perp$ is with respect to $q$). So there is an exact sequence $0\to \cU^\perp/\cU\to \cQ\to \cU^*\to 0$, which induces an exact sequence $0\to \cU^*\otimes \cU^\perp/\cU \to \cU^*\otimes \cQ \to \cU^*\otimes \cU^*\to 0$. The normal map $\cQ\otimes \cU^* \to S^2\cU^*$ restricts to zero on $\cU^*\otimes \cU^\perp/\cU$ and is induced by the natural projection $\cU^*\otimes \cU^* \to S^2\cU^*$, whose kernel is $\extp^2 \cU^* \simeq \cO(1)$. The upshot is the exact sequence
    $$0\to \cU^*\otimes \cU^\perp/\cU \to T_{\OGr(2,m)} \to \cO(1)\to 0 ,$$
    which gives the contact structure on $\OGr(2,m)$, with $L=\cO(1)$ and $\cH=\cU^*\otimes \cU^\perp/\cU$.
\end{itemize}
\end{sol}

\section{Open problems}

We list a few open problems related to the lectures. 

 \noindent {\bf Problem 1.} Hartshorne's conjecture: any smooth subvariety of $\PP^N$ of dimension $n$ and codimension $c$ is a complete intersection if $c<n/2$.

 \noindent {\bf Problem 1bis.} Hartshorne's conjecture in codimension $2$: equivalently (by Serre's construction) any rank two vector bundle on $\PP^N$ is decomposable for $N\ge 6$. 

 Major problem in the 1980's, but zero progress except Zak's theorem, and most people gave up \cite{laz-vdv, Zak_Severi}. 

 \noindent {\bf Problem 2.} Classify smooth subvarieties of $\PP^N$ with codegree four. Prove they "essentially" all come from the third row of the Magic Square. 
(see \cite{fu-liu}). 

 \noindent {\bf Problem 3.} Classify varieties with the OADP property (see \cite{ciliberto-oadp}).
 
 \noindent {\bf Problem 4.} Classify smooth Legendrian varieties.

 Buczynski showed how to construct many of them by cutting a given one by a linear space of codimension two. He also 
gave conditions that imply homogeneity. See \cite{buczynski} for more on this topic. The problem of characterizing homogeneous Legendrian 
varieties by intrinsic properties was recently addressed in \cite{hwang-legendrian}. 

 \noindent {\bf Problem 5.} LeBrun-Salamon conjecture: smooth (complex) contact manifolds with Picard number one of Fano type are adjoint varieties of simple Lie algebras 
 \cite{lebrun-salamon}. 

 Complex contact manifolds are twistor spaces of manifolds with holonomy $Sp(1)Sp(2n)$ in Berger's classification, also called quaternion-K\"ahler manifolds. 
So this is another holonomy class than $G_2$-manifolds! The Fano condition corresponds to asking that the quaternion-K\"ahler manifold has positive scalar curvature.

 It is known that contact manifolds with Picard number $>1$ are projectivized cotangent bundles \cite{kpsw}. This includes adjoint varieties of type $A$. 

 \noindent {\bf Problem 6.} Prove that the VMRT of such a contact manifold is a homogeneous Legendrian variety.

 That the variety of minimal rational tangents (VMRT) at a generic point is a smooth Legendrian variety was proved by Kebekus \cite{kebekus2}. 

 \noindent {\bf Problem 7.} Classify special Cremona transformations. 

Special Cremona tranformations are birational maps $\PP^n\dasharrow\PP^n$ 
that become regular after blowing-up a smooth connected subvariety. Those 
defined by quadratic polynomials, as well as their inverse (quadro-quadric transformations) were classfied in  \cite{ein-shepherd-baron}. 

 \noindent {\bf Problem 8.} Find nice projective embeddings of the homogeneous spaces 
$$E_6/PSp_8, \qquad E_7/PGL_8, \qquad E_8/Spin_{16}.$$  

The last one, of dimension $128$, potentially qualifies over the real numbers as a projective plane over $\OO\otimes\OO$,
whatever that means. Over the complex numbers no nice projective embedding has been described. 

\chapter{Appendix}

\section{Elements of projective duality}\label{app:proj-duality}

In this appendix we summarize some elements of projective duality used in the text. For further details we refer to \cite{GKZ} and \cite{Tevelev-PD}. Let $V\simeq \CC^{n}$ a vector space and $\PP(V)$ the projective space of one-dimensional subspaces of $V$. For any subspace $U\subset V$, the projectivization 
$\PP(U)\subset \PP(V)$ is a linear subspace. We denote by $V^*$ the dual of $V$ and $\cPP(V) = \PP(V^*)$;
a point in $\cPP(V)$ corresponds to a hyperplane in $\PP(V)$. For $L = \PP(U)\subset \PP(V)$ linear, its \emph{projective dual} $L^* = \PP(\ann(U)) \subset \cPP(V)$, is the  set of hyperplanes containing $L$. Note that $\dim L^* = \dim V - \dim L$ and $(L^*)^* = L$. This notion of duality extends to algebraic subvarieties, as we will now explain.  

Let $X \subset \PP^n$ be a closed subvariety; we denote by $X_\sm$ the smooth locus of $X$, which is open and dense in $X$. For a point $x\in X_{\sm}$,  the \emph{embedded tangent space} is 
\[
\hat{T}_xX = \PP( T_vC(X))\subset \PP^n,
\]
where $C(X)\subset \CC^{n+1}$ is the affine cone of $X$ and $v \in x$ is any nonzero vector. A hyperplane $H\subset \PP^n$ is \emph{tangent} to $X$ if is contains $\hat{T}_xX$ for some $x\in X_\sm$.

\begin{definition}\label{def:dual-of-X}
The \emph{projective dual variety of $X\subset \PP^n$} is the variety $X^\vee\subset \check{\PP}^n$ defined as the Zariski closure of the set of tangent hyperplanes:
\[
X^\vee:=\overline{\left\{\,H\in \check{\PP}^{n} \mid \exists\, x\in X_{\sm} \mbox{ s.t. }\hat{T}_x X\subset H \,\right\}}.
\]
\end{definition}

Consider the incidence variety $\con_X^0 = \{\, (x,H) \in \PP^n\times\cPP^n \mid x\in X_\sm \text{ and } H \supset \hat{T}_xX \,\}$. It sits inside the projectivized cotangent bundle $\PP (T^*\PP^n)\subset \PP^n\times\cPP^n$ since the latter is the locus of pairs $(x,H)$ such that $x\in H$. Moreover, the projection $\pi_1 \colon \con_X^0 \to X_\sm$ shows that $\con_X^0$ is the projectivized conormal bundle $\PP(N^*_{X_\sm/\PP^n})$. The closure of the image of the second projection $\pi_2 \colon \con_X^0 \to \cPP^n$ is precisely $X^\vee$. The Zariski closure $\con_X$ of $ \con_X^0$ is called the \emph{conormal variety} of $X$. It has many useful properties that shed light onto the geometry of the dual variety. One important result is the following.

\begin{theorem}[Biduality Theorem]\label{thm:biduality}
For any projective variety $X\subset \PP^n$, the bidual 
\[
(X^\vee)^\vee = X.
\]
Moreover, if $z$ is a smooth point of $X$ and $H$ is a smooth point of $X^\vee$, then $H$ is tangent to $X$ at $z$ if and only if $z$, regarded as a hyperplane in $\cPP^n$,  is tangent to $X^\vee$ at $H$.
\end{theorem}

\begin{proof}
For a proof, see \cite[I, Theorem 1.1]{GKZ}.
\end{proof}

\begin{theorem}
Let $X\subset \PP^n$ be an irreducible projective variety. Then
\begin{enumerate}
    \item If $X$ is smooth then $\con_X$ is smooth.
    \item If $X^\vee$ is a hypersurface then $\pi_2$ is birational.
    \item If $X$ is smooth and $X^\vee$ is a hypersurface, then $\pi_2$ is a resolution of singularities. 
\end{enumerate}
\end{theorem}

\begin{proof}
See \cite[Theorem 1.15]{Tevelev-PD}.
\end{proof}

Typically, $X^\vee \subset \cPP^n$ is a hypersurface,
defined by a polynomial $\Delta_X$ called 
the \emph{discriminant} of $X$.

To compute the codimension and the degree of $X^\vee$, at least when $X$ is smooth, one may follow Cayley's method, that consists in realizing $\con_X$ as the degeneracy locus of a section of a twisted jet bundle, see \cite[Chapter 2]{GKZ}. In particular, this leads to the following powerful result of Katz and Kleimann, later generalized by Holme.

\begin{theorem}\label{thm:KKH}
Let $X\subset \PP^n$ be an irreducible smooth projective variety. Consider the polynomial 
\[
c_X(q) = \sum_{i=0}^{\dim X} q^{i+1}\int_X c_{\dim X -i}(\Omega^i_X)\cdot c_1(\cO_X(1))^i.
\]
Then $\codim X^\vee = \mu$ is the maximum power of $q-1$ that divides $c_X(q) -c_X(1)$, and $\deg X^\vee = c_X^{(\mu)}(1)/ \mu!$
\end{theorem}

\begin{proof}
See \cite[II, Theorem 3.4]{GKZ}.
\end{proof}

\begin{ex}\label{ex:deg dual G(3,7)}
Let $X = \Gr(3,7)\subset \PP^{34}$ the Grassmannian of $3$-dimensional subspaces of $\CC^7$ in its Pl\"ucker embedding. One may compute that
\[
\begin{split}
   c_X(q) =  462\,q^{13}-3\,234\,q^{12}+10\,836\,q^{11}-22\,932\,q^{10}+34\,062\,q^{9}-37\,338\,q^{8} \\ +30\,952\,q^{7}-19\,600\,q^{6}+9\,492\,q^{5}-3\,500\,q^{4}+980\,q^{3}-210\,q^{2}+35\,q,
\end{split}
\]
hence $c_X'(1) = 7$. Therefore, $X^\vee$ is a degree seven hypersurface. 
\end{ex}

In the context of actions of algebraic groups, the biduality theorem has also some strong consequences. In \cite{PY}, Pyasetski\u{\i} has observed that if a connected algebraic group acts linearly on a vector space with finitely many orbits, then the dual representation has the same property; moreover the numbers of orbits are the same.

\begin{theorem}\label{thm:Py}
Suppose that a connected algebraic group $G$ acts linearly on a vector space $V$ with finitely many orbits. Then the dual action on $V^*$ has the same number of orbits. Let 
\[
V = \bigsqcup_{i=1}^N\cO_i \quad \text{and} \quad V^* = \bigsqcup_{i=1}^N \cO_i'
\]
be the orbit decompositions. Then the bijection is defined as follows: $\cO_i$ corresponds to $\cO_j'$ if and only if $\PP(\overline{\cO_i})$ is projectively dual to $\PP(\overline{\cO_j'})$.
\end{theorem}

\begin{proof}
The argument is to first show that the orbits in $V$ are affine cones, and so are the orbits in $V^*$; here it is used that there are finitely many orbits. Then the Biduality Theorem is applied to the projectivizations of the orbit closures. For details, see \cite[\S2.2]{Tevelev-PD}.
\end{proof}

\begin{ex}\label{ex:G36}
Consider the action of $\GL_6$ on $\extp^3 \CC^6$. As shown in 
\S\S~\ref{sec:fewvars}, this action has four orbits 
beside the origin $\cO_0 = \{0\}$: $\cO_1$ consisting of decomposable tensors (the affine cone of $\Gr(3,6)$); then the set $\cO_2$ of rank-five forms; finally $\cO_3$ and $\cO_4$ consisting of rank-six forms, $\cO_4$ being the dense orbit and $\overline{\cO_3}$  the affine cone over the tangential variety to $\Gr(3,6)$. Identifying $\CC^6$ with its dual, Projective Duality sends $\cO_0$ to $\cO_4$, $\cO_1$ to $\cO_3$, while the intermediate orbit closure $\cO_2$ is self-dual. 
\end{ex}

\section{Gradings of Lie algebras}
\label{sec_gradings}

Gradings of Lie algebras have appeared in an explicit (gradings of $\fg_2$ and $\fso_8$) or more intrinsic way (prehomogeneous spaces) during these lectures. Let us give a very concise overview of the subject, hopefully of practical use. We refer to \cite{Vinberg_graded, Kimura_prehom, Manivel_prehomogeneous} for more details. We will restrict to a simple situation, even though everything can be done more generally. 

\subsection*{\texorpdfstring{$\ZZ$}{Z}-Gradings}

Let $\fg$ be the complex Lie algebra of an algebraic group $G$.

\begin{definition}
A $\ZZ$-grading of $\fg$ is a decomposition $\fg=\bigoplus_{i\in \ZZ} \fg_i$ such that $[\fg_i,\fg_j]\subset \fg_{i+j}$. In particular $\fg_0$ is a subalgebra and $\fg_i$ is a $\fg_0$-representation for all $i$.
\end{definition}

Suppose $\fg$ is semisimple, fix a Cartan subalgebra $\fh$ and denote by $\Delta \subset \fh^*$ the set of roots of $\fg$. The elements in $\fh$ are commuting semisimple elements and thus their action on $\fg$ can be diagonalized; this yields the so-called Cartan decomposition $\fg=\fh \oplus \bigoplus_{\alpha\in \Delta}\fg_\alpha$, where the $\fg_\alpha$'s are one-dimensional and, if $h\in \fh$ and $x\in \fg_\alpha$, $[h,x]=\alpha(h) x$. Recall that one can choose a preferred basis of roots $\alpha_1,\dots,\alpha_n \in \Delta$ called \emph{simple} roots; their relative position (from which $\fg$ can be recovered) is encoded in the Dynkin diagram of $\fg$, whose vertices correspond indeed to simple roots.

\begin{ex}[Main example]
\label{main_ex_grading_Z}
Define a grading on $\fg$ as follows: choose  $h\in \fh$ such that $\alpha(h)\in \ZZ$ for all $\alpha\in \Delta$, and let 
$\fg_i:=\left\{x\in \fg \mid [h,x]=ix\right\}.$

By the Jacobi-identity, this is a grading of $\fg$; in fact all $\ZZ$-gradings are obtained in this way. 
\end{ex}

Let  $G_0\subset G$ be the connected subgroup corresponding to $\fg_0\subset \fg$. Since each $\fg_i$ is stabilized by $\fg_0$, this group acts on each $\fg_i$. One of the interests of $\ZZ$-gradings is the following:

\begin{theorem}[\cite{Vinberg_graded}]
\label{thm_finite_parabolic}
The action of $G_0$ on $\fg_1$ has only finitely many orbits. 
\end{theorem}

In order to give a proof of this theorem, we recall some basic facts about nilpotent elements in semisimple Lie algebras.

\subsection*{The nilpotent cone}
\label{sec_nilp_orbits}
Let $\fg$ be a semisimple Lie algebra and $G=Aut(\fg)$ the corresponding adjoint Lie group. An element $x\in \fg$ is \emph{nilpotent} (respectively \emph{semisimple}) if the linear operator $\ad(x):=[x,\cdot\,]\in End(\fg)$ is nilpotent (resp. diagonalizable). Orbits of semisimple elements are completely described by Chevalley's restriction Theorem (Theorem \ref{thm_Chevalley}): any semisimple element is contained in a Cartan subalgebra $\fh\subset \fg$ (a maximal subalgebra consisting of commuting semisimple elements) and all Cartan subalgebras are conjugate under the action of $G$; moreover, two elements in $\fh$ are conjugate under $G$ if and only if they are conjugate under the Weyl group, which is finite. In particular, semisimple orbits form a continuous family of dimension equal to the rank of $G$, defined as the dimension of any Cartan subalgebra. Finally, semisimple orbits are those that are closed in $\fg$.

Nilpotent orbits have somehow ``opposite'' properties to semisimple ones. For instance, they are never closed, since the closure of each nilpotent orbit contains $0$, and in fact these orbits can be caracterized by this property. More precisely, if $x$ is nilpotent then $\lambda x$ is in the same orbit for any $\lambda\in \CC^*$: nilpotent orbits are pointed cones. This follows from the Jacobson-Morozov Theorem, which states that any nilpotent element $x\in\fg$ can be extended into a $\fsl_2$-triple $(x,y,h)$ such that $\spann{x,y,h}\subset \fg$ is a subalgebra isomorphic to $\fsl_2$ with its canonical basis (with $h$ the semisimple element). In particular $[h,x]=2x$. Then the one dimensional torus $\exp(\CC h)\simeq \CC^*\subset G$ (where $h$ is seen as a linear operator on $\fg$ via the adjoint action, and $\exp$ is the exponential of linear operators) acts on $\CC x$ with weight $2$, i.e., $\exp(\lambda h)\cdot x=2^\lambda x$ for any $\lambda\in \CC$. 

This implies for instance that the nilpotent cone $\cN$, the union of all nilpotent elements in $\fg$, can be described  as 
$\cN=\left\{x\in \fg\mid f(x)=0\mbox{ for any }f\in \CC[\fg]^G\right\}\subset\fg.$ Combining this with the Chevalley's restriction Theorem, one can get a complete picture of orbits in semisimple Lie algebras, generalizing Jordan's reduction of matrices. Recall that the Weyl group $W$ is a finite reflection group actin on $\fh$. 

\begin{theorem}[Chevalley's restriction theorem]
\label{thm_Chevalley}
Restriction to any Cartan subalgebra $\mathfrak{h}$ induces an isomorphism between $G$-invariant polynomials on $\mathfrak{g}$ and $W$-invariant polynomials on $\mathfrak{h}$:
$ \CC[\mathfrak{g}]^G\simeq \CC[\mathfrak{h}]^W.$
\end{theorem}

The orbit picture is as follows. The GIT quotient $\fg/\hspace*{-1mm}/G=\Spec \CC[\fg]^G$ coincides with the finite quotient $\fh/W=\Spec \CC[\fh]^W$. Since $W$ is a finite reflection group, $\CC[\fh]^W$ is a polynomial algebra. The GIT quotient parametrizes  semisimple (or equvalently,  closed) orbits. On the other extreme, nilpotent orbits are all contracted to $0\in \Spec \CC[\fh]^W$ by the quotient morphism. Any $x\in\fg$ admits a unique Jordan decomposition $x=x_s+x_n\in\fg$ with $x_s$ semisimple, $x_n$ nilpotent and $[x_s,x_n]=0$;  then the image of $x$ in 
$\fg/\hspace*{-1mm}/G$ is the same as the image of $x_s$. A very important fact
which distinguishes nilpotent and semisimple orbits is the following:

\begin{theorem}
There are finitely many  orbits in the nilpotent cone of a semisimple Lie algebra.
\end{theorem}

The nilpotent orbits can be classified using the theory of $\fsl_2$-triples. 
\begin{ex}
Let us consider $\fg=\fsl_{n+1}$. Nilpotent elements can be put in Jordan normal form with only zeroes on the diagonal. The sizes of the blocks  determine uniquely the nilpotent orbit the element belongs to, hence a bijection between nilpotent orbits
and partitions of $n+1$.
\end{ex}

\begin{proof}[Proof of Theorem \ref{thm_finite_parabolic}]
We want to show that there are  finitely many $G_0$-orbits of elements $x\in \fg_1$. First we notice that any element $x\in \fg_1$ is nilpotent ($ad^n (x)=0$ for $n>>0$ since it sends $\fg_k$ to $\fg_{k+n}$). Then we show that, for each nilpotent $G$-orbit $\cO\subset \fg$, each irreducible component of $\cO\cap \fg_1$ is a $G_0$-orbit; since irreducible components are finitely many, the result then follows from the finiteness of nilpotent orbits.

To prove the claim it suffices to show that for any $x\in \cO\cap \fg_1$, the tangent space  at $x$ of its $G_0$-orbit is equal to the tangent space  at $x$
of $\cO\cap \fg_1$. The former is  $[\fg_0,x]$. The latter is $[\fg,x]\cap \fg_1$, and clearly contains the former. Now, consider an element $[y,x]\in \fg_1$ with $y\in \fg$. If the decomposition of $y$ with respect to the grading is  $y=\sum_k y_k$, then $[y_k,x]\in \fg_{k+1}$ for any $k$, hence $[y,x]=[y_0,x]\in [\fg_0,x]$.
\end{proof}

\begin{remark}
The fact that the number of nilpotent orbits in $\fg_1$ is finite also holds for any $\ZZ_m$-grading (see next section), with exactly the same proof.
\end{remark}

We conclude that $\fg_1$ is a $G_0$ prehomogeneous space, meaning that it contains a dense orbit. However, these are very special prehomogeneous spaces (called \emph{parabolic}), since a prehomogeneous space does not need to have finitely many orbits in general. We also point out that Vinberg managed to classify orbits in parabolic prehomogeneous spaces in terms of $\fsl_2$-triples \cite{Vinberg_Classification_of_homogeneous_nilpotent_elements_of_a_semisimple_graded_Lie_algebra}.

Let us come back to \ref{main_ex_grading_Z} and let us now suppose that $h$ is of a very special form: we ask that there exists $1\leq j\leq n$ such that $h(\alpha_j)=1$ and $h(\alpha_u)=0$ for $u\neq j$. Recall that any root $\alpha$ can be written as an integer linear combination of simple roots with non-negative or non-positive coefficients $\alpha=\sum_i n^\alpha_i \alpha_i$. One easily deduces that $\fg_0=\fh \oplus \bigoplus_{\alpha,\, n^\alpha_j=0}\fg_\alpha $. Then $\fg_0= \fg_0' \oplus \CC h_j$, where $\fg_0'$ is the semisimple Lie algebra whose Dynkin diagram is obtained by erasing the $j$-th node from the Dynkin diagram of $\fg$, and $h_j\in \fh$. Also $\fg_1$ can be recovered in a similar manner: it is the $\fg_0$-representation whose highest weight is $-\alpha_j|_{\fg_0}$. Let us give a few examples of such gradings that already appeared in these lectures.

\begin{ex}\label{example:e7e8}
Consider the exceptional group $E_6$ with its Lie algebra $\fe_6$, and the previous grading defined by  $j=2$ (in Bourbaki's notation, as in the diagram below). By removing the second root, we obtain $\fg_0=\fsl_6\oplus \CC h_2$, which corresponds to $G_0\simeq \GL_6$ acting on $\fg_1=\extp^3 \CC^6$. As we have already seen, this action has only five orbits: $\{0\}$, the cone over $\Gr(3,6)$, the set of partially decomposable tensors, the tangential variety of $\Gr(3,6)$ and the open orbit.
\begin{equation*}
 E_{6}=\dynkin[labels={\alpha_1,\alpha_2,\alpha_3,\alpha_4,\alpha_5,\alpha_6}]E6 \qquad 
E_{7}=\dynkin[labels={\alpha_1,\alpha_2,\alpha_3,\alpha_4,\alpha_5,\alpha_6,\alpha_7}]E7 
\qquad E_{8}=\dynkin[labels={\alpha_1,\alpha_2,\alpha_3,\alpha_4,\alpha_5,\alpha_6,\alpha_7,\alpha_8}]E8 
\end{equation*} 
\end{ex}

\begin{ex}
In fact, we can run the previous example also with $E_7$ and $E_8$, always choosing the second root. The result in both cases is that $\extp^3 \CC^n$ has finitely many orbits under the action of $\GL_n$ when $n=7$ (choose $E_7$) and $n=8$ (choose $E_8$). When $n=9$, the orbits of $\extp^3 \CC^9$ can still be classified by using $\ZZ_m$-gradings, but there are continuous families  (more precisely, a four dimensional family) \cite{VEl}.
\end{ex}

\begin{ex}
\label{ex_spinor_orbit_grading}
Again, consider the Lie algebras $\fe_6,\fe_7,\fe_8$, but this time choose $j=1$. Then $\fg_0=\fso_n\oplus \CC h_1$, with $n=10,12,14$ respectively, and $\fg_1=\Delta_n$ is a half-spin representation. We deduce that $\Delta_n$ has finitely many orbits in these cases. Once more, the next case $n=16$ has a continuous family of orbits and appears in a $\ZZ_m$-grading of $\fe_8$
\end{ex}

\subsection*{\texorpdfstring{$\ZZ_m$}{Zm}-Gradings}

We will be very brief on $\ZZ_m$-gradings and will focus on examples. The definition is similar to the one of $\ZZ$-gradings, but this time replace $\ZZ$ by $\ZZ_m:=\ZZ/m$. The general construction of $\ZZ_m$-gradings is formally similar to $\ZZ$-gradings. Fix a semisimple Lie algebra $\fg$; in the case of $\ZZ_m$-gradings, one should consider an affine Dynkin diagram of $\fg$.

\begin{remark}
If $\fg$ has no outer automorphisms, there is only one affine Dynkin diagram $G^{(1)}$,  obtained by adding to the Dynkin diagram of $\fg$ a node corresponding 
to the lowest root.  Lie algebras with outer automorphisms have additional affine Dynkin diagrams.
\end{remark}

Now, the choice of a function $f\colon \phi \to \ZZ$, where $\phi$ is the set of nodes of the diagram, defines a $\ZZ_m$-grading for a certain $m$, and all $\ZZ_m$-gradings are obtained this way (the construction for $\ZZ$-gradings is the same but starts from actual Dynkin diagrams). If $\phi$ vanishes on all but one node, say the $j$-th node, then $\fg_0$ is the semisimple Lie algebra whose Dynkin diagram is obtained by erasing the $j$-th node from the affine Dynkin diagram, and $\fg_1$ is the representation with highest weight $-\alpha_j|_{\fg_0}$.

\begin{ex}
The affine Dynkin diagram of $\fg_2$ is 
$ G^{(1)}_{2}=\dynkin[extended, labels={1,2,3}]G2 $.
Choosing a node gives a $\ZZ_m$-grading, where $m$ is the label of the node. So for instance choosing the first node on the left gives a $\ZZ_1$-grading, i.e., no grading at all: $(\fg_2)_0=(\fg_2)_1=\fg_2$. Choosing the last node on the right gives a $\ZZ_3$ grading with $(\fg_2)_0=\fsl_3$, $(\fg_2)_1=\CC^3$ and $(\fg_2)_{-1}=(\CC^3)^*$ (can you recognize it?). Choosing the central node gives a $\ZZ_2$-grading with $(\fg_2)_0=\fsl_2 \oplus \fsl_2$ and $(\fg_2)_1=\CC^2 \otimes S^3 \CC^2$.
\end{ex}

\begin{ex}
\label{ex_graded_so8}
One of the affine Dynkin diagrams of $\fso_8$ is $ SO_8^{(1)}=\dynkin[extended, labels={1,1,2,1,1}]D4 $.
Choosing the central node gives a $\ZZ_2$-grading with $(\fso_8)_0=(\fsl_2)^{\oplus 4}$ and $(\fso_8)_1=(\CC^2)^{\otimes 4} $.
\end{ex}

\begin{ex}
The affine Dynkin diagram of $\fe_8$ is $ E^{(1)}_{8}=\dynkin[extended, labels={1,2,3,4,6,5,4,3,2}]E8 $.
Choosing the upper node gives a $\ZZ_3$-grading with $(\fe_8)_0=\fsl_9$, $(\fe_8)_1=\extp^3 \CC^9 $ and $(\fe_8)_{-1}=\extp^6 \CC^9 $. Choosing the first node on the left gives a $\ZZ_2$-grading with $(\fe_8)_0=\fso_{16}$ and $(\fe_8)_1=\Delta_{16} $. These two examples were mentioned in the previous subsection.
\end{ex}

A remarkable property of the $G_0$-action on $\fg_1$ is that a version of Jordan reduction holds. In particular, 
similarly to the classification of nilpotent and semisimple elements in semisimple Lie algebras (which can be seen as $\ZZ_1$-graded algebras), Vinberg showed that there are finitely many nilpotent $G_0$-orbits in $\fg_1$, while semisimple orbits are parametrized by a quotient of an affine space by a complex reflection group \cite{Vinberg_graded}. 

\section{Tits' shadows}
\label{sec_tits}

By a \emph{Fano variety of subvarieties}, we will mean an algebraic variety (or scheme) parametrizing certain classes of subvarieties inside a fixed ambient variety. The theory of Tits' shadows (\cite{Tits}, see also \cite{LM03}) allows to describe certain Fano varieties of subvarieties inside a $G$-homogeneous space $X$. It works particularly well to describe Fano varieties of lines (parametrizing $\PP^1$'s inside $X$). This theory starts from the observation that $G$ will act as well on such a  Fano variety of subvarieties. If the action is again transitive, the Fano variety of subvarieties itself will be another $G$-homogeneous space, and one can decide combinatorially which one directly from the Dynkin diagram.
Let us explain how.

Suppose $G$ is  semisimple  with Dynkin diagram $D$. Recall that by $G$-homogeneous space $X$ we mean a rational projective $G$-homogeneous variety; it can be seen as a quotient $X=G/P$, where $P$ is a \emph{parabolic} subgroup. Such $G$-homogeneous classified by combinatorially classifying  parabolic subgroups $P$; the latter   are in correspondence with sets of simple roots, i.e., sets of nodes of the Dynkin diagram $D$. For $I$ such a set we let $P_I$ be the corresponding parabolic subgroup. If $\omega=\sum_{i\in I}\omega_i$ is the associated sum of fundamental weights, then $G/P_I$ is the minimal orbit inside the projectivization $\PP V_\omega$ of the irreducible representation with highest weight  $\omega$.

\begin{ex}
For $G=\SL_{n+1}$, choose $I=\{\alpha_i\}$ the $i$-th simple root. Then $G/P_I=G/P_i$ is the Grassmannian $\Gr(i,n+1)\subset \PP(\wedge^i\CC^{n+1})$.
\end{ex}

For $X=G/P_I$, we choose another set $J$ of simple roots, disjoint from $I$, and let $D'$ be the Dynkin diagram obtained by removing from $D$ the nodes of $J$. Denote by $G'$ a semisimple group wit Dynkin diagram $D'$. The homogeneous space $G/P_{I\cup J}$ projects onto both  $G/P_I$ and $G/P_J$, and the fibers of the projection $G/P_{I\cup J}\to G/P_J$ are isomorphic to $G'/P'_I$. They map isomorphically to $G/P_I$ where there images are the \emph{Tits shadows}. Thus $G/P_J$ can be seen as parametrizing a family of subvarieties, isomorphic to $G'/P'_I$, inside $G/P_I$. In many cases this family is the whole Fano variety of subvarieties $G'/P'_I\subset G/P_I$. 


\noindent {\it Fact}: Assume that $I$ is made of a single root $\{\alpha_i\}$, which is not a long root (this includes the case where all roots have the same length). Then the Fano variety of linear spaces $\PP^k$ of dimension $k$ inside $G/P_I$ is the union of all those homogeneous varieties $G/P_J$ such that $G'/P'_I \simeq \PP^k$. Equivalently, the connected component of $D'$ that contains $I$ must be of type $A_{k-1}$, with $\alpha_i$ being an extreme node of $D'$.

\begin{ex}
\label{lines_in_grassmannian}
Let us describe Tits' shadows yielding linear projective spaces inside the Grassmannians $X=\SL_{n+1}/P_i=\Gr(i,n+1)$.




In the following pictures, we have marked with $\bullet$ the node corresponding to $I$, with $\times$ the nodes of $J$; the connected component of $D'$ containing $I$ (of type $A_{k-1}$) has been framed in a rectangle. For simplicity, we fixed $n=7$ and $i=3$.

The Fano variety of lines in $X$ is given by the flag variety $\Fl(i-1,i+1,n+1)=\SL_{n+1}/P_{i-1,i+1}$, as shown by  Tits' shadow described by the following picture.

\setlength{\unitlength}{4mm}
\thicklines
\begin{picture}(18,2)(-1,0)
\put(13,.5){\dynkin[edge length=8mm] A{ot*tooo}}
\multiput(16.8,.2)(1,0){2}{\line(0,1){1}}
\multiput(16.8,.2)(0,1){2}{\line(1,0){1}}
\end{picture}

Geometrically this corresponds to the fact that each point of $\Fl(i-1,i+1,n+1)$, that is,  each flag of subspaces $V_{i-1}\subset V_{i+1}\subset \CC^{n+1}$, 
defines a line in $\Gr(i,n+1)$. Or course this is just the line  parametrizing $[U_i]\in \Gr(i,n+1)$ such that $V_{i-1}\subset U_i\subset V_{i+1}$. All lines in $\Gr(i,n+1)$ are obtained in this way. 

For linear subspaces of dimension two there are in general two families, obtained as follows: any flag $V_{i-2}\subset V_{i+1}\subset \CC^{n+1}$ from  $\Fl(i-2,i+1,n+1)$ defines the $\PP^2$ in $\Gr(i,n+1)$ parametrizing $[U_i]\in \Gr(i,n+1)$ such that $ V_{i-2}\subset U_i\subset V_{i+1}$.  And similarly for a flag $V_{i-1}\subset V_{i+2}\subset \CC^{n+1}$. Thus (except if $i$ is too small or too big) $\Hil_{\PP^2}(X)=\Fl(i-2,i+1,n+1) \cup \Fl(i-1,i+2,n+1).$

These two families are described by the following Tits' shadows:

\setlength{\unitlength}{4mm}
\thicklines
\begin{picture}(18,2)(-5.5,1)

\put(0,1.5){\dynkin[edge length=8mm] A{to*tooo}}
\multiput(1.8,1.2)(3.1,0){2}{\line(0,1){1}}
\multiput(1.8,1.2)(0,1){2}{\line(1,0){3.1}}
\put(16,1.5){\dynkin[edge length=8mm] A{ot*otoo}}
\multiput(19.8,1.2)(3.1,0){2}{\line(0,1){1}}
\multiput(19.8,1.2)(0,1){2}{\line(1,0){3.1}}
\end{picture}
\end{ex}

\begin{ex}
In order to cconclude with an example of exceptional type, consider the diagrams
corresponding to the generalized Grasmannians of type $E_6$, 

$\begin{array}{ccc}
E_6/P_1 \;\;\dynkin[edge length=6mm] E{*ooooo} \qquad & 
E_6/P_2 \;\; \dynkin[edge length=6mm] E{o*oooo} \qquad & 
E_6/P_6 \;\; \dynkin[edge length=6mm] E{ooooo*} \\
& & \\ 
E_6/P_3 \;\; \dynkin[edge length=6mm] E{oo*ooo} \qquad & 
E_6/P_4 \;\; \dynkin[edge length=6mm] E{ooo*oo} \qquad & 
E_6/P_5 \;\; \dynkin[edge length=6mm] E{oooo*o} 
\end{array}$

\medskip
We know that $E_6/P_1$ (and symmetrically $E_6/P_6$) is the Cayley plane,
while $E_6/P_2$ is the adjoint variety $E_6^{ad}$. Playing with Tits shadows, we see that $E_6/P_3$ is the Fano variety of lines on $E_6/P_1$, and also parametrizes families of $\PP^4$'s on both $E_6/P_2$ and $E_6/P_6$. Similarly, $E_6/P_4$ is the Fano variety of lines on $E_6/P_2$, and also the Fano variety of planes on both 
$E_6/P_1$ and $E_6/P_6$! 

In particular, we can recover all the homogeneous spaces of $E_6$ just from 
the Cayley planes and the adjoint variety. This implies that the whole representation theory of $E_6$ can be recovered from the minimal representation of dimension $27$, 
and the adjoint representation; a fact which, obviously, was already known to E. Cartan. 
\end{ex}

\begin{remark}
\label{rmk_lines_short_roots}
When $\alpha_i$ is a short root (which may happen only if $G$ is not simply laced) then the Fano variety of $\PP^k$'s inside $G/P_I$ is not homogeneous. For instance, for $k=1$ there are two orbits in general, the closed one being of the form $G/P_J$ such that $G'/P'_I\simeq \PP^1$. We already know a typical example: the $G_2$-quadric $\QQ^5=G_2/P_1$, whose space of lines is $\OGr(2,7)$ and contains two orbits;
the closed one is the $G_2$ adjoint Grassmannian $\GtwoGr(2,7)=G_2/P_2$ (recall that $ G^{}_{2}=\dynkin[labels={\alpha_1,\alpha_2}]G2 $).
\end{remark}

\addcontentsline{toc}{chapter}{References}
\bibliographystyle{plain}
\bibliography{main}
\end{document}